\documentclass[reqno,11pt,a4paper]{amsart}

\usepackage[utf8]{inputenc}
\usepackage[T1]{fontenc}

\usepackage{amsthm,xcolor,amssymb,comment,enumerate, tikz, mathtools}
\usetikzlibrary{arrows.meta, positioning, shapes, fit}

\usepackage{hyperref}
\hypersetup{
    colorlinks=true,       
    linkcolor=blue,        
    citecolor=green,       
    urlcolor=red,          
}
\usepackage{fullpage}

\newcommand{\R}{\mathbb R}

\newcommand{\C}{\mathbb C}
\newcommand{\Z}{\mathbb Z}

\renewcommand{\H}{\mathbb H}

\newcommand{\mc}{\mathcal}
\newcommand{\mf}{\mathfrak}
\newcommand{\bb}{\mathbb}

\newcommand{\Hom}{\operatorname{Hom}}
\newcommand{\End}{\operatorname{End}}
\newcommand{\id}{\operatorname{id}}

\newcommand{\inj}{\operatorname{inj}}

\newcommand{\SL}{\operatorname{SL}}

\newcommand{\On}{\operatorname{O}}
\newcommand{\SO}{\operatorname{SO}}
\newcommand{\SU}{\operatorname{SU}}

\newcommand{\Conf}{\operatorname{Conf}}
\newcommand{\su}{\mathfrak{su}}

\newcommand{\Diff}{\operatorname{Diff}}
\newcommand{\vol}{\operatorname{vol}}
\newcommand{\tr}{\operatorname{tr}}
\newcommand{\im}{\operatorname{im}}

\newcommand{\hyp}{\operatorname{hyp}}

\newcommand{\sph}{\operatorname{sph}}
\newcommand{\Real}{\operatorname{Re}}
\newcommand{\Imag}{\operatorname{Im}}

\newcommand{\rk}{\operatorname{rk}}

\newcommand\dS{\operatorname{dS}}

\newcommand\dom{\operatorname{dom}}

\newcommand\fid{\operatorname{fid}}
\newcommand\model{\operatorname{mod}}
\newcommand\appr{\operatorname{app}}

\newcommand\err{\operatorname{err}}
\newcommand\lcl{\operatorname{loc}}

\newcommand\sech{\operatorname{sech}}
\newcommand\csch{\operatorname{csch}}

\renewcommand{\Re}{\operatorname{Re}}

\newcommand\wt\widetilde

\newcommand\halfopen[2]{\ensuremath{[#1,#2)}}

\newcommand{\bs}{\backslash}
\newcommand{\ol}{\overline}

\newtheoremstyle{break}
  {}
  {}
  {\itshape}
  {}
  {\bfseries}
  {.}
  {\newline}
  {}

  \theoremstyle{break}

\newtheorem{thm}{Theorem}[section]

\newtheorem{prop}[thm]{Proposition}
\newtheorem{cor}[thm]{Corollary}
\newtheorem{lemma}[thm]{Lemma}

\newtheorem{defn}[thm]{Definition}

\newtheorem{rem}[thm]{Remark}

\newtheorem{exa}[thm]{Example}

\usepackage{xcolor}
\usepackage{tcolorbox}

\title{Transgressive harmonic maps and $\mathrm{SU}(1,1)$ self-duality solutions}

\author{Sebastian Heller}
\author{Lothar Schiemanowski}
\author{Hartmut Weiss}

\address{S. Heller \\BIMSA\\China
 }
 \email{sheller@bimsa.cn}

\address{L. Schiemanowski \\Universit\"at Kiel\\Germany
 }
 \email{schiemanowski@math.uni-kiel.de}

\address{H. Weiss \\Universit\"at Kiel\\Germany
 }
 \email{weiss@math.uni-kiel.de}

\begin{document}

\maketitle


\begin{abstract}
We establish a duality between harmonic maps from Riemann surfaces to hyperbolic 3-space and harmonic maps from Riemann surfaces to de Sitter 3-space, best viewed as a generalized Gau\ss{} map. On the gauge-theoretic side, it matches $\SU(2)$ and $\SU(1,1)$ solutions of Hitchin's self-duality equations via a signature flip along an eigenline of the Higgs field. Reversing this operation typically produces singular solutions, which occur where the eigenline becomes lightlike. Motivated by explicit model examples and this singular behavior, we extend this duality to a class of \emph{transgressive} harmonic maps 
$f\colon M\to \mathbb S^3$, which are harmonic to the hemispheres equipped with the hyperbolic metric, intersect the equator orthogonally, and have vanishing Hopf differential along the intersection. We construct large families by gluing and analyze their regularity; as an application, we obtain 
real holomorphic sections of the Deligne--Hitchin moduli space of arbitrarily large energy that are not twistor lines.
\end{abstract}

\setcounter{tocdepth}{2}
\tableofcontents

\section{Introduction}
The study of harmonic maps into symmetric spaces and their associated gauge-theoretic equations has long been a fertile ground for interactions between differential geometry, integrable systems, and mathematical physics. A paradigmatic example is the correspondence between (equivariant) harmonic maps from Riemann surfaces into hyperbolic 3-space $\mathbb H^3 = \SL(2,\C)/\SU(2)$ and solutions to the $\SU(2)$ self-duality equations on a Hermitian vector bundle of rank $2$, that is, solutions of the equations
\[
F_A + [\Phi \wedge \Phi^*] = 0, \qquad \ol{\partial}_A \Phi = 0,
\]
where $A$ is a unitary connection inducing the trivial connection on the determinant bundle, and $\Phi \in \Omega^{1,0}(\End_0 E)$ is the Higgs field.  Less well known is that there is also a correspondence between harmonic maps from Riemann surfaces into de Sitter 3-space $\dS_3 = \SL(2,\C)/\SU(1,1)$ (a Lorentzian symmetric space of constant positive curvature) and solutions to the $\SU(1,1)$ self-duality equations, which take the same form as above, but now the bundle metric is a Hermitian metric of signature $(1,1)$ and the connection is unitary with respect to this indefinite metric.

In this paper, we establish a geometric duality that transforms harmonic maps into $\mathbb H^3$ into harmonic maps into $\dS_3$, and, under some natural extra assumptions, vice versa. Crucially, this transformation extends to the associated gauge fields: $\mathrm{SU}(2)$ self-duality solutions are mapped to $\mathrm{SU}(1,1)$ self-duality solutions, while the converse in general produces {\em singular} solutions of Hitchin's self-duality equations. These singular solutions have a geometric counterpart, {\em transgressive} harmonic maps, i.e., harmonic maps into the union of two hyperbolic spaces $\H^3_\pm$, which extend transversally and conformally through the boundary $\mathbb S^2_\infty$ at infinity. They also have a complex analytic counterpart, certain real holomorphic sections of the Deligne--Hitchin moduli space for the natural real twistor structure covering $\lambda\mapsto-1/\bar\lambda$. We first analyze these correspondences and dualities, and then develop a gluing construction that produces large families of examples.
This provides a new geometric realization of the relation between definite and indefinite real forms in the self-duality equations, and opens up a path to extend the non-abelian Hodge correspondence to a singular setting.

The duality between harmonic maps into $\H^3_\pm$ and harmonic maps into $\dS_3$ is best understood as a generalization of the Gau\ss\@ map construction for hypersurfaces. It is a consequence of the Ruh--Vilms theorem that the Gau\ss\@ map of a conformal, harmonic immersion into $\R^3$ is a harmonic map into $\mathbb S^2$. Considering $\H^3_\pm$ as the two sheets of the 2-sheeted hyperboloid in Minkowski space $\R^{1,3}$, we define for a conformal immersion $f : M \to \H^3_\pm$ the hyperbolic Gau\ss\@ map to be the unique oriented normal of the immersed surface in the tangent space of the hyperbolic space, seen as a subspace of $\R^{1,3}$. If $f$ is harmonic, then so is its Gau\ss\@ map, see for example \cite{BaBo, DIK}.

We introduce a generalization of the Gau\ss\@ map in the non-conformal case. This map is not normal to the surface and therefore is dubbed \textit{oblique hyperbolic Gau\ss\@ map}, see Definition \ref{defn:hyperbolic-oblique-gauss-map-1}. The tangential part of this map depends on the choice of a square root of the hyperbolic Hopf differential, which, in general, is
only well-defined on the Hitchin covering -- the natural double cover on which $\Phi$ diagonalizes.

If $f$ is harmonic, its hyperbolic oblique Gau\ss\@ map is also harmonic. This is proved in Proposition \ref{prop:hyperbolic_gauss_map_harmonic}.

Conversely, starting from an immersion $N : M \to \dS_3$, there is an analogous definition of a map $f : M \to \H^3$, Definition \ref{defn:dual_map}. This definition depends on the choice of a square root of the Hopf differential of $N$. Proposition \ref{prop:dual_map_harmonic} shows that harmonicity of $N$ implies that $f$ is harmonic. The construction is involutive in the sense that the dual map of the hyperbolic Gau\ss\@ map is the original map, provided that the correct choice of the square roots of the Hopf differentials is made (Proposition \ref{prop:dual_involution}).

Under the correspondence between harmonic maps into $\H^3$ and solutions of the $\SU(2)$ self-duality equations and harmonic maps into $\dS_3$ and solutions of the $\SU(1,1)$ self-duality equations, the above construction admits a purely gauge-theoretic construction. The choice of a square root of the Hopf differential is then equivalent to the choice of an eigenline of the Higgs field. The analog of the hyperbolic Gau\ss\@ map construction is described in Theorem \ref{thm:SU2SU11} and the analog of the dual map construction is described in Theorem \ref{thm:SU11SU2}. On the level of the Hermitian metric, it corresponds to flipping the sign in the orthogonal complement of a chosen eigensubbundle of the Higgs field -- therefore turning a definite Hermitian metric into an indefinite Hermitian metric and vice versa. This construction is more general in the sense that the maps corresponding to the solutions of the self-duality equations need not be immersions. Moreover, an interesting phenomenon can be observed, which is not directly evident in the geometric setting: 
 the dual solution associated with an $\SU(2)$-solution is globally defined, whereas the dual solution associated with an $\SU(1,1)$-solution is only defined where the eigenline is not lightlike.
It turns out that this behavior can be understood on the level of the harmonic maps -- at least under suitable extra assumptions.

This becomes evident in the behavior of the \textit{model solutions}, explicit solutions of the $\SU(2)$ self-duality equations, for which all these correspondences and dualities can be explicitly computed. These solutions form a very concrete illustration of all of the above ideas, and we will briefly review their behavior. They are also one of the three building blocks in our gluing construction. These solutions are defined on the sliced complex plane $\{z = x+iy \in \C : x \neq 0\}$ and on quotients by suitable 1-dimensional lattices thereof. The underlying Higgs bundle is given by $\left( \underline{\C}^2, \bar{\partial}, \frac 1 2 t \begin{psmallmatrix} 0 & 1 \\ 1 & 0 \end{psmallmatrix} \right)$, where $\bar{\partial}$ is the standard holomorphic structure and $t > 0$. 
Then
\[
h_t^{\model} = \begin{pmatrix} \frac t {\tanh(t x)} & 0 \\ 0 & \frac{\tanh(tx)} t \end{pmatrix}
\]
solves the Hitchin self-duality equations $F^{h_t} + [\Phi \wedge \Phi^{*_{h_t}} ] = 0$. This metric develops a singularity at $x=0$, is positive definite on the component $x>0$, and negative definite on the component $x<0$. The solution of the $\SU(2)$ self-duality equation is given in unitary form in formula (\ref{eqn:model}). It turns out that the associated $\SU(1,1)$-solution is smooth through $x=0$; see Example \ref{exa:dual_SU11}. Moreover, the locus $x=0$ is exactly where the eigenline of the Higgs field becomes lightlike.

The associated harmonic map $f_{t,\hyp}$ into hyperbolic 3-space and its dual map $N_t$ are given in Example \ref{exa:modelharmonicmap}. It is instructive to consider the geometric behavior of the harmonic map $f_{t,\hyp}$. The image of the map $f_{t,\hyp}$, restricted to $x>0$, is a totally geodesic copy of $\H^2$ in $\H^3_+$, whereas for $x<0$ it is a totally geodesic copy of $\H^2$ in $\H^3_-$. The parametrization is such that the region close to $x = 0$ is mapped to the region close to the boundary at infinity of $\H^3_\pm$. The effect of the parameter $t$ is to dilate the domain.

On the other hand, the dual map $N_t$ is well-defined and smooth through $x=0$. This is expected, as the associated $\SU(1,1)$-solution is smooth. Geometrically, the locus $x=0$ can be characterized as the set where $N_t$ fails to be an immersive map.

The hemisphere model of hyperbolic 3-space is a convenient setting to ``unify'' the two copies $\H^3_\pm$: the upper hemisphere is identified with $\H^3_+$, the lower hemisphere is identified with $\H^3_-$, and the equatorial 2-sphere is identified with the boundary at infinity of both copies of hyperbolic space. Under this identification, the map $f_{t,\hyp}$ extends to a smooth map $f_{t,\sph}$ into $\bb{S}^3$ across $x=0$; see Example \ref{exa:modelharmonicmap_sphere}.

It turns out that there is a natural class of maps into $\bb{S}^3$, which reproduces this behavior. These maps will be called  \textit{transgressive} 
harmonic maps, and they are defined in Definition \ref{defn:transgressive}. Briefly, these are maps $f : M \to \bb{S}^3$ which are harmonic into hyperbolic 3-space when restricted to $f^{-1}(\H^3_\pm)$, that intersect the equatorial 2-sphere $\bb{S}^2_{\operatorname{eq}}$ orthogonally, and for which the associated Hopf differential vanishes along the intersection. The associated $\SU(2)$ self-duality solution is also defined only on $f^{-1}(\H^3_\pm)$ and develops singularities at $f^{-1}(\bb{S}^2_{\operatorname{eq}})$.

The crucial point is that the duality extends from harmonic maps into $\H^3_\pm$ to transgressive harmonic maps into $\bb{S}^3$. That is, given a transgressive harmonic map $f : M \to \bb{S}^3$, the oblique hyperbolic Gau\ss\@ map of $f|_{f^{-1}(\H^3_\pm)}$ extends to a smooth harmonic map $N : M \to \dS_3$ and its rank drops precisely at $f^{-1}(\bb{S}^2_{\operatorname{eq}})$. This is proved in Theorem \ref{thm:transobl}.

We prove a partial converse in Theorem \ref{thm:transgressive_map_from_de_sitter_map}: given a harmonic map into de Sitter 3-space, which is immersive away from a 1-dimensional submanifold and such that the rank drops in a controlled way, the dual map on the complement of this 1-dimensional submanifold smoothly extends to a transgressive harmonic map.

The moduli space of irreducible solutions to the $\SU(2)$ self-duality equations carries a natural hyperk\"ahler structure \cite{Hi}.  It admits a twistorial description, where the nonlinear equations of the gauge theory are encoded into complex-analytic data on a twistor space, subject to specific reality conditions. For the Hitchin moduli space, this twistor space (as first identified by Deligne \cite{Simpson}) is constructed by gluing the moduli space of $\lambda$-connections to the moduli space of $\lambda$-connections for the complex-conjugate Riemann surface (see Section \ref{sec:lambdacon}). Within this framework, the standard $\SU(2)$-solutions correspond to one component of the space of real holomorphic sections. The transgressive harmonic maps central to this paper are realized as real holomorphic sections lying in different components. Consequently, the geometric operation of constructing the dual map acquires a profound interpretation: it corresponds to selecting a different real lift from the moduli space into the space of 
$\lambda$-connections, see Section \ref{sec:dhlifts}.

Aside from the model solutions, we construct many new examples by means of a gluing procedure, inspired by the construction of large energy solutions of the $\SU(2)$ self-duality equations in \cite{MSWW}. The building blocks of these solutions are the model solutions, limiting configurations, and fiducial solutions. The limiting configurations are decoupled solutions of the self-duality equations on the complement of the zeros of the quadratic differential associated to the Higgs field. The model solutions decouple as $t\to\infty$. The fiducial solutions are a family of radially symmetric solutions of the self-duality equations, which also decouple as the parameter of this family goes to infinity. Consequently, we may construct an approximate solution of the self-duality equations by gluing the model solution and the fiducial solution at appropriate places. For details of this construction, see Section \ref{sec:approxsol}. These approximate solutions may be deformed to solutions, which in turn \textit{almost} give rise to transgressive maps.
\begin{thm}[Theorem \ref{thm:existence_hoelder_transgressive_maps}, Corollary \ref{cor:hoelder_dual_maps}]
  Let $M$ be a compact Riemann surface of genus at least $2$ and let $(E, \overline{\partial}_E, \varphi)$ be a Higgs bundle with $\deg E = 0$. Suppose that $q = \det \varphi$ has simple zeros and contains at least one Strebel cylinder (see Section \ref{sec:41}). Let $M^\vee$ denote the complement of the core loops in the Strebel cylinders and assume that it has two components.
  
  For all sufficiently large $t > 0$, there exist solutions $(A_t, t \Phi_t)$ of the $\SU(2)$ self-duality equations in the complex gauge orbit of $(E, \overline{\partial}_E, \varphi)$ with the following properties:
  \begin{enumerate}[(i)]
  \item near the central loops, the solution $(A_t, \Phi_t)$ is exponentially close to the model solution $(A_t^{\model}, \Phi_t^{\model})$,
  \item near the zeros of $q$, the solution $(A_t, \Phi_t)$ is exponentially close to the fiducial solution $(A_t^{\fid}, \Phi_t^{\fid})$,
  \item in the interior of $M^\vee - \{q=0\}$, the solution $(A_t, \Phi_t)$ is exponentially close to a limiting configuration $(A_\infty, \Phi_\infty)$.
  \end{enumerate}
  These solutions induce harmonic, equivariant maps $f_{t,\hyp}$ into $\H^3_\pm$, defined on the lift of $M^\vee$ to the universal cover of $M$. The maps  $f_{t,\hyp}$ extend to H\"older continuous equivariant maps $f_t : \wt{M} \to \bb{S}^3$. The hyperbolic oblique Gau\ss\@ maps of $f_{t,\hyp}$ extend to equivariant, H\"older continuous maps $N_t : \wt{M} \to \dS_3$, which are harmonic on the lift of $M^\vee$ to the universal cover.
\end{thm}
Analytically, our gluing construction combines techniques from gauge theory and geometric analysis -- in particular, the 0-calculus  and semiclassical estimates -- to control singular solutions of the self-duality equations.
The proof follows the standard template of a gluing construction, i.e., after defining the approximate solution the main work lies in perturbing this approximate solution to an actual solution using the Banach fixed point theorem. The central difficulty here is to establish uniform (in $t$) control of the linearization of the $\SU(2)$ self-duality equations at the approximate solutions. There are some major technical differences to the gluing constructions for globally smooth solutions of the self-duality equations carried out in \cite{MSWW} and \cite{FMSW}. Analytically, the singularity of the model solution implies that the linearization is most naturally interpreted as a uniformly degenerate or $0$-operator. Accordingly, we use the methods of the $0$-calculus to understand the mapping properties of the linearization. In particular, the $0$-calculus gives Fredholm properties in a certain indicial range, and this indicial range determines the boundary regularity of our solution, that is, it is responsible for the H\"older regularity of the maps $f_t$ and $N_t$. See Theorems \ref{thm:banach_space_iso} and \ref{thm:best_gluing_regularity} for details. The uniform control is much more involved. The basic idea, again standard, is to prove a certain weighted $C^0$ estimate, which is uniform in $t$, by a proof by contradiction (Theorem \ref{thm:C0_estimate}). This involves studying sequences which might violate a uniform estimate. These sequences are blown up around the points where the estimate is violated most egregiously. Each such blow-up sequence yields a non-trivial solution of an elliptic partial differential equation on some model space. The proof then breaks down into the application of a number of vanishing theorems, which show that each such solution must in fact be trivial. A substantial difficulty in extracting such solutions is that the linearization has a divergent part as $t\to \infty$. To be more precise, it has the form $L_t = \Delta_{A_t^{\appr}} - i t^2 \ast M_{\Phi_t^{\appr}}$. Here $M_{\Phi_t^{\appr}}$ acts on a rank 3 vector bundle and as $t\to \infty$ it approaches $16 \cdot \id_{K^\perp} \oplus 0 \cdot \id_K$, where $K$ is a rank 1 subbundle. This requires us to treat the kernel of $M_{\Phi_t^{\appr}}$ and its complement differently. It turns out that on $K^\perp$ techniques of semiclassical analysis can be brought to bear, in the spirit of \cite{BGIM}. In \cite{MSWW} this was avoided by considering an auxiliary operator $L_t^0 = \Delta_{A_t^{\appr}} - i \ast M_{\Phi_t^{\appr}}$. This approach was not viable in our setting, due to the $0$-singularity of our operator at the core loop. The cost of considering the full operator $L_t$ is substantially more involved vanishing theorems. We note that the proofs of the vanishing theorems associated to the zeros of the quadratic differential involve a fairly detailed study of the Bessel-type equations appearing in the linearization of the fiducial solution. These results (in particular  Propositions \ref{prop:no_fiducial_kernel}, \ref{prop:no_limiting_fiducial_kernel}) may be of independent interest to researchers using the fiducial solutions.

It should also be noted that the definition of a transgressive map requires the map to be at least differentiable through the equatorial 2-sphere. The H\"older regularity of the solutions in the theorem above is therefore insufficient to consider the maps $f_t$ as transgressive maps. The regularity of the solutions in the above theorem is below the threshold for elliptic regularity theory, as we discuss in more detail below. It is an interesting question what conditions are sufficient to guarantee smoothness of the solutions. We have not been able to answer this question in full detail, but the theorem below shows that under a certain symmetry assumption
-- mimicking the reflection symmetry of the model solution --
the solutions in the above theorem are in fact smooth.

\begin{thm}[Theorem \ref{thm:existence_smooth_transgressive_maps}]
  Suppose that $M$ is a compact Riemann surface of genus at least $2$ and let $(E,\overline{\partial}_E, \varphi)$ be a Higgs bundle with $\deg E = 0$. Suppose that $q = \det \varphi$ is simple and contains at least one Strebel cylinder. Suppose moreover that there is an antiholomorphic involution $\sigma : M \to M$ and an antilinear automorphism $\hat{\sigma} : E \to E$, such that $(E, \overline{\partial}_E, \varphi)$ is invariant under $\hat{\sigma}$. Suppose that the fixed point set of $\sigma$ consists of core loops of Strebel cylinders of $q$.

  In this case, the maps $f_t$ constructed in the previous theorem are smooth through the core loops, and therefore give rise to equivariant transgressive harmonic maps. Likewise, the maps $N_t$ are smooth through the core loops. 
\end{thm}
Remark \ref{rem:many-examples} shows that Riemann surfaces and Higgs bundles satisfying the conditions of the theorem can be found in any genus.
The proof of this theorem is based on the observation that under the symmetry conditions, $N_t$ is even. It turns out that this property, together with the regularity properties already shown in the previous theorem, suffices to show that $N_t$ is a weak solution of the harmonic map equation. A standard bootstrapping argument using elliptic regularity then shows that $N_t$ is in fact smooth. Theorem \ref{thm:transgressive_map_from_de_sitter_map} allows us to recover $f_t$ from $N_t$, and therefore $f_t$ is also smooth.
Finally, we reinterpret the previous existence theorem in the framework of Deligne--Hitchin moduli spaces: we obtain real holomorphic 
sections thereof of arbitrarily large energy. Here, the energy is a well-defined functional on the space of real holomorphic sections \cite{BeHR},
similar to the renormalized area for minimal surfaces in hyperbolic 3-space as in \cite{AlMa}.
\begin{thm}
For every $g>1$, there exists a Riemann surface $M$ of genus $g$ such that its $\mathrm{SL}(2,\C)$ Deligne--Hitchin moduli space
admits  
 $\tau$-real negative sections $s$ of arbitrarily large energy, which are not twistor lines.
\end{thm}
Some connections to other fields have already been identified. We now turn to a more detailed exploration of these links. Transgressive harmonic maps can be seen as a pair of harmonic maps into hyperbolic space, whose boundary data at infinity matches. The domain metric near the boundary is that of a hyperbolic funnel. Harmonic maps between hyperbolic spaces with specified boundary data have been studied in detail and much is known. From an analytical point of view, we would like to single out the articles \cite{LiTam0}, \cite{LiTam}. In these articles existence and uniqueness of harmonic maps between $\H^m$ and $\H^n$ with specified boundary data is investigated. Uniqueness is shown under the assumption that the map extends $C^1$ to the boundary and that the boundary map has nowhere vanishing energy density. Existence is shown under $C^{1,\alpha}$ regularity of the boundary data. They also construct a family of harmonic diffeomorphisms of $\H^2$, which are $C^{1/2}$ up to the boundary, and whose boundary maps are the identity. While there exist results for other regularity notions for boundary data, such as quasisymmetric, we are not aware of any work weakening the regularity requirements to $C^\alpha$, $\alpha > \frac 1 2$, for either the existence or uniqueness results of Li and Tam. We note that for the solutions constructed in Theorem \ref{thm:existence_hoelder_transgressive_maps} our methods yield H\"older regularity with a H\"older exponent arbitrarily close to $1$, but not $C^1$. It is a very interesting question under which conditions these maps extend to smooth maps.

The image of a transgressive harmonic conformal immersion consists of minimal surfaces in the hyperbolic hemispheres, intersecting the equatorial $\mathbb S^2$ orthogonally. This can again be seen as a boundary data problem at infinity, and in the case of a single copy of hyperbolic space this is a broad field of study known as the asymptotic Plateau problem, initiated in \cite{Anderson}. Again, much more is known, but these results are not directly relevant to our problem. In this conformal case, transgressive harmonic maps have been constructed before, without this terminology, \cite{BaBo,BHS,HH}. The present work unifies these geometric and analytic viewpoints, establishing a framework that may extend to higher-rank groups and other real forms.

As a word of warning, there is a notion of $U(p,q)$ Higgs bundles \cite{BGPG}, which is \textit{not} directly related to our notion of $\SU(1,1)$-self-duality equations.

\subsubsection*{Acknowledgements}
The authors would like to thank Jan Swoboda for many fruitful discussions during the initial stage of the project. HW would like to thank Oscar García-Prada for helpful conversations about Hitchin's equation for Hermitian metrics of indefinite signature. All authors were supported by the Deutsche Forschungsgemeinschaft within the priority program {\em Geometry at Infinity}. SH was supported by the Beijing Natural Science Foundation  IS23003.


\section{Preliminaries}
\subsection{Hyperbolic geometry and the conformal geometry of $\mathbb S^3$}\label{sect:geom_prelims}
The conformal geometry of $\mathbb S^3$ (and more generally of $\mathbb S^n$) is particularly well understood by means of the projectivization of the light cone in Minkowski space $\R^{1,4}$ ($\R^{1,n+1}$ respectively). This construction is classical, going back at least to Sophus Lie.

Let $\R^{1,4} = \R^5$ be equipped with the Minkowski inner product
\[
\langle \cdot \, , \cdot \rangle = -dx_0^2 + dx_1^2 + \ldots + dx_4^2.
\]
The {\em light cone} of $\R^{1,4}$ is defined to be
\[
\mc{L} = \left\{x \in \R^{1,4} : \langle x, x\rangle = 0 \right\}.
\]
As will be explained below, its projectivization $\bb{P}\mc{L}$ turns out to be diffeomorphic to $\mathbb S^3$. 
Denote by $\pi : \mc{L} \to \bb{P}\mc{L}$ the canonical projection. The projectivization -- and therefore $\mathbb S^3$ -- carries a natural conformal structure. To see this, observe that any local section $\sigma :  U \subset \bb{P}\mc{L} \to \mc{L}$ induces a Riemannian metric $\sigma^* \langle \cdot \,, \cdot \rangle$ on $U$. A simple calculation shows that two different lifts induce conformally equivalent metrics on $U$. In this way, a conformal structure is induced on $\bb{P}\mc{L}$.

The conformal group of $\mathbb S^3$ is then given by $\Conf(\mathbb S^3) = \On(1,4)/ \{\pm \id\}$, whereas the group of orientation-preserving conformal transformations is given by $\Conf^+(\mathbb S^3) = \left(\On(1,4) / \{\pm \id\}\right)_0 \cong \SO^+(1,4)$, where $\SO^+(1,n+1) = \left\{A \in \On(1,n+1) : \det A = 1, \; A_{00} > 0\right\}$.

Throughout this text, unless otherwise noted, by $\mathbb S^3$ we mean the slice of the light cone by the spacelike affine hyperplane $\{x_0 = 1\}$, i.e.\@
\[
\mathbb S^3 = \left\{ x \in \mc{L} : x_0 = 1 \right\}.
\]
The metric induced on $\mathbb S^3$ by the Minkowski metric is the standard round metric, and the restriction of the projection $\mathbb S^3 \to \bb{P}\mc{L}$ is a conformal diffeomorphism. Its inverse is a (global) section, denoted by
\[
\sigma_{\mathbb S^3} : \bb{P}\mc{L} \to \mathbb S^3, \qquad \left[x_0 : x_1 : \ldots : x_4 \right] \mapsto \left(1, \frac{x_1}{x_0}, \ldots, \frac{x_4}{x_0} \right).
\]
On the other hand, the slice of the light cone by the affine hyperplane $\{x_4 = 1\}$ is a Lorentzian two-sheeted hyperboloid, and the metric induced on the two sheets is the standard hyperbolic metric. For this reason, we denote
\[
\H^3_\pm = \H^3_+ \cup \H^3_- = \left\{ x \in \mc{L} : x_4 = 1 \right\},
\]
where $\H^3_+ = \left\{ x \in \mc{L} : x_4 = 1, x_0 > 0\right\}$ and $\H^3_- = \left\{ x \in \mc{L} : x_4 = 1, x_0 < 0 \right\}$. The projection $\H^3_\pm \to \bb{P}\mc{L}$ is a conformal diffeomorphism onto its image, and the image can be characterized as $\left\{ [x_0 : \ldots : x_4] : x_4 \neq 0 \right\}$.

The inverse defines a local section of $\mc{L}$ given by
\[
\sigma_{\H^3_\pm} : \left\{ [x_0 : \ldots : x_4 ] : x_4 \neq 0\right\} \to \H^3_\pm, \qquad \left[x_0 : \ldots : x_4 \right] \mapsto \left( \frac{x_0}{x_4}, \ldots, \frac {x_3}{x_4}, 1 \right).
\]
Since $\H^3_\pm$ carries the hyperbolic metric, its image in $\bb{P}\mc{L}$ inherits a hyperbolic metric. In this way, we can decompose $\bb{P}\mc{L}$ into two copies of hyperbolic space and a complement, which is given by $\left\{ [x_0 : \ldots : x_4] \in \bb{P}\mc{L} : x_4 = 0 \right\}$. Transporting this via $\sigma_{\mathbb S^3}$ to $\mathbb S^3$, we find that $\H^3_-$ corresponds to the ``lower'' hemisphere $\{x \in \mathbb S^3 : x_4 < 0 \}$, while $\H^3_+$ corresponds to the ``upper'' hemisphere $\{ x \in \mathbb S^3 : x_4 > 0 \}$.
The complement corresponds to the equatorial 2-sphere
\[
\mathbb S^2_{\operatorname{eq}} = \left\{ x \in \mathbb S^3 : x_4 = 0 \right\}.
\]
Hence, $\mathbb S^2_{\operatorname{eq}}$ can be considered to be the joint boundary at infinity of the two copies of hyperbolic 3-space.

In the sequel, the subgroup of the conformal group that fixes the equatorial 2-sphere will play an important role. Explicitly, it is the subgroup $\SO^+(1,3) \subset \SO^+(1,4)$.

It will also be useful to have an explicit map connecting the hyperbolic and spherical models. Such a map is given by
\[
\Xi = \sigma_{\H^3_\pm} \circ \pi : \mathbb S^3 \bs \mathbb S^2_{\operatorname{eq}} \to \H^3_\pm.
\]
We then have
\[
\Xi(x_0, \ldots, x_4) = \left( \frac {x_0}{x_4}, \ldots, \frac{x_3}{x_4}, 1 \right)
\]
and
\[
\Xi^{-1} (x_0, \ldots, x_4) = \left( 1, \frac{x_1}{x_0}, \ldots, \frac{x_3}{x_0}, \frac 1 {x_0} \right).
\]

To connect this geometric setup to the gauge-theoretic interpretation of harmonic maps, it is convenient to take a slightly different perspective on these spaces. The symmetric space $\SL(2,\C)/\SU(2)$ is isometric to hyperbolic 3-space. The Cartan involution is given by $A \mapsto (A^*)^{-1}$. The Cartan embedding $\SL(2,\C)/\SU(2) \hookrightarrow \SL(2,C)$ can therefore be given by $[A] \mapsto A A^*$. The image of this embedding is $\left\{ A \in \SL(2,\C) : A = A^*, A > 0 \right\}$ and provides the {\em matrix model} of hyperbolic 3-space. This model is very closely related to the light cone model. To see this, let us first observe that there is a natural isomorphism between the space of $2 \times 2$ Hermitian matrices
\[
\mc{H} = \left\{A \in \mf{gl}(2, \C) : A^* = A \right\}
\]
equipped with the quadratic form $-\det$, and $\R^{1,3}$ with the Minkowski metric. This identification is given by the isometry
\[
\R^{1,3} \to \mc{H}, \qquad \left(x_0, x_1, x_2, x_3\right) \mapsto
\begin{pmatrix}
  x_0 + x_1 & x_2 + i x_3 \\
  x_2 - i x_3 & x_0 - x_1
\end{pmatrix}.
\]
The subspace $\mc{H}_0 = \{ A \in \mc{H} : \tr A = 0\}$ of $\mc{H}$ corresponds to the standard Euclidean subspace $\R^3 \subset \R^{1,3}$.

We identify $\R^{1,4}$ with the direct sum $\mc{H} \oplus \R$ endowed with the quadratic form
\[
Q(A,r) = -\det(A) + r^2.
\]
The light cone is then given by
\[
\mc{L} = \left\{ (A,r) \in \mc{H} \oplus \R : \det(A) = r^2 \right\}.
\]

Note that under the above isometry, $\tr A$ corresponds to $2 x_0$. Therefore, $\mathbb S^3$ corresponds to
\[
\mathbb S^3 = \left\{ (A,r) \in \mc{L} : \tr A = 2 \right\}.
\]
The projection $\mc{L} \bs \{0\} \to \bb{P} \mc{L} \to \mathbb S^3$ is then given by $(A,r) \mapsto \frac 2 {\tr A} (A,r)$.

On the other hand, the hyperboloids can then be identified with
\[
\H^3_\pm = \{ (A,r) \in \mc{L} : r = 1 \} = \{ (A,1) \in \mc{H} \oplus \R : \det(A) = 1 \}.
\]
The map $\left\{ (A,r) \in \mc{L} : r \neq 0 \right\} \to \H^3_\pm$ is given by $(A,r) \mapsto \left( \frac 1 r A , 1 \right)$.

The equatorial 2-sphere is given by
\[
\mathbb S^2_{\operatorname{eq}} = \left\{ (A,r) \in \mathbb S^3 : r = 0\right\}.
\]
The subgroup $\SO^+(1,3) \subset \SO^+(1,4)$ fixing the equatorial 2-sphere is realized by $\SO^+(\mc{H}) \subset \SO^+(\mc{H} \oplus \R)$, and the action
\[
\SL(2,\C) \mapsto \SO^+(\mc{H}), \qquad A \mapsto \left( X \mapsto A X A^* \right)
\]
realizes the double cover $\SL(2,\C) \to \SO^+(\mc{H}) = \SO^+(1,3)$, providing a natural action of $\SL(2,\C)$ on the conformal 3-sphere.

For certain calculations, it will be useful to identify the 3-sphere with $\SU(2)$. Such an identification is furnished by
\[
\Upsilon : \mathbb S^3 \to \SU(2), \qquad \Upsilon(A,r) = r \id + i \mathring{A},
\]
where $\mathring{A}$ denotes the trace-free part of $A$, i.e.\@ $\mathring{A} = A - \frac 1 2 \tr A \id$.

Under this map, the equatorial 2-sphere $\mathbb S^2_{\operatorname{eq}}$ is mapped to $\{ A \in \SU(2) : \tr A = 0\}$.

With respect to these identifications, the map $\Xi$, defined on $\{B \in \SU(2) : \tr B \neq 0 \}$, becomes
\[
\Xi_{\SU(2)} \left( B \right) = \left( \frac 2 {\tr B} \left( \id - \mathring{B} \right), 1 \right)
\]
and
\[
\Xi_{\SU(2)}^{-1}(A,1) = \frac 2 {\tr A} \left( \id + i \mathring{A} \right).
\]
Note that very often we will consider $\H^3_\pm$ as a subset of $\mc{H}$ by forgetting the last coordinate.

\subsection{Gauge-theoretic background}
\label{sec:gauge_theory}

Let $M$ be a compact Riemann surface, and $E$ a complex vector bundle over $M$ of degree zero. We fix a trivialization of the determinant line $\det E = \Lambda^2E$. For concreteness, we can take $E = \underline{\C}^2$, the trivial complex vector bundle of rank $2$. 
We recall two equivalent formulations of Hitchin's self-duality equations, one in terms of connections and one in terms of Hermitian metrics.

1. Fix a Hermitian metric $h_0$ on $E$, inducing the trivial flat metric on $\det E$, and consider pairs $(\nabla,\Phi)$ where $\nabla$ is a unitary connection inducing the trivial flat connection on $\det E$, and $\Phi \in \Omega^{1,0}(M,\End_0(E))$ is a Higgs field. Let $\bar\partial^\nabla$ denote the $(0,1)$-part of $\nabla$. In a local unitary gauge, $\nabla = d + A$ for $A \in \Omega^1(U,\su(2))$ and $\bar\partial^\nabla =\bar\partial + A^{0,1}$. More globally, if $A$ is a connection $1$-form on the principal $\SU(2)$-bundle underlying $E$, we will write $\nabla = d_A$ and refer to $A$ as an $\SU(2)$-connection on $E$, as is customary in gauge theory. With this in place, the $\SU(2)$ self-duality equations read
\begin{equation}\label{eq:sd}
\bar\partial^\nabla \Phi = 0 \quad \text{and} \quad F^\nabla + [\Phi,\Phi^*]= 0 ,
\end{equation}
where the adjoint is taken with respect to $h_0$. The pair $(\bar\partial^\nabla, \Phi)$ defines an $\SL(2,\C)$-Higgs bundle on $E$. The equation is invariant under the unitary gauge group $\mathcal{G}$, i.e.\@ the group of unitary bundle automorphisms fixing $\det E$. The unitary gauge group acts by
\[
\nabla \mapsto \nabla \cdot g =g^{-1} \circ \nabla \circ g\quad \text{and} \quad \Phi \cdot g = \Phi \mapsto g^{-1}\Phi g
\]
on configurations $(\nabla, \Phi)$. The complex gauge group $\mc{G}^\C$, i.e.\@ the group of all complex-linear bundle automorphisms fixing $\det E$, acts by
\[
\nabla \mapsto \nabla \ast g = g^{-1} \circ \bar\partial^\nabla \circ g +  g^* \circ \partial^\nabla \circ (g^*)^{-1} \quad \text{and} \quad \Phi \ast g =\Phi \mapsto g^{-1}\Phi g,
\]
where $\partial^\nabla$ is the $(1,0)$-part of $\nabla$ (such that $\nabla = \bar\partial^\nabla + \partial^\nabla$). The action of $\mc{G}^\C$ does not preserve the equation, but may rather be used to transform a stable configuration $(\nabla,\Phi)$ -- that is, one such that $(\bar\partial^\nabla, \Phi)$ is stable as a Higgs bundle -- into a solution of the self-duality equation. This is one of the fundamental theorems of Higgs bundle theory, proven by Hitchin in his seminal paper \cite{HiSD}. 

2. 
Conversely, starting from the holomorphic data of a Higgs bundle, one can recover the same equations as a nonlinear condition for the harmonic metric. More precisely, 
if $(\bar\partial_E, \varphi)$ is an $\SL(2,\C)$-Higgs bundle on $E$, i.e.\@ $\bar\partial_E$ is a holomorphic structure on $E$ inducing $\det E \cong \mc{O}$ and $\bar\partial_E \varphi =0$, then we can recast Hitchin's equation as an equation for a Hermitian metric $h$ on $E$, now considered as variable. More precisely, let $\nabla^h$ denote the Chern connection of $h$ relative to $\bar\partial_E$ and $F^h$ its curvature. In a local holomorphic gauge, $\bar\partial_E = \bar\partial$ and $\nabla^h = d + h^{-1}(\partial h)$. Hitchin's equation then takes the shape
\begin{equation}\label{eq:harmmetric}
F^h + [\varphi, \varphi^{*_h}] = 0,
\end{equation}
where the adjoint is now taken with respect to the metric $h$. Again, if $(\bar\partial_E, \varphi)$ is stable, a solution exists and is called the \textit{harmonic metric}, and the triple $(\bar\partial_E, \varphi,h)$ a \textit{harmonic bundle} on $E$. If $h_0$ is the fixed metric from above, then $h( \cdot \,, \cdot) = h_0(H \, \cdot, \cdot)$ for $H$ $h_0$-Hermitian and positive definite, such that $g = H^{-1/2}$ is a complex gauge transformation satisfying $h_0 = h\cdot g$; thus, the pair
\[
\nabla = \nabla^h \ast g \quad \text{and} \quad \Phi = \varphi \ast g
\]
satisfies the self-duality equations in the form \eqref{eq:sd}.

When dealing with large-energy solutions of Hitchin's self-duality, the concept of a \emph{limiting configuration} \cite{MSWW} (or \emph{decoupled solution} of Hitchin's equation \cite{Moc}) has turned out to be useful. More precisely, let $(E,\bar\partial_E, t \varphi)$ be a ray of stable $\SL(2,\C)$-Higgs bundles, $t \in \R_+$, where, in addition, we assume that the holomorphic quadratic differential $q = \det \varphi$ has simple zeros only. We wish to describe the asymptotics of the family of harmonic metrics $h_t$ as $t \to \infty$. 

Equivalently, when fixing a background metric $h_0$ on $E$, we have $\bar\partial_E=\bar\partial_A$ for the Chern connection $A$ of $h_0$, and there will be a family of complex gauge transformations $g_t$ such that
\[
A_t = A \ast g_t \quad \text{and} \quad t\Phi_t = t \varphi \ast g_t
\]
solve the self-duality equations \eqref{eq:sd} for all $t \in \R_+$, i.e.\@ the configuration $(A_t,\Phi_t)$ solves the $t$-rescaled equation 
\begin{equation}\label{eq:t-Hitchin}
\bar\partial_A \Phi_t = 0 \quad \text{and} \quad F_{A_t} + t^2 [\Phi_t, \Phi_t^*]=0,
\end{equation}
for all $t \in \R_+$. Again, we wish to describe the asymptotics of the family of configurations $(A_t,\Phi_t)$ as $t \to \infty$. It turns out that, on the complement of the zero divisor $Z = q^{-1}(0)$, the family $(A_t,\Phi_t)$ converges to a configuration $(A_\infty, \Phi_\infty)$ defined on $M \setminus Z$ such that $\det \Phi_\infty = q$ and the limiting equation
 \[
\bar\partial_A \Phi_\infty = 0 \quad \text{and} \quad F_{A_\infty} = [\Phi_\infty, \Phi_\infty^*]=0 
\] 
is satisfied on $M \setminus Z$. Such a configuration $(A_\infty, \Phi_\infty)$ will be called \emph{a limiting configuration} for the Higgs bundle $(E,\bar\partial_E, \varphi)$. Note that any Higgs bundle $(E,\bar\partial_E, \varphi)$ such that $q = \det \varphi$ has simple zeros only is automatically stable. It is shown in \cite{MSWW} and \cite{Moc} that any such Higgs bundle admits a limiting configuration, which is unique up to gauge.

Since the point of view taken in the former is more relevant for this present work, we briefly describe how large-energy solutions of Hitchin's equation are obtained by desingularizing limiting configurations in \cite{MSWW}. Fix a Hermitian metric $h$ on $E$. Near a zero of $q$ (simple by assumption), the limiting configuration assumes a very specific shape in local coordinates, which admits a desingularization by a rotationally symmetric family of solutions of the $t$-rescaled Hitchin equation on the unit disk (called the fiducial solution, see \eqref{eqn:fid} for the precise shape). Gluing the limiting configuration to the fiducial solution using a partition of unity yields a family of approximate solutions, which may be deformed to true solutions of \eqref{eq:t-Hitchin} for sufficiently large $t$.

\subsection{Harmonic maps from surfaces to $ \dS_3$}
In this section, we explore harmonic maps from Riemann surfaces into de Sitter $3$-space from a gauge-theoretic perspective. This approach is analogous to the case of harmonic maps into hyperbolic 3-space, as discussed in  \cite{Donaldson} and \cite{OSWW}.

Recall that de Sitter $3$-space is usually defined as the one-sheeted hyperboloid
\[
\dS^3=\{ x \in \R^{1,3} : \langle x, x \rangle = 1 \,\}\subset\R^{1,3}
\]
equipped with the Lorentzian metric induced by the Minkowski inner product on $\R^{1,3}$. It has constant positive sectional curvature and has the structure of a Lorentzian symmetric space, as we will see below.

\subsubsection{The matrix model of $ \dS_3$}
We consider the {\em matrix model} of 
de Sitter 3-space
\[ \dS_3=\{g\in\mathrm{SL}(2,\C)\mid g^\dagger =g\},\]
where for $A=\begin{psmallmatrix} a&b\\c&d\end{psmallmatrix}\in\mathfrak{gl}(2,\C)$
\[A^\dagger=\begin{psmallmatrix} 1&0\\0&-1\end{psmallmatrix} A^*\begin{psmallmatrix}1&0\\0&-1\end{psmallmatrix}=\begin{psmallmatrix} \bar a&-\bar c\\-\bar b&\bar d\end{psmallmatrix}\]
is the adjoint of $A$ with respect to the standard indefinite Hermitian symmetric inner product $( \cdot \, , \cdot)$ on $\C^2$, defined by
\[((x_1,x_2),(y_1,y_2)):= x_1\bar y_1-x_2\bar y_2.\]
\begin{rem}\label{rem:twomamosdS3}
Using the identification $g\mapsto h=g\begin{psmallmatrix}1&0\\0&-1\end{psmallmatrix}$, we also use the model
\[\dS_3 = \{h\in\mathrm{GL}(2,\C)\mid \,\det(h)=-1, \, h^*=h\}=\{h\in\mathcal H\cong\R^{1,3}\mid -\det(h)=1\}\]
of Hermitian symmetric matrices of determinant $-1$ for de Sitter 3-space. This, in particular, recovers the description as the one-sheeted hyperboloid in Minkowski space.
\end{rem}

The space $ \dS_3$ is naturally equipped with a Lorentzian metric as follows: for $h\in \dS_3$
\begin{equation}\label{eq:herTmod}X,Y\in T_h \dS_3=\{X\in\mathfrak{gl}(2,\C)\mid X^\dagger=X,\, \tr(h^{-1}X)=0\}\end{equation}
we define
\begin{equation}\label{def:Gmet}G_h(X,Y):=-\tfrac{1}{2}\tr(h^{-1}X h^{-1}Y).\end{equation}

This Lorentzian metric agrees with the induced constant-curvature metric on the one-sheeted hyperboloid model of $\dS_3$ via the identification in Remark~\ref{rem:twomamosdS3}.
A direct computation shows that $\mathrm{SL}(2,\C)$ acts by isometries on $ \dS_3$ via
$$g \cdot A=gAg^\dagger,$$
for $g\in\SL(2,\C)$ and $A\in\dS_3$. Then, the stabilizer of the identity matrix in $\SL(2,\C)$ is $\SU(1,1)$.
This gives another realization of de Sitter 3-space
\[ \dS_3\cong\mathrm{SL}(2,\C)/\mathrm{SU}(1,1),\]
which is useful when studying harmonic maps from the gauge-theoretic point of view.

\begin{prop}\label{lem:dgds3}
There exists a complex rank 2 vector bundle $V\to  \dS_3$ with an indefinite metric $h$ and trivial connections $\nabla_L$ and $\nabla_R$ such that
\begin{itemize}
\item $\nabla=\tfrac{1}{2}(\nabla_L+\nabla_R)$ is metric on $(V,h);$ 
\item $i\;\mathfrak{su}_{1,1}(V)=T \dS_3$ as metric bundles, where the metric on $i\,\mathfrak{su}_{1,1}(V)$ is $-\tfrac{1}{2}\tr();$ 
\item the induced connection on $\mathfrak{su}_{1,1}(V)=T \dS_3$ by $\nabla$ is the Levi-Civita connection.
\end{itemize} 
Furthermore, $V\to  \dS_3$ admits a left $\mathrm{SL}(2,\C)$-action $L$ covering the action on $ \dS_3$ such that $\nabla_L$ and $\nabla_R$ are invariant. 
For $h=\text{id}\colon  \dS_3\to  \dS_3\subset\mathrm{SL}(2,\C)$ it holds
\[\nabla_R=\nabla_L \cdot h.\]
\end{prop}
\begin{proof}
Define $V=\C^2\to  \dS_3$ to be the trivial rank 2 bundle, with the indefinite metric 
\[\hat h_h=(h \, \cdot \, , \cdot )=( \cdot \, ,h \, \cdot)\]
on the fiber $V_h$,
for $(.\,,.)$ being the standard indefinite inner product on $\C^2$ as above. Define $\nabla_L=d$ to be the trivial
connection, and $\nabla_R=d \cdot h$, so that $\nabla= d+\tfrac{1}{2}h^{-1}dh.$
Then, a direct computation shows 
\begin{equation}
\begin{split}
d \hat h_h(s,t)&=d(hs,t)=(dh s,t)+(h ds,t)+(hs,dt)\\
&=\tfrac{1}{2}(\,(dh s,t)+(s,dh t)\,)+ \hat h_h(ds,t)+\hat h_h(s,dt)\\
&=\hat h_h(\nabla s,t)+\hat h_h(s,\nabla t)
\end{split}
\end{equation}
for all sections $s,t$ of $V$.
Furthermore, for $h\in  \dS_3$ we have
\[i\,\mathfrak{su}_{1,1}(V_h)=\{A\in  \mathfrak{sl}(2,\C) \mid (hA \, \cdot , \cdot)=( \cdot \,,hA \, \cdot )\}=\{A\in  \mathfrak{sl}(2,\C)\mid\,  h^{-1}A^\dagger h=A\}.\]
The isomorphism with $T_h \dS_3=\{X\in\mathfrak{gl}(2,\C)\mid X^\dagger=X,\, \tr(h^{-1}X)=0\}$
is given by 
\begin{equation}\label{eq:trvitan}A\mapsto X=hA.\end{equation}
That this isomorphism is an isometry follows directly from Definition \eqref{def:Gmet}.  
Following \cite{OSWW}, we call $i\,\mathfrak{su}_{1,1}(V_h)$ the {\em trace-free model} of the tangent bundle, and the incarnation in \eqref{eq:herTmod} the {\em Hermitian model} of the tangent bundle.

Next, we show that $\nabla$ is torsion-free on the tangent bundle $i\,\mathfrak{su}_{1,1}(V)=T \dS_3.$
Consider two vector fields given by 
$X,Y\colon U\subset  \dS_3\to\mathfrak{gl}(2,\C)$ with
$\tr(h^{-1}X_h)=\tr(h^{-1}Y_h)=0$ and $X^\dagger=X,\, Y^\dagger=Y.$
Likewise, in the trace-free model, the vector fields are given by
$A=h^{-1}X,B=h^{-1}Y\colon U \subset  \dS_3\to\mathfrak{sl}(2,\C)$ with $h^{-1}A^\dagger h=A$ and $h^{-1}B^\dagger h=B$
for all $h\in U.$  We compute
\begin{equation*}
\begin{split}
\nabla_XB-\nabla_YA&=X\cdot B- Y\cdot A+\tfrac{1}{2}[h^{-1}dh(X),B]_{\mathfrak{sl}(2,\C)}-\tfrac{1}{2}[h^{-1}dh(Y),A]_{\mathfrak{sl}(2,\C)}\\
&=X\cdot B- Y\cdot A+[A,B]_{\mathfrak{sl}(2,\C)}
\end{split}
\end{equation*}
and
\begin{equation*}
\begin{split}
h^{-1}[X,Y]&=h^{-1}(\,X\cdot (hB)-Y\cdot (hA)\,)=h^{-1}(\,XB+h X\cdot B-YA-h Y\cdot A\,)\\
&=AB+ X\cdot B- BA-Y\cdot A.
\end{split}
\end{equation*}
Thus, $\nabla_XA-\nabla_YB-h^{-1}[X,Y]=0$, which is equivalent to $\nabla$ being torsion-free, and hence the Levi-Civita connection.

Consider the left $\mathrm{SL}(2,\C)$-action $L$ on $V\to  \dS_3$  given by
\begin{equation}\label{eq:left}L_g(h,v):=(ghg^\dagger,gv)\end{equation} for $g\in\mathrm{SL}(2,\C).$
Clearly, $L_g^* \nabla_L=\nabla_L$ for all $g\in\mathrm{SL}(2,\C).$
The connection $\nabla_R$ is trivial by the parallel gauge $h^{-1}$:
\[\nabla_R \cdot h^{-1}=d-dh h^{-1}+h h^{-1}dh h^{-1}=d.\]
Let $\widetilde v$ be a parallel (i.e.\@ constant) section with respect to this trivialization of $\nabla_R.$ Then,
 with $v=h^{-1}\widetilde v$ the action $L_g$ becomes
\begin{equation}\label{eq:right}
\begin{split}
L_g(h, \widetilde v)&\cong L_g(h,h v)=(ghg^\dagger, gh v)\cong (ghg^\dagger, (ghg^\dagger)^{-1}gh  \widetilde v)= (ghg^\dagger, (g^\dagger)^{-1} \widetilde v).
\end{split}
\end{equation}
Clearly, the action preserves the trivial connection $\nabla_R.$
\end{proof}

\subsubsection{Equivariant harmonic maps into $ \dS_3$}
Let $N\colon M\to  \dS_3$ be a smooth map from a Riemann surface $M$ to de Sitter 3-space, and consider its differential
\[\phi:=N^{-1}dN\in\Omega^1(M,i\, \mathfrak{su}_{1,1}(V))\subset \Omega^1(M,\mathfrak{sl}(V))\]
where we denote the pull-back of $V$ to $M$ via $N$ again by $V$. 
It satisfies the integrability equation
\begin{equation}\label{eq:inteq1}
d^\nabla \phi=0
\end{equation}
where $\nabla=N^*\nabla$ is the pull-back to $M$ of the (Levi-Civita) connection on $ \dS_3$ as in Proposition \ref{lem:dgds3}.
The Dirichlet energy of $N$ is given by
\begin{equation}\label{eq:defEN}E(N):=-\tfrac{1}{2}\int_M G(dN\wedge*dN)= \tfrac{1}{4} \int_M \tr(\phi\wedge*\phi).\end{equation}
Then, 
$N$ is harmonic, i.e.\@ a critical point of $E$, if and only if it satisfies the Euler-Lagrange equation
\begin{equation}\label{eq:inteq2}
d^\nabla*\phi=0.
\end{equation}
Furthermore, harmonic maps into de Sitter 3-space can be characterized in terms of families of flat connections:
\begin{lemma}
Let $N$ be a harmonic map into $dS^3$ with $\nabla,\phi$  as above.  Decompose 
$\tfrac{1}{2}\phi=\Phi+\Psi$ with $\Phi\in\Omega^{1,0}(M, \mathfrak{sl}(V))$ and
$\Psi\in\Omega^{0,1}(M, \mathfrak{sl}(V))$. Then it holds
\begin{itemize}
\item $\bar\partial^\nabla\Phi=0$, and equivalently $\partial^\nabla\Psi=0$;
\item $\nabla\pm\phi$ are flat.
\end{itemize}
In fact, $\nabla^\lambda=\nabla+\lambda^{-1}\Phi+\lambda\Psi$ is flat for all $\lambda\in\C^*.$
\end{lemma}
\begin{proof}
The first part directly follows from  $d^\nabla\phi=0=d^\nabla*\phi$ using
$\Phi=\tfrac{1}{4}(\phi-i*\phi)$ and $\Psi=\tfrac{1}{4}(\phi+i*\phi).$
The second part  follows from the construction of $\nabla$ in Proposition \ref{lem:dgds3}. For the last part, we expand the curvature
\[F^{\nabla^\lambda}=\lambda^{-2}[\Phi,\Phi]+\lambda^{-1}\bar\partial^\nabla\Phi+F^\nabla+[\Phi,\Psi]+\lambda\partial^\nabla\Psi+\lambda^2[\Psi,\Psi],\]
and observe that the curvature is constant in $\lambda$ using $[\Phi,\Phi]=0=[\Psi,\Psi]$ and the first part of the lemma. Thus, by the second part, the curvature vanishes for all $\lambda$.
\end{proof}

Let $\widetilde M\to M$ be the universal cover, and $\rho\colon\pi_1( M,p_0)\to\mathrm{SL}(2,\C)$ be a representation. 
The fundamental group  $\pi_1( M,p_0)$ acts on $\widetilde M\to M$ via deck transformations from the right.
A map $N\colon\widetilde M\to  \dS_3$ is called ($\rho$-)equivariant if it satisfies 
\[\gamma^*N=\rho(\gamma^{-1})N\rho(\gamma^{-1})^\dagger\]
for all
$\gamma\in \pi_1( M,p_0).$
 The energy density
of an equivariant map is well-defined on $ M$, and we have the notion of equivariant harmonic maps.

\begin{thm}\label{thm:hards3asso}
Let $N\colon\widetilde M\to  \dS_3$ be $\rho$-equivariant and harmonic. There exists a complex rank 2 vector bundle
$W\to M$ equipped with an indefinite Hermitian inner product  $h$, a unitary connection $\nabla$ with respect to $h$, a Higgs field $\Phi$, and its
$h$-adjoint anti-Higgs field $\Psi$
such that
\begin{enumerate}
\item $\nabla^\lambda:=\nabla+\lambda^{-1}\Phi+\lambda \Psi$
is flat for all $\lambda\in\C^*;$ 
\item $\nabla^{-1}$ and $\nabla^{1}$ have monodromy representation $\rho$ and $(\rho^{-1})^\dagger;$
\item $N$ is a gauge transformation between $\nabla^{-1}$ and $\nabla^{1}$ on $\widetilde M$.
\end{enumerate}
Conversely, a family of flat connections as in (1) defines a harmonic map $N\colon\widetilde M\to  \dS_3$
via (3), which is equivariant with respect to $\rho$ in (2).
\end{thm}
\begin{proof}
Consider the equivariant map $N\colon\widetilde  M\to  \dS_3$, and take the pull-backs
$N^*V,N^*h,N^*\nabla$
together with the differential $\phi=2\Phi+2\Psi=N^{-1}dN.$ Then,
\[\nabla^\lambda=N^*\nabla+\lambda^{-1}\Phi+\lambda\Psi\]
is flat for all $\lambda$ as a direct consequence of the previous lemma. 
Moreover, by construction we have
$\nabla^{-1}=N^*\nabla_L$ and $\nabla^{1}=N^*\nabla_R$,
 as well as
$\nabla^{-1}.N=\nabla^1.$
The fundamental group acts on $N^{*}V$ $\to\widetilde M$ via
\[(p,v).\gamma=(\,p.\gamma,\rho(\gamma^{-1})v\,)\]
for $\gamma\in\pi_1( M,p_0)$. This action is compatible with $\nabla^{-1}=N^*\nabla_L$ and $\nabla^1=N^*\nabla_R,$
which implies that
$\nabla^{-1}$ and $\nabla^1$ are well-defined on the complex rank-2 Hermitian vector bundle
\[W:= N^*V/\pi_1( M,p_0)\to M\]
with its induced indefinite metric.
Then, also the interpolation
$\nabla^\lambda$ is well-defined on $W\to M.$
Finally, $\nabla^{-1}$ and $\nabla^1$ have monodromy $\rho$ and $(\rho^{-1})^\dagger$ by
\eqref{eq:left} and \eqref{eq:right}, respectively.

The converse direction follows by reversing the arguments.
\end{proof}
In particular, the data $(W,h,\nabla,\Phi,\Psi)$ provided by Theorem~\ref{thm:hards3asso} is an $\SU(1,1)$-self-duality solution, and the family $\nabla^\lambda$ is the associated family of flat $\SL(2,\C)$-connections. This is the $\SU(1,1)$-analogue of the standard non-abelian Hodge correspondence for $\SU(2)$ harmonic maps into $\H^3$.

\subsection{$\lambda$-connections and the Deligne--Hitchin moduli space}
\label{sec:lambdacon}
We recall the twistor description of the Hitchin moduli space in terms of the Deligne--Hitchin moduli space, and explain how real holomorphic sections correspond to solutions of the $\SU(2)$ or $\SU(1,1)$ self-duality equations (possibly with indefinite metric).
Additionally, we briefly review the relationship between families of flat connections,  sections of the Deligne--Hitchin moduli space, and harmonic maps into associated symmetric spaces. This material is based on \cite{Simpson}, \cite{BHR}, \cite{BeHR}.

Families of flat connections are closely linked to harmonic maps into symmetric spaces. For the purposes of this work, only the two (Riemannian and Lorentzian) symmetric spaces arising as quotients of $\SL(2,\C)$ by the compact form $\SU(2)$ and the split real form $\SU(1,1)$ will be of relevance. The relationship between flat connections and equivariant harmonic maps is perhaps most easily visible in the following set-up. Suppose that $M$ is a Riemann surface, and consider the trivial vector bundle $E = \underline{\C^2}$ over $M$ equipped with the standard Hermitian metric $h_0$.

Suppose that $(\nabla,\Phi,h_0)$ is a solution of the self-duality equation, i.e.,
\[\bar\partial^\nabla\Phi=0 \quad\text{and}\quad F^\nabla=-[\Phi,\Phi^*].\]
Note that these two equations automatically imply that $\nabla+\Phi+\Phi^*$ is flat, but even more is true:
for every $\lambda\in\C^*$, the connection
\begin{equation}\label{eq:associated-family}\nabla^\lambda=\nabla+\lambda^{-1}\Phi+\lambda\Phi^*\end{equation}
is flat. In fact, for $\lambda\in \mathbb S^1\subset\C^*$, $(\nabla,\lambda^{-1}\Phi,h_0)$ is another solution of the self-duality equations, with flat connection $\nabla^\lambda$. As the curvature of $\nabla^\lambda$ depends holomorphically on $\lambda$, flatness for all $\lambda\in\C^*$ follows.
Furthermore, the family of flat connections satisfies the following reality condition:
\[
\left(\nabla^{-\ol{\lambda}^{-1}} \right)^* = \nabla^\lambda \quad \text{for all } \lambda\in\C^*.
\]

On the other hand, if $\hat\nabla^\lambda=\hat\nabla+\lambda^{-1}\hat\Phi+\lambda \hat \Psi$ 
is a family of flat connections with $\hat\Phi\in\Omega^{1,0}(M,\mathfrak{sl}(2,\C))$, and
instead satisfies the reality condition
\begin{equation}\label{eq:sec-real-cond}
\left(\hat\nabla^{-\ol{\lambda}^{-1}} \right)^* = \hat\nabla^\lambda \cdot g,
\end{equation}
where $g = i \delta$ and $\delta = \begin{psmallmatrix} 1 & 0 \\ 0 & -1 \end{psmallmatrix}$, then $(\hat\nabla, \hat\Phi)$ solves the $\SU(1,1)$ self-duality equation with respect to the standard indefinite Hermitian metric $\hat h_0:=\langle \cdot \, , \delta \,\cdot \rangle$:
$\hat \nabla$ is unitary with respect to $\hat h_0,$ $\hat\Psi=\hat\Phi^\dagger$ is the adjoint of $\hat\Phi$ with respect to $\hat h_0$, and
\[\bar\partial^{\hat\nabla}\hat\Phi=0 \quad\text{and}\quad F^{\hat\nabla}=-[\hat\Phi,\hat\Phi^\dagger].\]

Similarly, as for the harmonic maps to $\mathbb H^3$, for a family of parallel frames $\hat F^\lambda : \wt{M} \to \SL(2,\C)$, $\lambda \in \mathbb S^1 \subset \C^*$, for $\hat\nabla^\lambda$, the maps $f_\lambda = \left(\hat F^{-\overline{\lambda}^{-1}} \right)^* \hat F^\lambda$ take values in $\{ A \in \SL(2,\C) : A = \delta^{-1} A^* \delta \}$, i.e.\@ the matrix model of de Sitter 3-space $\SL(2,\C)/\SU(1,1)$, and are harmonic. Note that $\delta^{-1} A^* \delta$ is the adjoint with respect to the indefinite metric $\langle \cdot \, , \delta \, \cdot \rangle$.

The Deligne--Hitchin moduli space is a complex analytic construction of the twistor space of the Hitchin moduli space of solutions to the self-duality equations. The construction is due to Deligne, and in its original form may be found in \cite{Simpson}. Our presentation is based on \cite{BHR}. It relies on the notion of $\lambda$-connections, which may be seen as an interpolation between Higgs pairs and flat connections.

A $\lambda$-connection on a complex vector bundle $E$ over a Riemann surface $M$ is a triple $(\ol{\partial}^E, D, \lambda)$ consisting of a complex number $\lambda$, a holomorphic structure $\ol{\partial}^E$, and a $C^\infty$ differential operator $D : \Omega^0(M,E) \to \Omega^{1,0}(M,E)$ satisfying the $\lambda$-product rule
\[
D (f s) = \lambda \partial f \otimes s + f Ds
\]
and the integrability condition $D \ol{\partial}^E + \ol{\partial}^E D = 0$. We will also assume throughout that $\ol{\partial}^E$ induces the standard holomorphic structure on $\det E = \underline{\C}$. Likewise, we assume that $D$ induces on $\det E$ the operator $\lambda \partial$.

For $\lambda = 0$, the operator $D$ is a zeroth-order differential operator and can therefore be identified with an endomorphism-valued $(1,0)$-form, and is therefore the same as a $\SL(2,\C)$-Higgs pair on $E$.
For a $\lambda$-connection with $\lambda \neq 0$, the operator $\overline{\partial}^E + \lambda^{-1} D$ defines a connection on $E$, and due to the integrability condition on $\overline{\partial}^E$ and $D$, this connection turns out to be a flat $\SL(2,\C)$-connection on $E$.

In the case of $E= \underline{\C^2}$, the group of complex gauge transformations $\mc{G}^{\C}$ is $C^\infty(M, \SL(2,\C))$. Elements of this group act on $\lambda$-connections via
\[
(\ol{\partial}^E, D, \lambda) \cdot g = \left( g^{-1} \circ \ol{\partial}^E \circ g, g^{-1} \circ D \circ g, \lambda \right).
\]
To get a well-behaved moduli space, we consider a (poly-)stability condition. A $\lambda$-connection $(\ol{\partial}^E, D, \lambda)$ is called {\em stable} if any $D$-invariant holomorphic line subbundle $F \subset (E, \ol{\partial}^E)$ satisfies $\deg F < 0$. It is {\em poly-stable} if $E$ splits as a direct sum of stable $\lambda$-connections whose underlying holomorphic vector bundles have degree $0$. Note that for $\lambda\neq0$, stability of a $\lambda$-connection is equivalent to irreducibility of the associated flat connection.

The moduli space of (poly-stable) $\lambda$-connections
\[
\mc{M}_{\operatorname{Hod}} = \mc{M}_{\operatorname{Hod}}(M) = \left\{ (\ol{\partial}^E, D, \lambda) \text{ poly-stable } \lambda\text{-connection} \right\} / \mc{G}^{\C}
\]
is called {\em Hodge moduli space}. This is a complex space equipped with a natural holomorphic fibration $\pi : \mc{M}_{\operatorname{Hod}} \to \C$, $\left[(\ol{\partial}^E, D, \lambda)\right] \mapsto \lambda$.

The Deligne--Hitchin moduli space extends this fibration from $\C$ to $\bb{CP}^1$ by gluing the Hodge moduli space over $M$ with the Hodge moduli space over the conjugate Riemann surface $\ol{M}$. More precisely, we define
\[
\mc{M}_{\operatorname{DH}} = \left(\mc{M}_{\operatorname{Hod}}(M) \sqcup \mc{M}_{\operatorname{Hod}}\left(\overline{M}\right) \right)/_{\sim},
\]
where the equivalence relation is given by
\[
\left(\ol{\partial}^E, D, \lambda\right)_M \sim \left(\lambda^{-1} D, \lambda^{-1} \ol{\partial}^E, \lambda^{-1} \right)_{\ol{M}}.
\]
By the non-abelian Hodge correspondence \cite{HiSD}, this space is biholomorphic to the twistor space of the Hitchin moduli space
\[
\mc{T}\left(\mc{M}_{\operatorname{Hit}}\right) = \mc{M}_{\operatorname{Hit}} \times \bb{CP}^1
\]
equipped with the complex structure
\[
I_{x,\lambda} = \left( \frac{1-|\lambda|^2}{1+|\lambda|^2} I + \frac{\lambda + \ol{\lambda}}{1+|\lambda|^2} J - \frac{i(\lambda - \ol{\lambda})}{1+|\lambda|^2} K \right) \oplus i,
\]
where $(I,J,K)$ denotes the hyperk\"ahler triple at $x \in \mc{M}_{\operatorname{Hit}}$ and $i$ denotes the standard complex structure on $\bb{CP}^1$. The fibre-preserving biholomorphism is given by
\[
\left([\nabla, \Phi], \lambda \right) \mapsto \left[ (\ol{\partial}^\nabla+  \lambda \Phi^*, \lambda \partial^\nabla + \Phi, \lambda) \right],
\]
i.e.\@ by the associated family of $\lambda$-connections \eqref{eq:associated-family}.
Any twistor space carries an antiholomorphic involution covering the antipodal map $\lambda \mapsto -\overline{\lambda}^{-1}$. On $\mc{T}\left(\mc{M}_{\operatorname{Hit}}\right)$, this map is simply given by
\[
\mc{T}\left(\mc{M}_{\operatorname{Hit}} \right) \to \mc{T}\left(\mc{M}_{\operatorname{Hit}} \right), \quad \left([\nabla, \Phi], \lambda\right) \mapsto \left([\nabla, \Phi], -\overline{\lambda}^{-1}\right).
\]
This involution $\tau : \mc{M}_{\operatorname{DH}} \to \mc{M}_{\operatorname{DH}}$ can also be written in terms of $\lambda$-connections as follows:
\[
\left[(\ol{\partial}^E, D, \lambda)_M \right] \mapsto \left[((\ol{\partial}^E)^*, -D^*, -\ol\lambda)_{\ol M }\right]= \left[ \left(\ol{\lambda}^{-1} D^*, -\ol{\lambda}^{-1} \left(\ol{\partial}^E\right)^*, - \ol{\lambda}^{-1} \right)_M\right].
\]

A section $s$ is called {\em stable}, if $s(\lambda)$ is stable for every $\lambda \in \bb{CP}^1$.
A section $s : \bb{CP}^1 \to \mc{M}_{\operatorname{DH}}$ is called {\em $\tau$-real} or {\em real}, if \[s\left(-\ol{\lambda}^{-1}\right) = \tau(s(\lambda))\quad \forall \lambda\in\C P^1.\]

Suppose that a real section $s$ lifts over $\C$ to a family of stable $\lambda$-connections $(\ol{\partial}^\lambda, D^\lambda, \lambda)$. In this case, we can associate a family of flat connections $\nabla^\lambda = \ol{\partial}^\lambda + \lambda^{-1} D^\lambda$, and the $\tau$-reality condition becomes
\begin{equation}\label{eq:realitycon}
\left(\nabla^{-\overline{\lambda}^{-1}}\right)^* = \nabla^\lambda \cdot g(\lambda)
\end{equation}
for some family of gauge transformations $g(\lambda)$. 
If the lift is given by a Laurent polynomial of degree $1$ in $\lambda$, and the section and the Higgs field at $\lambda=0$  are stable,
this is (gauge-)equivalent to precisely one of the two reality conditions we saw earlier, see Theorem \ref{thm:posnegsecLift} below.
In fact, by \cite{BHR}, $g(\lambda)=g$ is then constant, and the reality condition implies $(g^*)^{-1}g=\pm\mathrm{Id}.$ If $(g^*)^{-1}g=\mathrm{Id},$ the section is called {\em negative} (following different conventions in \cite{BHR}), and if $(g^*)^{-1}g=-\mathrm{Id},$ the section is called {\em positive}.
In principle, $\tau$-negative sections correspond to $\SU(2)$-harmonic bundles, and
$\tau$-positive sections correspond to $\SU(1,1)$-harmonic bundles equipped with an indefinite Hermitian metric; the former may have singularities and underlie our construction of transgressive harmonic maps.
\begin{rem}
A word of warning: even if a section is stable, there can be a lift which is not stable at $\lambda=0.$ The simplest example for this phenomenon is well-known:
Let $\nabla$ be the oper whose projective structure is the uniformization of the compact Riemann surface $M$. Then, $\lambda\mapsto (\bar\partial^\nabla, \lambda \partial^\nabla,\lambda)$ is a lift of the section associated to the corresponding solution of the self-duality equations, but it is
unstable at $\lambda=0.$ This lift is gauge-equivalent to the standard lift \eqref{eq:associated-family} by a family of gauge transformations which does not extend holomorphically to $\lambda=0$ as an isomorphism. Moreover, this constant lift 
is real with respect to the second reality condition \eqref{eq:sec-real-cond}. This means that, by choosing different lifts, we can switch between the $\mathrm{SU}(2)$ and $\mathrm{SU}(1,1)$ theories. 
We will generalize this observation in the next section.
\end{rem}

\begin{rem}
In the case of stable sections, the only possibility to switch between the $\mathrm{SU}(2)$ and $\mathrm{SU}(1,1)$ theories is by gauging with $\lambda$-dependent gauge transformations which become singular at $\lambda=0.$
In the case of sections which are not stable, i.e.\@ which admit reducible connections in their associated family $\nabla^\lambda$ for certain $\lambda_0\in\C^*$, $\lambda$-dependent gauge transformations which become singular at $\lambda_0$
can be used to switch types, see \cite[Lemma 4.3]{BHS} for examples in the parabolic setup over the punctured sphere.
\end{rem}

We will primarily be interested in irreducible, holomorphic sections of $\mc{M}_{\operatorname{DH}}$. In this case, \cite[Lemma 2.2]{BHR} ensures that there exists a holomorphic lift over $\C$ which is stable at $\lambda=0$. 
An irreducible section $s : \bb{CP}^1 \to \mc{M}_{\operatorname{DH}}$ will be called {\em admissible} if it has a lift of the form
\begin{equation}\label{eq:admissiblelift}
\lambda \mapsto \left( \ol{\partial}^E + \lambda \Psi, \lambda \partial^E + \Phi, \lambda \right)
\end{equation}
such that $(\ol{\partial}^E,\Phi)$ is a stable Higgs pair on $M$ and $(\partial^E,\Psi)$ is a stable Higgs pair on $\ol{M}$.
The associated family of flat connections is then of the form $\lambda^{-1} \Phi + \partial^E + \ol{\partial}^E + \lambda \Psi$, i.e.\@ its Laurent series is collapsed as in the example we discussed at the beginning of the section. By \cite{HH} and \cite{Hgraft}, there exist irreducible real sections which admit a lift of the form \eqref{eq:admissiblelift} but which are not admissible.

We have the following theorem, which improves \cite[Theorem 3.6 and Lemma 3.9]{BHR}.
\begin{thm}\label{thm:posnegsecLift}
Suppose $s : \bb{CP}^1 \to \mc{M}_{\operatorname{DH}}$ is a holomorphic, admissible  $\tau$-real section of the fibration $\mc{M}_{\operatorname{DH}} \to \bb{CP}^1$.
If $s$ is $\tau$-negative, it is a twistor line, i.e.\@ it is given by the solution of a $\mathrm{SU}(2)$ self-duality equation for some
Hermitian metric. If $s$ is $\tau$-positive, it is given by the solution of a $\mathrm{SU}(1,1)$ self-duality equation for some
indefinite Hermitian metric.
\end{thm}
\begin{proof}
The first case is proven in \cite[Theorem 3.6]{BHR}. For the second case of $\tau$-positive holomorphic sections, we have to construct an
appropriate indefinite Hermitian metric. Let $\nabla^\lambda=\lambda^{-1}\Phi+\nabla+\lambda\Psi$ be the lift, and $g(\lambda)$ be a gauge
such that \eqref{eq:realitycon} holds. As in \cite{BHR}, $g$ must be $\lambda$-independent because $s$ is admissible. And because $s$ is positive, we have $(g^*)^{-1}g=-\mathrm{id},$ or $g=-g^*.$ Consider $ h=ig,$ which is then Hermitian symmetric and of signature $(1,1)$, as
its determinant is $-1$. This is a map into de Sitter 3-space $ \dS_3$, the latter space being diffeomorphic to the product of the 2-sphere with an open interval.
Therefore, $h$ has a topological degree, as discussed in Section 3 of \cite[BHR]{BHR}, where the case of degree $0$ is specifically addressed. 
We define the corresponding Hermitian inner product $\hat h$ of signature $(1,1)$ by multiplying the standard positive definite inner product on the eigenlines of $h$ with the respective eigenvalue of $h$. Then $\nabla$ is unitary with respect to $\hat h$, and $\Phi+\Psi$ is Hermitian symmetric with respect to $\hat h$.
Thus, by the converse direction of Theorem \ref{thm:hards3asso}, $\hat h$ is a harmonic indefinite Hermitian metric for which $\nabla^\lambda$ is the associated family. Consequently, we obtain a solution to the $\mathrm{SU}(1,1)$ self-duality equations.
\end{proof}

We will construct solutions of the $\SU(1,1)$ self-duality equations for which the Higgs pair is not stable. In particular, we are not in the exact situation described in the above theorem. We have not been able to drop the condition in Theorem \ref{thm:posnegsecLift} that the Higgs fields of the section are stable. However, admissible $\tau$-positive sections have been constructed in \cite[Theorem 3.4]{BHR}.

\section{Hitchin self-duality equations and transgressive harmonic maps}
\label{sect:geometric_dualities}

\subsection{Oblique Gau\ss\@ maps for harmonic maps into $\H^3$}
We will introduce in Definition \ref{defn:transgressive} the class of transgressive harmonic maps into the conformal 3-
sphere. Such maps are allowed to 'intersect infinity' (i.e. the equatorial sphere) in a controlled, orthogonal way. Their generalized 
Gau\ss\@ map then extends as a globally defined harmonic map into de Sitter 3-space, which may drop rank precisely along the 
boundary curve. This will be the geometric counterpart of certain $\SU(1,1)$-type solutions of the self-duality equations constructed 
later.

Minimal surfaces in hyperbolic $3$-space $\H^3$ which intersect the boundary at infinity orthogonally have 
been studied in \cite{BaBo, HH}. Their Gau\ss\@ map extends smoothly through their singularity set, and can be interpreted as the conformal Gau\ss\@ map (see \cite{Blaschke,BPP-Schwarzian}) of their extension to the conformal $3$-sphere. Existence of non-conformal harmonic maps into hyperbolic $3$-space
which intersect the boundary orthogonally follows from \cite{Hgraft}. 
This motivates a generalization of the Gau\ss\@ map to harmonic (not necessarily conformal) maps into $\H^3$.

In the following definition, we view $\H^3_\pm$ as $\{ x \in \R^{1,3} : \langle x, x \rangle = -1 \,,\, \pm x_0>0\}\subset\R^{1,3}$
and
$\dS_3=\{ x \in \R^{1,3} : \langle x, x \rangle = 1 \,\}\subset\R^{1,3}$, see Remark \ref{rem:twomamosdS3}.
\begin{defn}
  \label{defn:hyperbolic-oblique-gauss-map-1}
  Let $M$ be a Riemann surface and $f : M \to \H^3_\pm$ an immersion. By $Q \in \Gamma(K^2)$ we denote the {\em hyperbolic Hopf differential} $Q = \left(f^*\langle \cdot \, , \cdot \rangle\right)^{2,0} = \langle \partial f, \partial f \rangle$.
Assume that there exists
 $\omega \in \Omega^{1,0}(M)$ such that $\omega^2 = Q$.

  A map $N : M \to  \dS_3$ is the {\em oblique hyperbolic Gau\ss\@ map} of $f$ (with respect to $\omega$), if
  \begin{enumerate}
  \item $N(x) \in T_{f(x)} \H^3_\pm$ for every $x \in M$,
  \item $\langle df, N \rangle = \omega + \ol{\omega}$,
  \item $\langle N, N \rangle = 1$,
  \item $N$ is compatible with the orientation of  $M$, see \eqref{eq:defOrN} below.
  \end{enumerate}
\end{defn}
Note that in general a quadratic differential does not admit a square root, e.g.\ if $Q$ is holomorphic with simple zeros. But when $Q$ is holomorphic, we can always pass to a --- possibly branched --- covering where such a square root exists. This covering is called the {\em Hitchin curve} (or {\em spectral curve}). We will implicitly allow ourselves to pass to this branched cover when necessary.

The first condition on $N$ constrains $N(x)$ to lie in a $3$-dimensional Euclidean subspace of $\R^{1,3}$. The second condition further picks out an affine line in this subspace. The third equation is a quadratic equation on this affine line. The following proposition guarantees that this quadratic equation always has two distinct solutions, and exactly one of them represents the orientation of $M:$ if $z=x+iy$ is a local holomorphic coordinate on $M$, we impose the condition
\begin{equation}\label{eq:defOrN}\det(f,N,f_x,f_y)>0.\end{equation}

The next proposition shows that this generalised Gau\ss\@ map exists and is unique, once a square root $\omega$ of the Hopf differential is fixed.
\begin{prop}\label{pro:existN}
  For an immersion $f : M \to \H^3_\pm$ on a connected Riemann surface and a square root $\omega \in \Omega^{1,0}(M)$ of the hyperbolic Hopf differential, there exists a unique oblique hyperbolic Gau\ss\@ map $N$. If $f$ is smooth, then $N$ is smooth.
\end{prop}
\begin{proof}
Given $x \in M$, we split $T_{f(x)} \H^3_\pm$ into $\im d_xf$ and its orthogonal complement. Let $\nu$ be a unit vector in $\im d_xf^\perp$ representing
the orientation of $M,$ i.e., $\nu,f_x,f_y$ are positively oriented for oriented coordinates $(x,y)$ of $M.$

Let $n_0 \in \im d_xf$ be the dual vector of $\omega_x + \ol{\omega}_x$, where we identify $T_xM$ and $\im d_xf$. In particular, the metric on $T_xM$ is the pull-back metric of $\langle \cdot \, , \cdot \rangle$ on $\im d_xf$.

With the ansatz $N(x) = n_0 + t \nu$, the equation for $N(x)$ becomes $\langle n_0, n_0 \rangle + t^2 = 1$. Therefore, it suffices to show that $0 \leq \langle n_0 , n_0 \rangle < 1$ to see that for every $x \in M$ there are precisely two real solutions for $t$. The positive solution yields the oblique hyperbolic Gau\ss\@ map.
Note that since $n_0 \in T_{f(x)} \H^3_\pm$, the lower bound on $\langle n_0, n_0 \rangle$ follows from the fact that $\langle \cdot \, , \cdot \rangle$ is positive definite on $T_{f(x)} \H^3_\pm$.

To see that $\langle n_0, n_0 \rangle < 1$ it suffices to show that $\eta_x := \omega_x + \ol{\omega}_x \in T_x^* M$ has norm smaller than $1$.
Let $e_1, e_2$ be an oriented orthonormal basis, such that $J e_1 = b e_2$. Then, $0 < b \leq 1$. 
Observe that $\omega(e_1)^2 = Q(e_1, e_1)$ and
\begin{equation*}\begin{split}
Q(e_1, e_1) &= \frac 1 4 \langle df(e_1) - i df(Je_1), df(e_1) - i df(Je_1) \rangle \\&= \frac 1 4 \left( \langle df(e_1), df(e_1) \rangle - \langle df (Je_1), df(Je_1) \rangle\right) = \frac 1 4 (1 - b^2).
\end{split}\end{equation*}
Therefore, $\omega(e_1) = \pm \sqrt{1-b^2}$. Since $\omega$ is a $(1,0)$-form, $\omega(Je_1) = i \omega(e_1)$, $\omega(e_2)$ is imaginary and therefore $\eta (e_2) = 0$. This implies $\eta^\sharp = \pm \sqrt{1-b^2} e_1$, which clearly has norm less than one.

Smoothness of $N$ follows from the construction, as $\langle n_0, n_0 \rangle<1$ holds globally.
\end{proof}

\begin{rem}
\label{rem:Riemtarget}
The construction of the oblique Gau\ss\@ map can be generalized to an immersion of a Riemann surface $M$ into any $3$-dimensional Riemannian manifold. We will later make use of this when the target manifold is the round $3$-sphere $\mathbb{S}^3$. 
\end{rem}

\begin{prop}
\label{prop:hyperbolic_gauss_map_harmonic}
  Suppose $f : M \to \H^3_\pm$ is a harmonic immersion, $\omega \in \Omega^{1,0}(M)$ is a square root of the hyperbolic Hopf differential, and $N : M \to  \dS_3$ is the oblique hyperbolic Gau\ss\@ map.
  Then $N$ is harmonic.
\end{prop}
\begin{proof}
We compute in local complex coordinates and show that $N_{z\bar z}$ is everywhere proportional to $N$, i.e. $N$ solves the harmonic map equation into $\dS_3$.

First, note that the second equation in Definition \ref{defn:hyperbolic-oblique-gauss-map-1} is equivalent to the equation
\[
\langle N, \partial f \rangle = \omega.
\]
We will use this complex form to prove the result.

It is well known that $\langle \partial f, \partial f\rangle $ is holomorphic if $f$ is harmonic, and there are two cases to consider. Either $Q = \langle \partial f, \partial f \rangle$ vanishes everywhere (and $f$ is conformal) or it has isolated zeros (and $f$ is non-conformal away from the zeros).

Let us first assume $Q$ vanishes identically. Choose holomorphic coordinates. Then, the equations in Definition \ref{defn:hyperbolic-oblique-gauss-map-1} become
\[
\langle N, f \rangle = 0, \quad \langle N, f_z \rangle = 0, \quad \langle N, N \rangle = 1.
\]
Here, in the second equation $\langle \cdot \, , \cdot \rangle$ is to be understood as the complexification of the Minkowski inner product on $\R^{1,3}$. Observe that this equation also has a conjugate version $\langle N, f_{\ol{z}} \rangle = 0$.
Since $f$ is an immersion, these conditions ensure that $f, f_z, f_{\bar{z}}, N$ form a basis of the complexification $\R^{1,3} \otimes \C$.
The harmonicity of $N$ in holomorphic coordinates is equivalent to $N_{z\bar{z}} = \mu N$ for some function $\mu$. To compute $N_{z\bar{z}}$, it is convenient to first compute $N_z$ and $f_{zz}$. To this end, let 
\[
N_z = a_1 f + a_2 f_z + a_3 f_{\bar{z}} + a_4 N\quad\text{ and }\quad f_{zz} = b_1 f + b_2 f_z + b_3 f_{\bar{z}} + b_4 N.
\]
Let $u = \langle N, f_{zz} \rangle$ and define $B=\langle f_z,f_{\bar z}\rangle$. Note that $B>0$ as $f$ is a conformal immersion.
Elementary calculations using the equations yield 
\[
-a_1 = \langle N_z, f \rangle = 0, \quad B a_3 = \langle N_z, f_z \rangle = -u, \quad B a_2 = \langle N_z, f_{\bar{z}} \rangle = 0, \quad a_4 = \langle N_z, N \rangle = 0,
\]
and similarly
\[
b_1 = \langle f_{zz}, f \rangle = 0, \quad B b_3 = \langle f_{zz}, f_z \rangle = 0, \quad B b_2 = \langle f_{zz}, f_{\bar{z}} \rangle = B_z, \quad b_4 = \langle f_{zz}, N \rangle = u.
\]
Therefore
\[
N_z = - \frac u B f_{\bar{z}}, \qquad f_{zz} = \frac {B_z} B f_z + u N.
\]
Note that $N_{\ol{z}} = \ol{N_z}$ and $f_{\ol{z} \ol{z}} = \ol{f_{zz}}$. Therefore
\begin{align*}
(N_z)_{\bar{z}} & = - \left(\frac u B\right)_{\bar{z}} f_{\bar{z}} - \frac u B f_{\bar{z} \bar{z}}  = - \left(\frac u B\right)_{\bar{z}} f_{\bar{z}} - \frac u B \frac{\overline{B_z}}{B} f_{\bar{z}} - \frac{|u|^2}{B} N.
\end{align*}
On the other hand,
\begin{align*}
(N_{\bar{z}})_z & = -\left( \frac {\overline{u}} B \right)_z f_z - \frac{\overline{u}}B f_{zz}  = -\left( \frac {\overline{u}} B \right)_z  f_z - \frac{\overline{u}}B \frac{B_z} B f_z - \frac{|u|^2}{B} N.
\end{align*}
Since $f, f_z, f_{\bar{z}}, N$ form a basis, the identity $(N_z)_{\bar{z}} = (N_{\bar{z}})_z$ implies that the coefficients of $f_z$ and $f_{\bar{z}}$ vanish. Therefore
\[
N_{z\bar{z}} = - \frac{|u|^2}B N
\]
and $N$ is harmonic.

Next, consider the case that $Q$ has only isolated zeros. Around any point where $Q$ does not vanish, there exist a holomorphic coordinate $z$, such that $Q = dz^2$. Moreover, we can arrange $\omega=dz$. In these coordinates, the defining equations above become
\[
\langle N, f \rangle = 0, \quad \langle N, f_z \rangle = 1, \quad \langle N, N \rangle = 1.
\]
We proceed as in the previous case. Here, we obtain the identities
\[
-a_1 = \langle N_z, f \rangle = -1, \qquad \qquad \quad a_2 + B a_3 + a_4 = \langle N_z, f_z \rangle = -u,
\]
\[
B a_2 + a_3 + a_4 = \langle N_z, f_{\bar{z}} \rangle = 0, \qquad a_2 + a_3 + a_4 = \langle N_z, N \rangle = 0
\]
and
\[
b_1 = \langle f_{zz}, f \rangle = 1, \qquad \qquad b_2 + B b_3 + b_4 = \langle f_{zz}, f_z \rangle = B_z,
\]
\[
B b_2 + b_3 + b_4 = \langle f_{zz}, f_{\bar{z}} \rangle = B_z, \qquad b_2 + b_3 + b_4 = \langle f_{zz}, N \rangle = u,
\]
where again $B = \langle f_z, f_{\bar z}\rangle$. Here, the equation $\langle f_z, f_z \rangle = 1$ implies $B > 1$. These equations can be solved to obtain
\[
N_z = f - \frac u {B-1} f_{\bar{z}} + \frac u {B-1} N,
\]
\[
f_{zz} = f + \frac{B_z - u}{B-1} f_z - \frac u {B-1} f_{\bar{z}} + \frac{(B+1) u - B_z}{B-1} N.
\]
Using $N_{\overline{z}} = \overline{N_z}$ and $f_{\overline{zz}} = \overline{f_{zz}}$, we obtain
\begin{align*}
(N_z)_{\bar{z}}  = \left( f - \frac{u}{B-1} f_{\bar{z}} + \frac{u}{B-1} N \right)_{\bar{z}}  &= \left( 1 - \left( \frac{u}{B-1}\right)_{\overline{z}}  - \frac{u}{B-1} \frac{\overline{B_z - u}}{B-1} \right) f_{\bar{z}} \\
& + \left( \left(\frac{u}{B-1}\right)_{\overline{z}} - \frac{u}{B-1} \frac{\overline{(B+1) u - B_z}}{B-1} - \frac{|u|^2}{B-1} \right) N,
\end{align*}
\begin{align*}
(N_{\bar{z}})_z  = \left( f - \frac{\overline{u}}{B-1} f_z + \frac{\overline{u}}{B-1} N\right)_z 
& = \left( 1 - \left( \frac{\overline{u}}{B-1} \right)_z - \frac{\overline{u}}{B-1} \frac{B_z - u}{B-1} \right) f_z \\
& + \left( -\frac{\overline{u}}{B-1} \frac{B_z+(B+1)u}{B-1} - \left(\frac{\overline{u}}{B-1}\right)_z + \frac{|u|^2}{(B-1)^2} \right) N.
\end{align*}
Since $f, f_z, f_{\bar{z}}, N$ form a basis, these formulas imply that the coefficients in front of $f_z$ respectively $f_{\bar{z}}$ in the two equations vanish:
\[
1 - \left( \frac{\overline{u}}{B-1} \right)_z - \frac{\overline{u}}{B-1} \frac{\overline{B_z} -\overline{ u}}{B-1} = 0.
\]
This identity can moreover be used to simplify the factor of $N$, and we obtain
\[
N_{\bar{z}z} = \left(1 - \frac{(B+1) |u|^2}{(B-1)^2}\right) N,
\]
proving that $N$ is harmonic in the coordinate domain.
This shows that $N$ is harmonic on the dense set $M_{\hyp} \cap \{ q\neq 0\}$. Since $N$ is smooth, this implies harmonicity of $N$ on $M$.
\end{proof}

The behavior in the gluing region of our analytical construction is locally modeled by the following examples.
\begin{exa}\label{exa:modelharmonicmap}
For $t > 0$ consider the map
$
f_{t,\hyp}^{\model} : \left\{ x+iy \in \C : x \neq 0 \right\} \to \H^3_\pm
$ given by
\[
f_{t,\hyp}^{\model}(x,y) =
\begin{psmallmatrix}
  \frac{1}{4} \, {\left(8 \, t^{2} y^{2} + \cosh\left(4 \, t x\right) + 1\right)} \operatorname{csch}\left(2 \, t x\right) \\
  0 \\
  \frac{1}{4} \, {\left(8 \, t^{2} y^{2} + \cosh\left(4 \, t x\right) - 3\right)} \operatorname{csch}\left(2 \, t x\right) \\
  2 \, t y \operatorname{csch}\left(2 \, t x\right)
\end{psmallmatrix}.
\]
Then, $f_t$ is harmonic, and the associated oblique hyperbolic Gau\ss\@ map is
\[
N_t^{\model}(x,y) =\tfrac{1}{{4 \, t \cosh \left(2 \, t x\right)}}
\begin{psmallmatrix}
  -8 \, t^{2} y^{2} - \cosh\left(4 \, t x\right) + 3\\
   2\sqrt{t^{2} \operatorname{csch}\left(2 \, t x\right)^{2} \sinh\left(4 \, t x\right)^{2} - 2 \, \cosh\left(4 \, t x\right) + 2}\\
  -8 \, t^{2} y^{2} - \cosh\left(4 \, t x\right) - 1
  \\
  -8 \,t \, y\end{psmallmatrix}.
\]
Note in particular that $f^{\model}_{t,\hyp}$ changes sign across $x=0$, so it cannot be smoothly continued there as a map into $\H^3$, whereas $N^{\model}_t$ is even and extends smoothly across $x=0$. This mirrors what will happen for our glued solutions: the hyperbolic map will develop a 'fold', but the de Sitter Gau\ss\@ map stays smooth. Let us also observe that $f_{t,\hyp}^{\model}$ is odd in the variable $x$ and $N_t^{\model}$ is even. In particular, at $x=0$ the map $N_t^{\model}$ fails to be an immersion.

The derivation of the maps $f_{t,\hyp}^{\model}$ and $N_t^{\model}$
will be explained in the next subsection.  
\end{exa}

\subsection{The dual map}
The construction of the oblique hyperbolic Gau\ss\@ map associated to a harmonic map can be reversed. This is captured by the following definition.

\begin{defn}\label{defn:dual_map}
  Let $M$ be a Riemann surface and $N : M \to  \dS_3$ be an immersion. By $Q \in \Gamma(K^2)$ we denote the {\em de Sitter Hopf differential} $Q = \left(N^*\langle \cdot \, , \cdot \rangle\right)^{2,0} = \langle \partial N, \partial N \rangle$.
  Assume that there exists $\omega \in \Omega^{1,0}(M)$ such that $-\omega^2 = Q$.

  A map $f : M \to \H^3_\pm$ is a {\em dual map} of $N$ (with respect to $\omega$), if
  \begin{enumerate}
  \item $f(x) \in T_{N(x)}  \dS_3$ for every $x \in M$,
  \item $\langle dN, f \rangle = - \left(\omega + \ol{\omega}\right)$,
  \item $\langle f, f \rangle = -1$,
  \item $f$ is compatible with the orientation of  $M,$ i.e., $\det(N,f,N_x,N_y)>0.$
  \end{enumerate}
\end{defn}
The sign convention in the equation for $\langle dN, f \rangle$ is chosen to ensure that the dual map associated to an oblique hyperbolic Gau\ss\@ map is the original map, see Proposition \ref{prop:dual_involution} below. 
We will refer to $f$ as the dual of $N$, and to $N$ as the (oblique) Gau\ss\@ map of $f$.

\begin{prop}\label{prop:Ntof}
Let $M$ be a connected Riemann surface. For any immersion $N : M \to  \dS_3$ and square root $\omega \in \Omega^{1,0}(M)$ of the de Sitter Hopf differential, there exists a unique dual map $f : M \to \H^3_\pm$. 
\end{prop}
\begin{proof}
As in the proof of Proposition \ref{prop:hyperbolic_gauss_map_harmonic}, we must verify that for every $x \in M$, the set of defining equations in Definition \ref{defn:dual_map} possesses exactly one solution.
Once again, the initial condition restricts $f(x)$ to lie within a $3$-dimensional linear subspace, while the second equation further confines $f$ to lie in an affine line within that subspace. Consequently, the second-to-last equation defines a quadratic equation on that line.
The key distinction lies in the signature of $T_{N(x)}  \dS_3$, which is $(1,2)$ instead of $(0,3)$. Consequently, we must conduct a case-by-case analysis based on the signature of $\im d_xN$.

Assume initially that $\im d_xN$ has signature $(0,2)$. In this case, we find a $\nu \in T_{N(x)}  \dS_3$ that is orthogonal to $\im d_xN$ and satisfies $\langle \nu, \nu \rangle = -1$.
Let $f_0 \in \im d_xN$ be the unique vector that satisfies $\langle d_xN, f_0 \rangle = \eta(x)$, where $\eta = - \left( \omega + \ol{\omega} \right)$.
Now, let us consider the ansatz $f(x) = f_0 + t \nu$. By substituting this into the equation $\langle f(x), f(x) \rangle = -1$, we obtain $\langle f_0, f_0 \rangle - t^2 = -1$. Since $\langle f_0, f_0 \rangle \geq 0$, it is evident that this equation has two solutions. As before, fixing the orientation is equivalent to choosing exactly one of these two solutions.

Next, assume that $\im d_xN$ has signature $(1,1)$. On $T_x M$ we introduce the indefinite metric $g = \langle d_xN \cdot, d_xN \cdot \rangle$. 
Assume that we are away from the zeros of $Q$ and pick some $X \in T_x M$ such that $\omega(X) = r \in \R_{>0}$.
 Then we compute
\[
-r^2=Q(X,X) = \frac 1 4 \left( g(X,X) - g(J X, JX) \right) - \frac i 2 \langle X, J X \rangle
\]
Denote $a = g(X,X)$ and $b = g(JX, JX)$. The real part of the equation then becomes $a - b = -4r^2$, while the imaginary part yields $g(X, JX) = 0$. The orthogonality and the signature assumption imply that either $a < 0$ and $b > 0$ or $a > 0$ and $b < 0$. However, the equation $a = b - 4r^2 < b$ implies that only the first case can actually occur. Furthermore, $b = a + 4r^2 < 4r^2$.

Let us make the ansatz
\[
f(x) = \alpha d_xN X + \beta d_xN JX + \gamma \nu
\]
for some $\nu \in T_{N(x)}  \dS_3$ with $\nu \perp \im d_xN$ and $\langle \nu, \nu \rangle = 1$ and $\alpha,\beta,\gamma \in \R$.
The equation $\langle dN, f \rangle = -(\omega + \ol{\omega})$ then yields $f(x) = -\frac {2 r} a d_xN X + \gamma \nu$ for some $\gamma\in\R$. To solve
\[
-1 = \langle f(x), f(x) \rangle = \tfrac{4 r^2}{a^2} a + \gamma^2,
\]
we need to ensure that $\frac{4 r^2}{a} \leq -1$. Since $a = b - 4 r^2$ we may rewrite
\[
\tfrac{4 r^2}{a} = - \tfrac{1}{1 - \tfrac b {4 r^2}}<-1
\]
since $0 < b < 4r^2$. Therefore, the equation $\gamma^2 = -1 -4 \frac {r^2}a$ has exactly two solutions, where one of them is compatible with the orientation.

The last possibility is that the metric on $\im d_xN$ is degenerate. Since there are no negative semidefinite subspaces in Minkowski space, the metric must be positive semidefinite.
Let $X$ be such that $\omega(X) = r > 0$. Using the same argument as before, we have that $g(X,X) = a$ and $g(JX,JX) = b$ satisfy $a = b - 4r^2$ and $g(X,JX) = 0$. The degeneracy implies that either $a = 0$ or $b = 0$. However, since the metric is positive semidefinite, this implies that $a = 0$ and $b = 4r^2$.

Consider the subspace $V = \{ v \in T_x  \dS_3 : \langle v, d_xN JX \rangle = 0 \}$. This is a $2$-dimensional space of signature $(1,1)$. We can find a unique second light-like vector $\nu \in V$ with $\langle \nu, X \rangle = 1$. Now make the ansatz $f(x) = \alpha d_xN X + \beta \nu$. We want to solve
\[
\langle d_xN X, f(x) \rangle = - 2 r
\]
and $\langle f(x), f(x) \rangle = -1$. The first equation becomes $\beta = - 2r$ and the second equation becomes $2 \alpha \beta = -1$. Therefore, we obtain
$
f(x) = \frac 1 {4r} d_xN X - 2 r \nu.
$
One can verify that the unique solution is compatible with the orientation of $M.$

We deal with zeros of $Q$ and smoothness (in the case of harmonic $N$) in Section \ref{sec:twist-construction} below.
\end{proof}

\begin{prop}
  \label{prop:dual_map_harmonic}
If $N : M \to  \dS_3$ is a harmonic immersion with dual map $f : M \to \H^3_\pm$, then $f$ is harmonic.
\end{prop}
\begin{proof}
This computation is analogous to the one used in the proof of Proposition \ref{prop:hyperbolic_gauss_map_harmonic}.
\end{proof}

\begin{prop}
\label{prop:dual_involution}
Suppose $f: M \to \H^3_\pm$ is a harmonic immersion, $\omega \in \Omega^{1,0}(M)$ is a square root of the hyperbolic Hopf differential $Q = (f^* \langle \cdot \, , \cdot \rangle)^{2,0}$, and $N : M \to  \dS_3$ is the associated oblique hyperbolic Gau\ss\@ map. Then $f$ is the dual map associated to $N$.
\end{prop}
\begin{proof}
By definition $N$ satisfies
\[
\langle N, f \rangle = 0, \qquad \langle N, df \rangle = \omega + \ol{\omega}, \qquad \langle N, N \rangle = 1.
\]
This implies
\[
\langle f, N \rangle = 0, \qquad \langle f, dN \rangle = -\left(\omega + \ol{\omega}\right)
\]
On the other hand $\langle f, f\rangle = -1$, since $f$ maps to $\H^3_\pm$. Therefore, $f$ satisfies the dual map equations for $\omega$. We still need to check that $\omega^2 = -Q = (N^* \langle \cdot \, , \cdot \rangle)^{2,0}$.

Since $Q$ is holomorphic, in a dense set on $M$ we may assume $Q = dz^2$ or $Q = 0$ in coordinates. These cases correspond to $\langle f_z, f_z \rangle = 1$ and $\langle f_z, f_z \rangle = 0$ respectively. In Proposition \ref{prop:hyperbolic_gauss_map_harmonic} we have shown that if $Q = 0$, i.e.\@ $\langle f_z, f_z \rangle =0$, then $N_z = - \frac u B f_{\bar{z}}$. This immediately implies $\langle N_z, N_z\rangle = 0$.

On the other hand, if $\langle f_z, f_z \rangle = 1$, then we have shown that $N_z = f - \frac u {B-1} f_{\bar{z}} + \frac u {B-1} N$. In this case, using additionally the relations $\langle f, N \rangle = 0$, $\langle f_{\bar{z}}, N \rangle = 1$, $\langle N, N \rangle = 1$, we find
\[
\langle N_z, N_z \rangle = \langle f, f \rangle = - 1.
\]
This shows that in these coordinates $(N^* \langle \cdot \,, \cdot \rangle)^{2,0} = -dz^2 = -Q$. That the orientations match can be shown by a direct calculation, or it follows from Section \ref{sec:twist-construction} below.
 \end{proof}
In particular, the constructions {\em Gau\ss\@ map} and {\em dual } are inverse to each other (once a square root of the Hopf differential is fixed).

\subsection{The conformally invariant Hopf differential}
We now recast the above constructions in a conformally invariant way. Instead of working in $\H^3_\pm \subset \R^{1,3}$, we pass to the projectivized light cone $\bb{P}\mc{L}\cong \mathbb S^3$, which compactifies hyperbolic space by adding the boundary 2-sphere $\mathbb S^2_{\mathrm{eq}}$.

As we observed in Section \ref{sect:geom_prelims}, $\H^3_\pm$ can be naturally regarded as a subspace of $\bb{P}\mc{L} \cong \mathbb S^3$. For maps into $\mathbb S^3$, we can define oblique Gau\ss\@ maps in a manner similar to that in Definition \ref{defn:hyperbolic-oblique-gauss-map-1}, see also Remark \ref{rem:Riemtarget}. These two notions are closely related, as will be evident in the subsequent discussion. For the remainder of this section, we once again consider $\H^3_\pm$ as the slice of $\mc{L}$ by the affine hyperplane $x_4 = 1$.

To establish a connection between these two different oblique Gau\ss\@ maps, it is advantageous to adopt a uniform definition of the Hopf differential that applies to both cases. In the following, $L$ denotes the tautological line bundle induced by $\mc{L} \to \bb{P}\mc{L}$.

\begin{defn}
Let $M$ be a Riemann surface. For a smooth map $f : M \to \bb{P}\mc{L}$ the {\em conformally invariant Hopf differential} is the section $Q$ of $K^2 \otimes f^*L^{-2}$, which is locally defined by
\begin{equation}\label{eq:hopf-scale.in}
Q|_U = \langle \partial \hat f, \partial \hat f \rangle \otimes \hat f^{-2},
\end{equation}
where $\hat f$ is a local lift of $f$, i.e.\@ $f : U \to \mc{L}$ satisfies $\pi \circ \hat f = f$.
\end{defn}
When $f$ takes values in $\H^3_\pm\subset \bb{P}\mc{L}$, this conformally invariant $Q$ reduces to the standard Hopf differential of the hyperbolic lift $f_{\hyp}$ through the natural trivialization of $f^*L^{-2}$.

Note that a local lift of $f$ is exactly the same thing as a local section of $f^*L$ without zeros. To see that this definition is consistent, let $\check{f} : U \to \mc{L}$ be another lift of $f$. Then, there exists a nowhere vanishing function $\lambda: U \to \R$ such that $\check{f} = \lambda \hat{f}$. Given this, 
the computation
\begin{align*}
  \langle \partial\check{f}, \partial \check{f} \rangle \otimes \check{f}^{-2} & = \langle \partial (\lambda \hat f), \partial (\lambda \hat f) \rangle (\lambda \check{f})^{-2} 
   = \lambda^2 \langle \partial \hat f, \partial \hat f \rangle \lambda^{-2} \hat f^{-2}  = Q|_U
\end{align*}
shows that the definition of $Q$ does not depend on the choice of the lift. The computation used the fact that $\langle \hat{f}, \hat{f} \rangle = 0$ and its consequence $\langle \partial \hat{f}, \hat{f} \rangle = 0$.

For any map $f : M \to \bb{P}\mc{L}$ we denote by $f_{\sph} : M \to \mathbb S^3$ the {\em spherical lift} of the map, i.e.,\@ $f_{\sph} = \sigma_{\mathbb S^3} \circ f$. Similarly, by $f_{\hyp} : M_{\hyp} \to \H^3_\pm$ we denote the {\em hyperbolic lift} of the 
map, i.e.,\@ $f_{\hyp} = \sigma_{\H^3_\pm} \circ f|_{M_{\hyp}}$ where $M_{\hyp} = f^{-1} ( \dom \sigma_{\H^3_\pm} )$. 

\begin{defn}
Let $f : M \to \bb{P}\mc{L}$ be an immersion and $\omega \in \Gamma( K \otimes f^*L^{-1})$ a square root of the conformally invariant Hopf differential, i.e.\@ $\omega^2 = Q$. Moreover, let $\eta = \omega + \ol{\omega} \in \Omega^1(M, f^*L^{-1})$.

The {\em spherical oblique Gau\ss\@ map} is defined to be a smooth map $N_{\sph} : M \to T\mathbb S^3$ satisfying $N_{\sph}(x) \in T_{f_{\sph}(x)} \mathbb S^3$,
\[
\langle N_{\sph}(x), d_xf_{\sph} \rangle = \eta(x) \otimes f_{\sph}(x),
\]
 $\|N_{\sph}(x)\| = 1$ for every $x \in M$, and which is compatible with the orientation on $M$.

Similarly, the {\em hyperbolic oblique Gau\ss\@ map} is defined to be a smooth map $N_{\hyp} : M_{\hyp} \to T\H^3_\pm$ satisfying $N_{\hyp}(x) \in T_{f_{\hyp}(x)} \H^3_\pm$ and
\[
\langle N_{\hyp}(x), d_xf_{\hyp} \rangle = \eta(x) \otimes f_{\hyp}(x),
\]
 $\|N_{\hyp}(x)\| = 1$ for every $x \in M_{\hyp}$, and which is compatible with the orientation on $M$.
\end{defn}
The definition of the hyperbolic oblique Gau\ss\@ map here coincides with the previous definition after identifying $N_{\hyp}(x) \in T_{f_{\hyp}(x)} \H^3_\pm \subset \{ x \in \R^{1,4} : x_4 = 0\}$ and $N_{\hyp}(x) \in \R^{1,3}$. Observe also that if $\omega \in \Gamma(K \otimes f^*L^{-1})$ is a square root of the conformally invariant Hopf differential $Q$, then $\omega \otimes f_{\hyp} \in \Omega^{1,0}(M)$ is a square root of the hyperbolic Hopf differential $Q_{\hyp}= \langle \partial f_{\hyp}, \partial f_{\hyp}\rangle $. Similarly, $\omega \otimes f_{\sph} \in \Omega^{1,0}(M)$ is a square root of the spherical Hopf differential $Q_{\sph}=\langle \partial f_{\sph}, \partial f_{\sph}\rangle $. Analogous to Proposition \ref{pro:existN}, we have:
\begin{prop}
  For an immersion $f : M \to \bb{P}\mc{L}$ of a connected Riemann surface and a square root of the conformally invariant Hopf differential $\omega \in \Gamma( K \otimes f^*L^{-1})$, there exists a unique oblique spherical Gau\ss\@ map.
\end{prop}

For the oblique spherical Gau\ss\@ map $N_{\sph}$ of $f\colon M\to \H^3_\pm\subset\mathbb S^3$, one obtains the oblique hyperbolic Gau\ss\@ map via
\[
N_{\hyp}(x) = \frac{d_{f_{\sph}(x)} \Xi\,  N_{\sph}(x)}{\|d_{f_{\sph}(x)} \Xi\,  N_{\sph}(x)\|},
\]
where $\Xi$ is the conformal diffeomorphism between $\H^3_\pm$ and $\mathbb S^3\bs \mathbb S^2_{\operatorname{eq}}$ introduced in Section \ref{sect:geom_prelims}. To see that $N_{\hyp}$ is an oblique hyperbolic Gau\ss\@ map, first note that $\Xi$ has conformal factor $\tfrac 1 {x_4^2}$. Then, we compute
\begin{align*}
  \langle N_{\hyp}(x), d_xf_{\hyp} \rangle & = \left\langle  \frac{d_{f_{\sph}(x)} \Xi\,N_{\sph}(x)}{\|d_{f_{\sph}(x)} \Xi\,N_{\sph}(x)\|}, d_{f_{\sph}(x)} \Xi\,df_{\sph}(x) \right\rangle \\
  & = \frac 1 {f_{\sph}(x)_4^2} \frac 1 {\|d_{f_{\sph}(x)} \Xi\,N_s(x)\|}  \langle N_{\sph}(x), df_{\sph}(x) \rangle \\
  & = \frac 1 {f_{\sph}(x)_4} \eta(x) \otimes f_{\sph}(x)  \,= \,\omega(x) \otimes f_{\hyp}(x),
\end{align*}
where we used that $f_{\hyp}(x) = \Xi(f_{\sph}(x)) =  \frac{f_{\sph}(x)}{f_{\sph}(x)_4}$ for $f_{\sph}(x)=({f_{\sph}(x)}_0,\dots,{f_{\sph}(x)}_4)$ and 
\[
\|d_{f_{\sph}(x)} \Xi\,N_{\sph}(x)\| = \frac 1 {f_{\sph}(x)_4} \|N_{\sph}(x)\| = \frac 1 {f_{\sph}(x)_4}.
\]
In fact, using $d_p\Xi\, v = \frac 1 {p_4} v - \frac{v_4}{x_4^2} p$, we may further compute
\[
d_{f_{\sph}(x)} \Xi\,N_{\sph}(x) = \frac 1 {f_{\sph}(x)_4} N_{\sph}(x) - \frac {N_{\sph}(x)_4}{f_{\sph}(x)_4^2} f_{\sph}(x)
\]
and therefore
\[
N_{\hyp}(x) = N_{\sph}(x) - \frac {N_{\sph}(x)_4}{f_{\sph}(x)_4} f_{\sph}(x).
\]

\subsection{Transgressive harmonic maps}
Under certain conditions on $f$, it turns out that the map $N_{\hyp}$, considered as a map into $ \dS_3$, extends smoothly from $f_{\sph}^{-1}(\mathbb S^3 \bs \mathbb S^2_{\operatorname{eq}})$ to $M$. The transgressive maps, which are the central focus of our study, are a particular class of such maps.

\begin{defn}
  \label{defn:transgressive}
A smooth map $f : M \to \bb{P}\mc{L}$ is {\em transgressive}, if
\begin{enumerate}
\item the conformally invariant Hopf differential vanishes along $f_{\sph}^{-1}(\mathbb S^2_{\operatorname{eq}})$,
\item $f_{\sph}$ intersects $\mathbb S^2_{\operatorname{eq}}$ orthogonally, that is, for every $x \in f_{\sph}^{-1}(\mathbb S^2_{\operatorname{eq}})$ there exists a non-zero $v \in T_x M$ with $d_xf\,v \perp T_{f_{\sph}(x)} \mathbb S^2_{\operatorname{eq}}$.
\end{enumerate}
The first property will also be called {\em conformality at infinity}. Note that a transgressive map $f$ is an immersion in a neighbourhood of
$\Gamma = f_{\sph}^{-1}(\mathbb S^2_{\operatorname{eq}})$, and $\Gamma \subset M$ is a $1$-dimensional submanifold.
\end{defn}
As we will see later, the harmonic maps $f_{t,\hyp}^{\model}$ produce examples of such maps. Another class of interesting examples are given by minimal surfaces in $\H^3$ which pass through infinity orthogonally. Crucially, the hyperbolic oblique Gau\ss\@ map of a transgressive harmonic map (restricted to $f^{-1}(\H^3_\pm)$) extends to a smooth map across $\Gamma$. This is the content of the next theorem.

\begin{thm}\label{thm:transobl}
Let $f : M \to \bb{P}\mc{L}$ be a transgressive immersion and $\omega \in \Gamma(K \otimes f^* L^{-1})$ a square root of the conformally invariant Hopf differential. Let $N_{\sph} : M \to T\mathbb S^3$ be a spherical oblique Gau\ss\@ map associated to $f$ and $\omega$. Then, the associated hyperbolic oblique Gau\ss\@ map extends to a smooth map $N : M \to \dS_3$.

The set of points $\Gamma = f_{\sph}^{-1}(\mathbb S^2_{\operatorname{eq}})$, where $f_{\sph}$ intersects the equatorial 2-sphere, is contained in the set of points $\{ x \in M : \rk d_xN < 2\}$, where $N$ is not an immersion.
\end{thm}
In particular, $N$ is harmonic on all of $M$, is spacelike away from $\Gamma$, and its differential drops rank  along $\Gamma$.
\begin{proof}
Since $N_{\hyp}(x) = N_{\sph}(x) - \frac {N_{\sph}(x)_4}{f_{\sph}(x)_4} f_{\sph}(x)$ away from $\Gamma$, it suffices to show that $\frac{N_{\sph}(x)_4}{f_{\sph}(x)_4}$ extends as a smooth function.

The condition that $f_{\sph}$ intersects $\mathbb S^2_{\operatorname{eq}} = \{ x \in \mathbb S^3 : x_4 = 0 \}$ orthogonally implies that $f_{\sph,4}$ has a 
first-order zero at every $x \in f_{\sph}^{-1}(\mathbb S^2_{\operatorname{eq}})$.
On the other hand, the fact that the conformally invariant Hopf differential vanishes at every $x \in f_{\sph}^{-1}(\mathbb S^2_{\operatorname{eq}})$ implies that for each such $x$ the equations for $N_{\sph}(x)$ become
\[
N_{\sph}(x) \in T_{f_{\sph}(x)} \mathbb S^3, \qquad \langle N_{\sph}(x), df \rangle = 0, \qquad \|N_{\sph}(x)\| = 1.
\]
Again, by the orthogonality condition, there exists a $v \in T_x M$ with $d_xf v \perp T_{f_{\sph}(x)} \mathbb S^2_{\operatorname{eq}}$. This implies that $N_{\sph}(x) \in T_{f_{\sph}(x)} \mathbb S^2_{\operatorname{eq}}$ or equivalently that $N_{\sph}(x)_4$ vanishes. Since $N_{\sph}(x)_4$ is smooth, this implies that $N_{\sph}(x)_4$ vanishes at least to first order. Therefore, the quotient $\frac{N_{\sph}(x)_4}{f_{\sph}(x)_4}$ is smooth.

The second part follows directly from Proposition \ref{prop:dual_involution} together with the existence of a dual map of an immersion $N$
provided by Proposition \ref{prop:Ntof}.
\end{proof}

\begin{rem}\label{rem:ortho}
The proof of Theorem \ref{thm:transobl} shows that $N$ does not extend smoothly if $f$ is only conformal at infinity but does not intersect 
$\mathbb S^2_{\operatorname{eq}}$ orthogonally.
\end{rem}

The previous theorem can be partially reversed. We start with some preliminary discussion: Assume that $N$ is smooth, and that there is
a 1-dimensional submanifold $\Gamma\subset M$ such that $N_{\mid M\setminus\Gamma}$ is an immersion and $\mathrm{rank}\,dN\leq1 $ along $\Gamma.$
Further, we assume that for local coordinates $(x_1,x_2)\colon U\subset M\to\R^2$ there is a $\Gamma$-defining function $d\colon U\to \R$ without critical points, i.e., $\Gamma\cap U=d^{-1}(\{0\})$, such that 
\begin{equation}\label{eq:detsingcond}d^2=\pm\det(\begin{pmatrix} \langle \tfrac{\partial N}{\partial x_i}, \tfrac{\partial N}{\partial x_j}\rangle\end{pmatrix}_{i,j}).\end{equation}
Then, the rank of $dN$ is exactly 1 along $\Gamma.$ Furthermore, there is a smooth rank 2 subbundle $E\leq N^* T \dS_3$
with $\mathrm{im} (dN)\subset E:$ in fact, there always exists a local non-vanishing vector field $X$ with $d_pN(X)=0$ for all $p\in U\cap \Gamma.$ Then, $dN(\tfrac{1}{d}X)$ extends smoothly across $\Gamma\cap U$, and spans together with $dN(Y)$ -- for a pointwise linearly independent vector field $Y$ -- the rank 2 bundle $E$ because of \eqref{eq:detsingcond}. Moreover, it follows that the induced signature of $E$ is locally constant.

For short, we say in the above situation that the rank of $dN$ drops {\em transversally without signature change}.
Geometrically, the hypothesis says that $N$ is harmonic into $\dS_3$, is an immersion off a curve $\Gamma$, and, along $\Gamma$, 'folds' in a controlled, non-degenerate way,  just as the model map $N^{\model}_t$ does (rank exactly $1$, no change of causal type in the image).
\begin{thm}
  \label{thm:transgressive_map_from_de_sitter_map}
Let $N\colon M\to  \dS_3$ be a harmonic map with square Hopf differential $Q=-\omega^2.$ Assume that there is
a 1-dimensional submanifold $\Gamma\subset M$ such that $\Gamma$ does not contain any zeros of $Q$ and 
the rank of $N$ drops to 1 transversally without signature change along $\Gamma.$ Then, there is a transgressive harmonic map $f\colon M\to \bb{P}\mc{L}$ with
$f_{\sph}^{-1}(\mathbb S^2_{\operatorname{eq}})\subset\Gamma$.
\end{thm}
\begin{proof}
We follow the construction in Proposition \ref{prop:Ntof}. First, assume that $E$ has signature $(0,2)$.
Let $z=x+iy$ be a holomorphic coordinate with $Q=-dz^2,$
 and $\nu$ be the positively oriented vector field of constant length $-1$ perpendicular to $E$. The metric induced by $N$ is
 \[g= a(dx^2)+ (a+2)(dy)^2\]
 for some function $a\colon U\to\R.$ Thus, $a(a+2)=\pm d^2,$  where $d$ is as in \eqref{eq:detsingcond}. Since $E$ is by assumption of type $(0,2)$, the metric coefficient
  $a$ vanishes along $\Gamma$ and $a+2$ does not. Note that  $f=-\tfrac{2}{a}dN(\tfrac{\partial}{\partial x})+t\nu$  holds  away from $\Gamma,$ where $t$ is the positive solution of $t^2=\tfrac{4}{a}+1$. Then
 \begin{equation}\label{eq:boundaryfdual}p\in U\mapsto d(p)\,( -\tfrac{2}{a(p)}d_pN(\tfrac{\partial}{\partial x})+t(p) \nu(p),1)\,\in\mathcal L\setminus\{0\}\end{equation}
 extends smoothly through $\Gamma$. Hence, the spherical lift of $f$ extends smoothly through $\Gamma$, and the conformally invariant Hopf differential vanishes along $\Gamma$ by its definition \eqref{eq:hopf-scale.in} and because the Hopf differential of $f$ is $Q$ as a consequence of
 the proof of Proposition \ref{prop:dual_involution}. By Remark \ref{rem:ortho}, the intersection of $f$ with $\mathbb S^2_{\operatorname{eq}}$ along $\Gamma$ is orthogonal.
 
Now assume we are in the second case, such that $E$ has type $(1,1).$ Let $\nu$ be the positively oriented vector field of constant length $1$ perpendicular to $E$. If $a$ vanishes along $\Gamma$, we can proceed as in the first case: away from $\Gamma$,
$f=-\tfrac{2}{a}dN(\tfrac{\partial}{\partial x})+t\nu$
where $t$ is the positive solution of
$-1=\tfrac{4}{a}+t^2$. For $d^2=-a(a+2)$, \eqref{eq:boundaryfdual} extends smoothly through $\Gamma$. As in the first case, the map $f$ is transgressive.

Finally, if $a+2$ vanishes along $\Gamma$, then  $f=-\frac{2}{a}dN(\frac{\partial}{\partial x})+t\nu$ away from $\Gamma$, and $f$ extends smoothly across $\Gamma$ as a map to hyperbolic 3-space.
\end{proof}

\begin{exa}\label{exa:modelharmonicmap_sphere}
  The harmonic maps $f_{t,\hyp}^{\model}$ from Example \ref{exa:modelharmonicmap} induce transgressive maps
  \[
  f_{t,\sph}^{\model} : \{x + iy \in \C : x\neq 0\} \to \bb{S}^3 \subset \R^{1,4}
  \]
  \[
  f_{t,\sph}^{\model}(x,y) =
  \begin{pmatrix}
    1 \\
    0 \\
    \frac{8 \, t^{2} y^{2} + \cosh\left(4 \, t x\right) - 4}{4 \, {\left(8 \, t^{2} y^{2} + \cosh\left(4 \, t x\right) + 1\right)}} \\
    \frac{2 \, t y}{8 \, t^{2} y^{2} + \cosh\left(4 \, t x\right) + 1} \\
    \frac{\sinh\left(2 \, t x\right)}{8 \, t^{2} y^{2} + \cosh\left(4 \, t x\right) + 1}
  \end{pmatrix}
  \]
  These maps are evidently smooth through $x=0$, and it can be checked that they are transgressive harmonic maps with dual maps $N_t^{\model}$.
\end{exa}

\subsection{The twist construction}\label{sec:twist-construction}
We next describe the oblique hyperbolic Gau{\ss} map in terms of solutions to the $\mathrm{SU}(2)$ self-duality equations. Consider a solution given
by the unitary connection $\nabla=d+A$ (with respect to the Hermitian metric $h$) and the Higgs field $\Phi$ on $M$, i.e.,
\[\bar\partial^\nabla\Phi=0\quad\text{and} \quad F^\nabla=-[\Phi,\Phi^{*_h}].\]
Assume that $\det\Phi=-\omega^2.$ This condition means that the Hopf differential of the metric induced by the associated
harmonic map into $\mathbb H^3$ is $q=\omega^2.$
In particular, $\det\Phi$ is a square if and only if the hyperbolic Hopf differential admits a square root~$\omega$. Now consider the holomorphic 
eigenline bundle $L$ of $\Phi$ with respect to $\omega,$ and split $V=L\oplus L^\perp$ with respect to the Hermitian metric $h.$ 
Correspondingly, 
\begin{equation}\label{eq:SDSsplit}
\begin{split}
\nabla&=\begin{pmatrix} \nabla^L & \gamma\\ -\gamma^* & \nabla^{L^*}\end{pmatrix},\quad
\Phi=\begin{pmatrix} \omega & \alpha\\0&-\omega\end{pmatrix},\quad
\Phi^*\,=\,\begin{pmatrix} \bar\omega & 0\\\alpha^*&\bar\omega\end{pmatrix},\\
\end{split}
\end{equation}
where $\nabla^L$ and $\nabla^{L^*}$ are dual Hermitian line bundle connections, $\gamma\in\Omega^{(0,1)}(M,L^2)$ and
$\alpha\in\Omega^{(1,0)}(M,L^{-2})$, and $\gamma^*$ and $\alpha^*$ are the Hermitian adjoints. Consider the associated
equivariant harmonic map $f$ to $\H^3$. The induced metric is 
\[g=\omega^2+\omega\bar\omega+\bar\omega\omega+\tfrac{1}{2}(\alpha\alpha^*+\alpha^*\alpha)+\bar\omega^2.\]
The map $f$ can be computed as follows: take a local special Hermitian trivialization of $V=L\oplus L^\perp$, i.e., write \eqref{eq:SDSsplit}
with respect to a determinant-$1$ Hermitian frame of $L$ and $L^\perp:$
\[\nabla=d+\omega_0,\quad \Phi=\omega_{-1},\quad \Phi^*=\omega_1.\]
Let $F\colon U\to\mathrm{SL}(2,\C)$ be a local parallel frame for the flat connection $\nabla+\Phi+\Phi^*,$ i.e., a solution
of the ODE
\[dF+(\omega_0+\omega_{-1}+\omega_1)F=0.\]
Then, the harmonic map is given by
\[f=\bar F^TF\]
(up to the $\mathrm{SL}(2,\C)$-action on $\mathbb H^3$). Moreover, the oblique Gau\ss\@ map is then given by
\begin{equation}\label{eq:gaugeobliqueN}N=\bar F^T\begin{psmallmatrix}1&0\\0&-1\end{psmallmatrix}F.\end{equation}
In fact, one can directly check that $N$ satisfies the required properties (1)--(4) in Definition \ref{defn:hyperbolic-oblique-gauss-map-1}. Note that, away from zeros of $\omega$ and using
$a:=\frac{\alpha}{\omega},$ the oppositely oriented oblique Gau\ss\@ map is
\[\widetilde N=\bar F^T\left(
\begin{array}{cc}
 \frac{4-a^2}{a^2+4} & \frac{4 a}{a^2+4} \\
 \frac{4 a}{a^2+4} & \frac{a^2-4}{a^2+4} \\
\end{array}
\right)F.\]
The induced metric of $N$ is
\[\hat g=-\omega^2-\omega\bar\omega+\tfrac{1}{2}(\gamma\bar\gamma+\bar\gamma\gamma)-\bar\omega\omega-\bar\omega^2,\]
in accordance with Proposition \ref{prop:dual_involution}.
\begin{rem}
 By \cite{Donaldson}, the Hermitian metric $h$
on $V$ which solves the self-duality equation is given by an equivariant harmonic map into the space of Hermitian metrics of determinant~$1$, with respect to a parallel frame of the flat connection $\nabla+\Phi+\Phi^*$. The gauge theoretic meaning of \eqref{eq:gaugeobliqueN} is the following: $$\hat h:=h_{\mid L}\oplus -h_{\mid L^\perp}$$ is a Hermitian metric of signature $(1,1)$ on $V=L\oplus L^\perp.$ With respect
to a parallel frame of the flat connection $\nabla+\Phi+\Phi^*$, $\hat h=N$ is a harmonic map into the de Sitter 3-space of Hermitian metrics of determinant $-1$ and signature $(1,1).$
\end{rem}
Define
\begin{equation}\label{eq:SDSsplittwist}
\begin{split}
\hat\nabla&=\begin{pmatrix} \nabla^L & \alpha\\ \alpha^* & \nabla^{L^*}\end{pmatrix},\quad
\hat\Phi=\begin{pmatrix} \omega & 0\\-\gamma^*&-\omega\end{pmatrix},\quad
\hat\Phi^\dagger\,=\,\begin{pmatrix} \bar\omega & \gamma\\0&-\bar\omega\end{pmatrix}.\\
\end{split}
\end{equation}
We observe that $\hat\nabla$ is unitary, and $\hat\Phi^\dagger$ is the adjoint of $\hat\Phi$, both with respect
to the indefinite Hermitian metric $\hat h.$ Furthermore, $L^\perp$ is an eigenline bundle of $\hat\Phi^*$ with respect to $-\omega.$
We summarize our observations as follows:
\begin{thm}\label{thm:SU2SU11}
Let $(\nabla,\Phi,h)$ be a solution of the $\mathrm{SU}(2)$ self-duality equations such that $\det\Phi=-\omega^2$. Then, 
$(\hat \nabla,\hat\Phi,\hat h)$ is a solution of the $\mathrm{SU}(1,1)$ self-duality equations.

If $f$ is the equivariant harmonic map into hyperbolic 3-space associated to $(\nabla,\Phi,h)$, then the associated equivariant harmonic map
into de Sitter 3-space corresponding to $(\hat\nabla,\hat \Phi,\hat h)$ is the oblique hyperbolic Gau\ss\@ map $N$ of $f.$
\end{thm}
\begin{rem}
In contrast to Proposition \ref{pro:existN}, Theorem \ref{thm:SU2SU11} does not need the assumption that $f$ is an immersion.
\end{rem}

We now reverse the above construction. That is, we begin with a solution $(\hat \nabla,\hat\Phi,\hat h)$ of the $\mathrm{SU}(1,1)$ self-duality 
equations. We assume that $\det\hat\Phi=-\omega^2,$ and consider the eigenline bundle $\hat L$ of $\hat\Phi$ with respect to $\omega$, and its $\hat h$-orthogonal complement line bundle $L^\dagger.$ Initially, we assume that $\hat L$ and $\hat L^\dagger$ are complementary. Since $\hat h$ is of signature $(1,1)$, $\hat h$ restricted to $\hat L$ and $\hat L^\dagger$ is either positive definite and negative definite, respectively, or vice versa.  In any case, we can define a positive definite Hermitian metric
\[h:=\pm(\hat h_{\mid \hat L}\oplus -\hat h_{\mid \hat L^\perp}).\]
Then, writing $\hat\nabla, \hat\Phi, \hat\Phi^\dagger$ as in \eqref{eq:SDSsplittwist} with respect to $\hat L^\dagger\oplus\hat L$,
and reversing the construction of \eqref{eq:SDSsplit} to \eqref{eq:SDSsplittwist} yields a solution $(\hat\nabla,\Phi,h)$ of the
$\mathrm{SU}(2)$ self-duality equations for the positive definite Hermitian metric $h.$ 

Recall that by Theorem \ref{thm:hards3asso},  every (equivariant) harmonic map into de Sitter 3-space yields a solution of 
the $\mathrm{SU}(1,1)$ self-duality equations.
Summarizing, we obtain the following theorem, and 
in particular the statement about smoothness of $f$ in  Proposition \ref{prop:Ntof} for harmonic $N$.

\begin{thm}\label{thm:SU11SU2}
Let $(\hat \nabla,\hat\Phi,\hat h)$  be a solution of the $\mathrm{SU}(1,1)$ self-duality equations such that $\det\Phi=-\omega^2$ and such that
the $-\omega$ eigenline bundle $\hat L$ of $\hat\Phi$ is nowhere null. Then, 
$(\nabla,\Phi,h)$ is a solution of the $\mathrm{SU}(2)$ self-duality equations.

If $N$ is the equivariant harmonic map into de Sitter 3-space associated to $(\hat \nabla,\hat\Phi,\hat h)$, then the harmonic map
into $\mathbb H^3$ associated to $(\nabla,\Phi,h)$ is the dual harmonic map $f$ of $N.$
\end{thm}

\subsubsection{The model solution}\label{sec:modelsolu}
Consider the standard Hermitian metric $h_0=\langle.,.\rangle$ on the trivial $\C^2$-bundle over the  plane  $\C.$
On $M=\C\setminus i\R$, and for $t>0$, consider 
the non-vanishing function \[\rho(x+iy):=\tfrac{\exp (2 x)-1}{\exp (2 x)+1}=\tanh(x).\]
Consider
the unitary connection $\nabla_t^{\model}=d+A_t^{\model}$
and the Higgs field $\Phi_t^{\model}$
given by
\begin{equation}\label{eqn:model}
A_t^{\model} = \tfrac t{\sinh(2 tx)} \begin{psmallmatrix} -i & 0 \\ 0 & i \end{psmallmatrix} dy\,,
\quad
\Phi_t^{\model} = \tfrac 1 2 \begin{psmallmatrix} 0 & \rho^{-1}(tx) \\ \rho(tx) & 0 \end{psmallmatrix} dz.
\end{equation}
As $t\to \infty$, away from $x = 0$, the connection $A_t^{\model}$ converges to the trivial connection, and the Higgs field converges to a constant Higgs field. 
This yields a solution of the {\em rescaled} $\mathrm{SU}(2)$ self-duality equations
\[
\bar\partial^{A_t^{\model}}\Phi_t^{\model}=0\quad \text{and}\quad F_{A_t^{\model}} + t^2 \left[ \Phi_t^{\model} \wedge \left(\Phi_t^{\model} \right)^* \right] = 0.
\]
Equivalently, $(\nabla_t^{\model},t \Phi_t^{\model} ,h_0)$ is a solution to the $\mathrm{SU}(2)$ self-duality equations.

This solution can also be obtained  by applying the complex gauge transformation
\[
g_t^{\model} = \begin{pmatrix} \rho(tx)^{1/2} & 0 \\ 0 & \rho(tx)^{-1/2} \end{pmatrix}
\]
to the pair consisting of the trivial connection and the constant Higgs field $\tfrac 1 2 \begin{psmallmatrix} 0 & 1 \\ 1 & 0 \end{psmallmatrix} dz$. In other words,
\[
\left(A_t^{\model}, \Phi_t^{\model} \right) = \left(d, \tfrac 1 2 \begin{psmallmatrix} 0 & 1 \\ 1 & 0 \end{psmallmatrix} dz\right) \ast g_t^{\model}.
\]

Note that, for $t>0$,
the holomorphic diffeomorphism $\Psi_t\colon\C\to\C,\, z\mapsto tz$
satisfies
\[\Psi_t^*(d+A_1^{\model},\Phi_1^{\model})=(d+A_t^{\model},\,t\,\Phi_t^{\model}).\]
Consider the $\mathrm{SU}(2)$ gauge transformation
\[g_t:=\left(
\begin{array}{cc}
 \frac{e^{2 t x}+1}{\sqrt{2 e^{4 t x}+2}} & \frac{e^{2 t x}-1}{\sqrt{2 e^{4 t x}+2}} \\
 \frac{1-e^{2 t x}}{\sqrt{2 e^{4 t x}+2}} & \frac{e^{2 t x}+1}{\sqrt{2 e^{4 t x}+2}} \\
\end{array}
\right).\]
Then,
\[g_t^{-1}\Phi_t^{\model}g_t=
\begin{psmallmatrix}
 -\frac{1}{2} & \text{csch}(2 t x) \\
 0 & \frac{1}{2} \\
\end{psmallmatrix}
 dz\]
and
\begin{equation}
\begin{split}
(d+A_t^{\model}) \cdot g_t&=d+
\begin{psmallmatrix}
 0 & t\, \text{sech}(2 t x) \\
 -t\, \text{sech}(2 t x) & 0 \\
\end{psmallmatrix}
dx+\begin{psmallmatrix}
 -2 i t\, \text{csch}(4 t x) & -i t\, \text{sech}(2 t x) \\
 -i t \,\text{sech}(2 t x) & 2 i t\, \text{csch}(4 t x) \\
\end{psmallmatrix}dy\\
&=d+\begin{psmallmatrix}
 -t\, \text{csch}(4 t x) & 0 \\
 -t \,\text{sech}(2 t x) & t\, \text{csch}(4 t x) \\
\end{psmallmatrix}dz+
\begin{psmallmatrix}
 t\, \text{csch}(4 t x) & t\, \text{sech}(2 t x) \\
 0 & -t\, \text{csch}(4 t x) \\
\end{psmallmatrix} d\bar z.
\end{split}
\end{equation}

Then, a direct computation shows that the corresponding harmonic map and its oblique Gau\ss\@ maps are given by $f_t$ and $N_t$ as in Example \ref{exa:modelharmonicmap}.
\begin{rem}\label{rem:zylindermodel}
Note that for every $\sigma>0$, the model solutions and all involved gauge transformations
are well-defined on the cylinder
$\mathcal Z=\C/ i \sigma \Z$ away from the central curve corresponding to the imaginary axis.
\end{rem}

\subsection{Lifts of sections of the Deligne--Hitchin moduli space}\label{sec:dhlifts}
While all complex structures $I_\lambda$ on $\mc{M}_{\operatorname{Hit}} $ for $\lambda\neq0$ are biholomorphic, the complex structure
at $\lambda=0$ is quite different. For example, it admits a compact analytic subspace. Furthermore, any $\lambda$-connection for $\lambda\neq0$ is automatically semi-stable (e.g., reducible), while this is not true for Higgs fields. Moreover, Higgs fields might be stable and still admit
proper invariant subbundles of negative degree.
This makes the reinterpretation of the oblique hyperbolic Gau\ss\@ map through the lens of twistor theory possible.

To enhance the comprehensiveness of the main theorem in this section, we introduce some additional notation: $\mathrm{G}^-:=\SU(2)$ and $\mathrm{G}^+:=\SU(1,1)$, allowing us to discuss self-duality solutions for $\sigma\in\{-,+\}$.

\begin{thm}\label{thm:twisor-oblique}
Let $s$ be an admissible $\tau$-real holomorphic section of the Deligne--Hitchin moduli space over $M$ of sign $\sigma\in\{-,+\}$.  Assume that the determinant of the Higgs field at $\lambda=0$ is a square.
Then, there is an open non-empty subset $U\subset M$ and a non-admissible lift of  $s$ on $U$ which gives a solution of
the $G^{-\sigma}$-self-duality equations. If $s$ is negative, then $U=M.$
\end{thm}
\begin{proof}
Let $\det\Phi=-\omega^2$, where $\Phi$ is the Higgs field of an admissible lift. 
First assume $\sigma=-.$ By Theorem \ref{thm:posnegsecLift}, there exists a (positive definite) Hermitian metric $h$ such that $(\nabla,\Phi,h)$
is a $\mathrm{SU}(2)$ self-duality solution. Consider the orthogonal complement $L^\perp$ of the eigenline bundle $L$ of $\Phi$ with respect to
$\omega$. Define $g(\lambda)=\begin{psmallmatrix}\lambda&0\\0&1\end{psmallmatrix}$ with respect to $E=L\oplus L^\perp,$ and
\begin{equation}\label{eq:twistingnabla}\hat\nabla^\lambda:=\nabla^\lambda \cdot g(\lambda).\end{equation}
Then, $\hat\nabla^\lambda$ is the non-admissible lift of $s$ on $M$ which gives a solution to the $\mathrm{SU}(1,1)$ self-duality equations. Note that the corresponding Higgs field of $\hat\nabla^\lambda$ is not stable, as it has $L^\perp\cong L^*$ as an eigenline bundle of positive degree.

If $\sigma=+$, we again consider the eigenline bundle $L$ of $\Phi$ with respect to
$\omega$. First, note that if $L$ is globally a null line bundle, then $L$ would have degree $0$, contradicting stability.
Consider the maximal open set $U\subset M$ on which $L$ is not null, and let $L^\perp$ be the orthogonal complement over $U$.
Reversing the construction of the first part yields a $\mathrm{SU}(2)$ self-duality solution on $U.$
\end{proof}

\begin{exa}
  \label{exa:dual_SU11}
Consider the model $\mathrm{SU}(2)$ self-duality solution from Section \ref{sec:modelsolu}. Then, for $t=1$, the associated family of flat connections is given by
\begin{equation}\label{exa:assof}
\begin{split}
\nabla^\lambda=&d+ 
\begin{psmallmatrix}
 - \csch(4  x) & 0 \\
 - \sech(2  x) & \csch(4  x) 
\end{psmallmatrix}
dz+
  \begin{psmallmatrix}
\csch(4  x) &  \sech(2  x) \\
 0 & - \csch(4  x)\end{psmallmatrix}
  d\bar z\\
&  +\lambda^{-1}
  \begin{psmallmatrix}
 -\frac{1}{2} & \csch(2  x) \\
 0 & \frac{1}{2}\end{psmallmatrix} dz
 +\lambda
  \begin{psmallmatrix}
 -\frac{1}{2} &0 \\
  \csch(2  x) & \frac{1}{2}\end{psmallmatrix} d\bar z,
\end{split}
\end{equation}
while the corresponding  associated family for the $\mathrm{SU}(1,1)$-solution is
\begin{equation}\label{exa:assof11}
\begin{split}
\hat\nabla^\lambda=&d+ 
\begin{psmallmatrix}
 - \csch(4  x) & \csch(2  x) \\
0 & \csch(4  x) 
\end{psmallmatrix}
dz+
  \begin{psmallmatrix}
\csch(4  x) &  0 \\
  \csch(2  x)  & - \csch(4  x)\end{psmallmatrix}
  d\bar z\\
&  +\lambda^{-1}
  \begin{psmallmatrix}
 -\frac{1}{2} & 0 \\
  - \sech(2  x) & \frac{1}{2}\end{psmallmatrix} dz
 +\lambda
  \begin{psmallmatrix}
 -\frac{1}{2} &\sech(2  x) \\
 0& \frac{1}{2}\end{psmallmatrix} d\bar z.
\end{split}
\end{equation}
Gauging \eqref{exa:assof11} by the  gauge
\[g=
\begin{psmallmatrix}
 \frac{1}{\sqrt{1-e^{4 x}}} & \frac{e^{2 x}}{\sqrt{1-e^{4 x}}} \\
 \frac{e^{2 x}}{\sqrt{1-e^{4 x}}} & \frac{1}{\sqrt{1-e^{4 x}}} \\
\end{psmallmatrix}\]
-- which is $\mathrm{SU}(1,1)$-valued for $x<0$ -- yields
\[d+
\begin{psmallmatrix}
 \frac{\tanh (2 x)}{2 \lambda } & \frac{-\lambda  \sech(2 x)-\sech(2 x)}{2 \lambda } \\
 \frac{\lambda  \sech(2 x)-\sech(2 x)}{2 \lambda } & -\frac{\tanh (2 x)}{2 \lambda } \\
\end{psmallmatrix}
 dz+ 
\begin{psmallmatrix}
 \frac{1}{2} \lambda  \tanh (2 x) & \frac{1}{2} (\lambda  \sech(2 x)+\sech(2 x)) \\
 \frac{1}{2} (\lambda  \sech(2 x)-\sech(2 x)) & -\frac{1}{2} \lambda  \tanh (2 x) \\
\end{psmallmatrix} d\bar z.
\]
This latter family of flat connections smoothly extends  through the central curve $\{x=0\}$, and adheres to the positive $\tau$-symmetry.
\end{exa}

\section{Analytic preliminaries}
 
\subsection{Singular solutions to Hitchin's equation}\label{sec:41}
Throughout this section, let $M$ be a compact Riemann surface and $(E, \bar\partial_E, \varphi)$ a stable $\SL(2,\C)$-Higgs bundle over $M$ such that $q = -\det \varphi$ has simple zeros. Recall that the vertical foliation associated to $q \in H^0(M,K^2)$ is given by curves $\gamma: I \to M$ such that
\[
 q( \gamma'(t), \gamma'(t)) < 0 
\]  
for all $t \in I$. If locally $q = \omega^2$, the vertical foliation is given by $\ker \Re \omega$. In particular, if $z= x + i y$ is a local holomorphic coordinate such that $q = dz^2$, then the foliation is given by $\ker dx$, and hence is integrated by vertical lines in the $z$-plane. Singularities are given by the zeros of $q$. A leaf is called \emph{critical} if it starts or ends at a zero. Simple zeros are three-pronged singularities, i.e.\ are met by three critical leaves.
  
A holomorphic quadratic differential $q$ is called a \emph{Strebel differential} if the closure of the union of the critical leaves is compact. In that case, it forms a finite graph $\Gamma \subset M$ whose complement is the union of regions $V_1, \ldots, V_{N_L}$ foliated by closed leaves isotopic to simple closed curves $c_1, \ldots, c_{N_L}$. These regions are biholomorphic to annuli and are called \emph{Strebel cylinders}. 
The number of Strebel cylinders $N_L$ is bounded by $3g -3$ if $g$ is the genus of $M$, since the core loops $c_1, \ldots, c_{N_L}$ are pairwise non-isotopic. Strebel differentials with simple zeros exist in abundance: Strebel differentials are dense in the space of all differentials by \cite{Douady-Hubbard}, and hence Strebel differentials with simple zeros are dense in the space of differentials with simple zeros.
 
Assume now that $q = -\det \varphi$ is a Strebel differential with simple zeros. Fix core loops $c_1, \ldots, c_{N_L}$ in Strebel cylinders $V_1, \ldots, V_{N_L}$. There are biholomorphisms 
\[
V_j \to \{z \in \C : |\Real z| < \tau_j \}/ (i \sigma_j \Z), \qquad \tau_j, \sigma_j > 0,
\]
identifying $q$ with $\frac 14dz^2$. By scaling, we may assume that $\tau_j >1$ for all $j$. From this, we derive the following normal form for the Higgs bundle over $V_j$.

\begin{lemma}\label{lem:nf_Strebel}
Let $V_j$ be a Strebel cylinder and $V_j \to \{z \in \C : |\Real z| < \tau_j \}/ (i \sigma_j \Z)$ a biholomorphism with respect to which $q = \frac 14dz^2$. Then there exists a local holomorphic trivialization of $E$ over $V_j$ such that, with respect to the above biholomorphism, 
\[
\varphi = \frac 12 \begin{pmatrix}
 0 & 1 \\ 1 & 0	
 \end{pmatrix} dz\,.
\]
\end{lemma}

\begin{proof}
Since $- \det \varphi = \frac 14 dz^2$, the eigenvalues of $\varphi$ are given by $\pm \frac 12 dz$. If $(e_+,e_-)$ is a holomorphic eigenframe of $\varphi$, the desired holomorphic trivialization is provided by the frame $e_1 = e_++e_-$ and $e_2 = e_+-e_-$.
\end{proof}

This is precisely the Higgs bundle underlying the model solution \eqref{eqn:model}. Our goal in this and the following section is to construct solutions to Hitchin's equation on $M$ which are singular along the core loops $c_j$ and are asymptotic (in a yet to be specified sense) to the model solution on each Strebel cylinder $V_j$. We will use gluing techniques in the spirit of \cite{MSWW} to produce solutions with large energy, i.e.\ energies exceeding some sufficiently large threshold. The existence question for general Higgs bundle data is thus left open.

\subsection{The analytic setup}
We briefly describe the setup for solving the Hitchin self-duality equation, following the approach used in \cite{MSWW}. Assume that $M$ is a Riemann surface equipped with a conformal metric $g$ and that $E = \underline{\C}^2$ is the trivial rank $2$ complex vector bundle over $M$ equipped with a Hermitian metric $h$, which we take to be the standard Hermitian inner product on each fiber $\C^2$. Let $t > 0$. For a pair $(A, \Phi)$ consisting of an $\SU(2)$-connection $A$ on $E$ and a $(1,0)$-form $\Phi \in \Omega^{1,0}(\End_0 E)$ we define
\[
\mc{H}_t(A, \Phi) = \left( F_A + t^2 [\Phi \wedge \Phi^*], \bar\partial_A \Phi \right).
\]
Then $\mc{H}_t(A, \Phi) = 0$ if and only if $(A, \Phi)$ solves the $t$-rescaled Hitchin self-duality equations, i.e.\ $(A, t \Phi)$ solves the Hitchin self-duality equations. 

Recall from Section \ref{sec:gauge_theory} that there is a natural action of the complex gauge group $\mc{G}^{\C}$ on pairs $(A,\Phi)$. For a fixed pair and a complex gauge transformation $g \in \mc{G}^{\bb{C}}$ we denote
\[
\mc{O}_{(A,\Phi)}(g) = \left(A, \Phi\right) \ast g .
\]
Finally, for $\gamma \in \Gamma( i \mf{su}(E))$ we denote
\[
\mc{F}_t(\gamma) = \operatorname{pr}_1(\mc{H}_t(\mc{O}_{(A,\Phi)}(\exp(\gamma)))).
\]
Note that if $\bar\partial_A \Phi = 0$, then $\bar\partial_{A \ast g} (\Phi \ast g) = 0$, and therefore only the curvature part of the equation needs to be considered. Therefore, if $(A,\Phi)$ is such that $\Phi$ is $\bar\partial_A$-holomorphic, then $\mc{F}_t(\gamma) = 0$ implies that $(A,\Phi) \ast \exp(\gamma)$ solves the self-duality equation.

It turns out that this operator is elliptic. Indeed, one finds that
\[
d_0 \mc{F}_t (\gamma) = i \ast \Delta_A \gamma + t^2 M_{\Phi} \gamma,
\]
where
\[
M_\Phi (\gamma) = \left[ \Phi^* \wedge [\Phi, \gamma]\right] - \left[ \Phi \wedge [\Phi^*, \gamma] \right].
\]
We define $L_t = - i \ast \, d_0 \mc{F}_t$ and obtain
\[
L_t (\gamma) = \Delta_A \gamma - i \ast \, t^2 M_{\Phi} (\gamma).
\]

\subsection{Approximate solutions}\label{sec:approxsol}

As we saw in Section \ref{sec:modelsolu}, the model solution converges to a pair consisting of a flat connection and a constant Higgs field away from the loop $\{x=0\}$. This opens the door to constructing approximate solutions by finding limiting configurations containing a Strebel cylinder. Indeed, in this section we construct such approximate solutions by gluing in the family of model solutions in each Strebel cylinder. At zeros of the quadratic differential, we will need to glue in {\em fiducial} solutions \cite{MSWW}, which will be discussed in the next subsection.

\subsubsection{The fiducial solution}
In the following, the underlying Hermitian metric is the standard one.
The limiting fiducial solution on $\C^*$ is given by
\begin{equation}
\label{eqn:limiting_fid}
A_\infty^{\fid} = \frac 1 8 \begin{pmatrix} 1 & 0 \\ 0 & -1 \end{pmatrix} \left( \frac {dz}z - \frac{d \bar z}{\bar z} \right), \qquad \Phi_\infty^{\fid} = \begin{pmatrix} 0 & r^{1/2} \\ r^{-1/2} z & 0 \end{pmatrix} dz,
\end{equation}
where $r=|z|$. While the determinant of the Higgs field has a simple zero at $z=0$, the solution becomes singular at $z=0$.
There is a family $(A_t^{\fid}, \Phi_t^{\fid})$ which approaches this solution as $t \to \infty$ and satisfies
\[
F_{A_t^{\fid}} + t^2 [\Phi_t^{\fid} \wedge (\Phi_t^{\fid})^*] = 0.
\]

We first define
\[
A_0^{\fid} = 0, \qquad \Phi^{\fid}_0 = \begin{pmatrix} 0 & 1 \\ z & 0 \end{pmatrix} dz.
\]
Observe that $(A_0^{\fid}, \Phi^{\fid}_0)$ and $(A_\infty^{\fid}, \Phi^{\fid}_\infty)$ are complex gauge equivalent via the complex gauge transformation
\[
g_\infty = \begin{pmatrix} r^{-1/4} & 0 \\ 0 & r^{1/4} \end{pmatrix}.
\]

Consider the family of complex gauge transformations
\begin{equation*}
g^{\fid}_t
=
\begin{pmatrix}
r^{-1/4} e^{-\ell_t(r)/2} & 0 \\
0 & r^{1/4} e^{  \ell_t(r)/2}
\end{pmatrix},
\end{equation*}
where $\ell_t$ is the unique solution of the ODE
\[
\left(\frac{d^2}{dr^2} + \frac 1 r \frac d{dr}\right) \ell_t = 8 t^2 r \sinh(2 \ell_t)
\]
with the asymptotics $\ell_t(r) \sim \frac 1 2 \log r$ as $r \to 0$ and which decays exponentially as $t \to \infty$, see \cite{MSWW} and the references therein.
Then, the family of solutions to the rescaled self-duality equation is given by $(A_t^{\fid}, \Phi_t^{\fid})=(A_0^{\fid}, \Phi^{\fid}_0) * g^{\fid}_t$, i.e. 
\begin{equation}
\label{eqn:fid}
A_t^{\fid} =  \left(\frac 18 + \frac{r}{4} \frac{\partial \ell_t}{\partial r}\right)\begin{pmatrix} 1 & 0 \\ 0 & -1 \end{pmatrix} \left( \frac {dz}{z} - \frac{d \bar z}{\bar z} \right), \qquad 
\Phi_t^{\fid} = \begin{pmatrix} 0 & r^{1/2} e^{\ell_t(r)}\\ r^{-1/2} e^{-\ell_t(r)}z & 0 \end{pmatrix} dz.
\end{equation}

Here are some properties of $\ell_t$, which show that we actually interpolate between the trivial and the limiting fiducial solution:
\begin{enumerate}
\item For fixed $t$ and $r \searrow 0$ one has
\[
\ell_t(r) \sim - \frac 1 2 \log(r) + b_0 + \ldots,
\]
where $b_0$ is an explicit constant.
\item There exists a constant $C > 0$ such that
\[
|\ell_t(r)| \leq C \exp \Bigl( - \frac 8 3 t\, r^{3/2} \Bigr) \frac 1 {(t r^{3/2})^{1/2}}
\]
uniformly for $t \geq t_0 > 0$, $r \geq r_0 > 0$.
\item There exists a constant $C>0$, independent of $t$, such that
\[
\sup_{r \in (0,1)} r^{1/2} e^{\pm \ell_t(r)} \leq C.
\]
\end{enumerate}

\subsubsection{Construction of the approximate solutions}
Let $M$ be a compact Riemann surface and $(E, \bar\partial_E, \varphi)$ a stable $\SL(2,\C)$-Higgs bundle over $M$. Assume moreover that $q = - \det \varphi$ is a Strebel differential with simple zeros. Let $V_1, \ldots, V_{N_L}$ denote the collection of Strebel cylinders. Fix a Hermitian metric $h_0$ on $E$ and suppose that $(A_\infty, \Phi_\infty)$ is a limiting configuration for the Higgs bundle $(E, \bar\partial_E, \varphi)$. Then, by \cite{MSWW}, the following two propositions hold:
\begin{prop}
\label{prop:local_form_near_zero}
If $q(p_0) \neq 0$ for $p_0 \in M$, then there exists a neighborhood $U$ of $p_0$, holomorphic coordinates $z$ on $U$, and a unitary frame of $(E,h_0)$ over $U$, such that in these coordinates and this frame $q = dz^2$, $A_\infty$ is the trivial connection, and \[\Phi_\infty = \begin{pmatrix} 0 & 1 \\ 1 & 0 \end{pmatrix} dz.\]
\end{prop}

\begin{prop}
\label{prop:std_local_form}
If $q(p_0) = 0$ for $p_0 \in M$, then there exists a neighborhood $U$ of $p_0$, holomorphic coordinates $z$ on $U$, and a unitary frame of $(E, h_0)$ over $U$, such that in these coordinates and this frame $q = z \, dz^2$, and $(A_\infty, \Phi_\infty)$ coincides with $(A_\infty^{\fid}, \Phi_\infty^{\fid})$.
\end{prop}

Using Lemma \ref{lem:nf_Strebel} we obtain:
\begin{prop}
\label{prop:Strebel_local_form}
Let $V_j$ be a Strebel cylinder and $V_j \to \{z \in \C : |\Real z| < \tau_j \}/ (i \sigma_j \Z)$ a biholomorphism with respect to which $q = -\frac 14dz^2$. Then there exists a unitary frame of $E$ over $V_j$ such that, with respect to this trivialization, $A_\infty$ becomes the trivial flat connection and 
\[
\Phi_\infty=\frac 1 2 \begin{pmatrix} 0 & 1 \\ 1 & 0 \end{pmatrix} dz.
\]
\end{prop}

The existence of these local normal forms makes it evident that any limiting configuration for a Higgs bundle of this type can be turned into a \emph{framed limiting configuration} in the sense of the following definition. Here $\mathbb{D} \subset \C$ denotes the unit disk.

\begin{defn}\label{def:FLC}
Let $(E, \bar\partial_E, \varphi)$ be a stable $\SL(2,\C)$-Higgs bundle over the compact Riemann surface $M$ such that $q = - \det \varphi$ is a Strebel differential with simple zeros. A {\em framed limiting configuration} for this Higgs bundle consists of the following data:
\begin{enumerate}
\item a Hermitian metric $h_0$ on $E$,
\item a limiting configuration $(A_\infty, \Phi_\infty)$ for the Higgs bundle $(E, \bar\partial_E, \varphi)$,
\item open neighborhoods $W_1, \ldots, W_{N_Z}$ around each zero of $q$, together with holomorphic coordinates $W_j \to \mathbb{D} \subset \C$ and a unitary frame of $E$ over $W_j$, such that with respect to this trivialization $(A_\infty,\Phi_\infty)=(A_\infty^{\fid}, \Phi_\infty^{\fid})$,
\item open sets $V_1, \ldots, V_{N_L}$ together with biholomorphisms $V_j \to \{z \in \C : |\Real z| < 1 \}/ (i \sigma_j \Z)$, $\sigma_j > 0$, and a unitary frame of $E$ over $V_j$, such that with respect to this trivialization $A_\infty$ becomes the trivial flat connection and 
\[
\Phi_\infty = \frac 1 2 \begin{pmatrix} 0 & 1 \\ 1 & 0 \end{pmatrix} dz,
\]
\item open sets $U_1, \ldots, U_{N_I}$ together with coordinates $U_j \to \mathbb{D}$ and unitary frames, such that with respect to this trivialization $A_\infty$ becomes the trivial flat connection and 
\[
\Phi_\infty \,=\, \begin{pmatrix} 0 & 1 \\ 1 & 0 \end{pmatrix} dz.
\]
\end{enumerate}
Furthermore, we assume that $\{U_j, V_k, W_l\}_{j,k,l}$ forms an open cover of $M$, and we also assume that the loops $\{ \Real z = 0 \} \subset V_k$ do not intersect any of the sets $U_j$. Let us also assume that the $\{V_k, W_l\}_{k,l}$ are pairwise disjoint. The sets $c_k = \{\Real z = 0\} \subset V_k$ will be called {\em core loops}. The set of zeros of $q$ will be denoted by $Z$. 
\end{defn}

{\em Remarks.}
\begin{enumerate}
\item Note that, by construction, the sets $U_1, \ldots, U_{N_I}$ and $V_1, \ldots, V_{N_L}$ do not contain any zeros of $q$.
\item Using Propositions \ref{prop:local_form_near_zero} and \ref{prop:std_local_form}, any Higgs bundle over a compact Riemann surface whose quadratic differential has only simple zeros can be given the structure of a framed limiting configuration. This construction results in $N_L = 0$ and corresponds to the construction in \cite{MSWW}.
\item The assumption on the number of components is not necessary to construct singular solutions of the $\SU(2)$ self-duality equations. Instead, it is required to construct a harmonic map whose behavior is consistent with that of the model solution. The two different components in the complement of the union of the core loops then correspond to the preimage of $\mathbb H^3_+$ and $\mathbb H^3_-$, respectively.
\end{enumerate}

For the rest of the article we assume that we are given such a framed limiting configuration $(A_\infty, \Phi_\infty)$. We define
\[
M^\vee = M \bs \bigcup_{k=1}^{N_L} c_k, \qquad M^\times = M^\vee \bs Z.
\]
Let $\chi : \R \to [0,1]$ be a smooth function, which is $1$ on $[-1/4, 1/4]$ and vanishes on $(-\infty,-1/2) \cup (1/2, \infty)$. Define a complex gauge transformation on $M^\vee$ via
\[
g_t^{\appr}(p) =
\begin{cases}
\exp\left(\chi(|z(p)|) \log (g_t^{\fid}(z(p)))\right), & p \in W_l \text{ for some } 1 \leq l \leq N_Z, \\
\exp\left( \chi(\Real(z(p))) \log ( g_t^{\model}(z(p)))\right), & p \in V_k \text{ for some } 1 \leq k \leq N_L, \\
\id_E, & \text{ otherwise,}
\end{cases}
\]
where $z$ denotes the chosen local coordinate on the corresponding chart.
The approximate solution is then defined as the gauge transformation of the limiting configuration by $g_t^{\appr}$:
\[
(A_t^{\appr}, \Phi_t^{\appr}) = (A_\infty, \Phi_\infty) \ast g_t^{\appr}.
\]
Note that, in contrast to the approximate solutions constructed in \cite{MSWW}, our approximate solutions are actually singular. They smooth out the singularities of the limiting configuration at the zeros of $q$, just as in \cite{MSWW}, but they also introduce singularities at the core loops of the Strebel cylinders via the model solution. The analysis of the gluing problem is hence situated on the non-compact surface $M^\vee$, and we will see that the precise form of the model solution turns it into a problem for $0$-differential (or uniformly degenerate) operators, see the following section.

\subsection{Some 0-calculus background}
\label{sec:0-calculus}
Here we collect some results from the $0$-calculus, which we will apply in the following sections. The $0$-calculus goes back to work of Mazzeo and Melrose \cite{MM}, and has been extended to a more general class of operators in \cite{Mazzeo}. We cite the latter for the statements of the main theorems. This review is partially based on the presentation in \cite{Usula}, Chapter 2.

Suppose $\ol{M}$ is a compact manifold with boundary of dimension $n+1$. In our application, $\dim M^\vee = 2$, so $n=1$. We state the general $(n+1)$-dimensional formulation for completeness. Suppose that $\rho$ is a boundary defining function on $\ol{M}$. By $M$ we mean $\ol{M} \setminus \partial \ol{M}$. We will also denote $\partial \ol{M}$ by $Y$. Assume that we are given a metric $\ol{g}$ on $\ol{M}$. The metric $g = \rho^{-2} \ol{g}$ on $M$ is a complete, infinite-volume metric.

Given $p \in Y$, we define
\[
\ol{M}_p = \{(t, \upsilon_1, \ldots, \upsilon_n) : t \geq 0\} \subset T_p \ol{M},
\]
where $(t, \upsilon_1, \ldots, \upsilon_n)$ are the linear coordinates induced by coordinates $(\rho, y_1, \ldots, y_n)$ of $\ol{M}$ near $p \in Y$ (here $(y_1, \ldots, y_n)$ is a chart on $Y$ pulled back to $\ol{M}$). On this we define the hyperbolic metric
\[
g_p = \frac{dt^2 + d\upsilon^2}{t^2}.
\]
Let $\ol{E}, \ol{F}$ be vector bundles over $\ol{M}$.

A differential operator $P : \Gamma(\ol{E}) \to \Gamma(\ol{F})$ is called a $0$-differential operator of order $m$ if, for any $p \in Y$, there exists a chart $(y_1, \ldots, y_n)$ centered at $p$ and trivializations of $\ol{E}$ and $\ol{F}$ near $p$, such that $P$ can be written as
\[
P = \sum_{j + |\beta| \leq m} P_{j, \beta}(\rho, y) (\rho \partial_\rho)^j (\rho \partial_y)^\beta,
\]
for functions $P_{j,\beta}$ smooth in the coordinates $(\rho, y_1, \ldots, y_n)$ with $\rho\geq 0$. We denote the space of such differential operators by $\Diff_0^m(\ol{E}, \ol{F})$. A typical example of such a differential operator are the Laplacian operators of hyperbolic space, expressed as an operator on the compactification. Operators in $\Diff_0^m(\ol{E}, \ol{F})$ have analogous boundary behavior as these Laplacians.

In these coordinates, the $0$-symbol of $P$ is defined by
\[
\sigma_0(P)(\rho,y) (\tau,\eta) = \sum_{j+|\beta|=m} P_{j,\beta}(\rho,y) \tau^j \eta^\beta,
\]
where $(\tau,\eta) \in \R \times \R^n$. The operator $P$ is called $0$-elliptic if $\sigma_0(P)(\rho, y)(\tau,\eta)$ is invertible for all $(\tau,\eta) \neq (0,0)$. Note that this definition only depends on the highest-order part of the operator. Connection Laplacians are $0$-elliptic, and, up to lower-order terms, these are the only types of operators studied in this paper.

For a point $p \in Y$, the {\em normal operator} $N_p(P)$ is a $0$-differential operator between $\ol{M}_p \times \ol{E}_p$ and $\ol{M}_p \times \ol{F}_p$, and it is defined in terms of the local expression above as
\[
N_p(P) = \sum_{j+|\beta| \leq m} P_{j,\beta}(0,p) (t \partial_t)^j (t \partial_\upsilon)^\beta.
\]
Given $s \in \C$, we also define
\[
I_s(P) = (\rho^{-s} P \rho^s)|_Y  \in \Gamma(\Hom(\ol{E}|_Y, \ol{F}|_Y)).
\]
The family $(I_s(P))_s$ is called the {\em indicial family} of $P$. A number $\mu \in \C$ is called an {\em indicial root} of $P$ at $p \in Y$ if $I_\mu(P)(p) : \ol{E}_p \to \ol{F}_p$ is not invertible. This operator also has the local expression
\[
I_s(P)(p) = \sum_{j\leq m} P_{j,0}(0,p) s^j.
\]

By $L^2_0$ and $C^{0,\alpha}_0$ we denote the standard $L^2$ space and H\"older space with respect to the metric $g$.
Denote
\[
\mc{V}_0 = \{ V \in \Gamma(T \ol{M}) : V|_Y = 0 \}.
\]
Then we define
\[
L^{2,k}_0 = \{ u : V_1 \ldots V_l u \in L_0^2\; \forall V_j \in \mc{V}_0, l \leq k \},
\]
\[
C^{k,\alpha}_0 = \{ u : V_1 \ldots V_l u \in C^{0,\alpha}_0\; \forall V_j \in \mc{V}_0, l \leq k \}.
\]
On these spaces we define the norms
\[
\|u\|^2_{L^{2,k}_0} = \sum_{j=0}^k \|\nabla^j u\|^2_{L^2_0},
\]
and
\[
\|u\|_{C^{k,\alpha}_0} = \|\nabla^k u\|_{C^{0,\alpha}_0}+ \sum_{j=0}^{k-1} \|\nabla^j u\|_{C^0_0}.
\]
Here $\nabla$ denotes the connection with respect to the complete metric $g = \rho^{-2}$ and all norms are computed with respect to this metric.

The weighted space $\rho^\delta L^{2,k}_0$ is given by $\{ \rho^\delta v : v \in L^{2,k}_0\}$, and the norm is given by $\|u\|_{\rho^\delta L^{2,k}_0} = \|\rho^{-\delta} u\|_{L^{2,k}_0}$. Similarly, the weighted space $\rho^\delta C^{k,\alpha}_0$ is given by $\{\rho^\delta v : v \in C^{k,\alpha}_0 \}$, and its norm is $\|u\|_{\rho^\delta C^{k,\alpha}_0} = \|\rho^{-\delta} u\|_{C^{k,\alpha}_0}$.

The following theorems contain the basic facts used when working with elliptic 0-operators: elliptic regularity, Fredholm properties, and boundary regularity.

\begin{thm}[Theorem 3.8 in \cite{Mazzeo}, Proposition 13 in \cite{Usula}]
\label{thm:0-elliptic-regularity}
Suppose $P$ is a $0$-elliptic operator of order $m$ and that $u \in \rho^\delta L^2_0$ satisfies $P u = v$ with $v\in \rho^\delta L^{2,k}_0$. Then $u \in \rho^\delta L^{2,k+m}_0$. 
\end{thm}

\begin{thm}[Theorem 6.1, Proposition 7.17 in \cite{Mazzeo}, Theorem 18 in \cite{Usula}]
\label{thm:fredholm-0-calculus}
Suppose $P$ is a formally self-adjoint $0$-differential operator whose indicial roots are constant on $Y$, and suppose that the symbol of $P$ is $0$-elliptic.

Suppose that there exist $\delta_- < n/2 < \delta_+$ such that there are no indicial roots between $\delta_-$ and $\delta_+$.

Suppose that the normal operator $N_p(P)$ is invertible as a map $L^{2,m}_0(\ol{E}_p) \to L^2_0(\ol{E}_p)$.

Then, for every $\delta \in (\delta_-, \delta_+)$, the maps
\[
P : \rho^{\delta - n/2} L_0^{2, k+m}(\ol{E}) \to \rho^{\delta - n/2} L^{2,k}_0(\ol{E}),
\]
\[
P : \rho^{\delta } C^{k+m, \alpha}_0(\ol{E}) \to \rho^{\delta} C^{k,\alpha}_0(\ol{E}),
\]
are Fredholm maps of index $0$, and they all have the same kernel.
\end{thm}

\begin{thm}[Proposition 7.17 in \cite{Mazzeo}]
\label{thm:polyhomogeneous_expansion}
Suppose $P$ is a $0$-elliptic operator of order $m$ and that $u \in \rho^{\delta-n/2} L^2_0$ satisfies $Pu = 0$ with $\delta > \delta_-$, where $\delta_-$ is as in Theorem \ref{thm:fredholm-0-calculus}, and that $\delta$ is not an indicial root of $P$.

Then there exists a polyhomogeneous expansion of $u$. More precisely, there exist functions $u_{j,l,p} \in C^\infty(Y)$ such that
\[
u \sim \sum_j \sum_{l=0}^\infty \sum_{p=0}^{p_j} \rho^{s_j + l} \left(\log \rho\right)^p u_{j,l,p}(y),
\]
where $s_j$ denotes the indicial roots with $\Real s_j > \delta$.
\end{thm}

\subsection{Function spaces adapted to the gluing procedure}
In this section, function spaces are defined on which the linearization of our operator turns out to be invertible in a controlled way. To this end, we first introduce weight functions that measure the distance to the core loops and the zeros of the quadratic differential.

Let $\sigma : \R \to [-1,1]$ be a smooth function, such that $\sigma(t) = t$ for $t \in [-1/2,1/2]$ and $\sigma(t) = 1$ for $|t| \geq 1$.

First, we define weight functions that vanish on the core loops. Given $1 \leq k \leq N_L$, let $z : V_k \to \{z \in \C : |\Real z| < 1\}/(i \sigma_k \Z)$ be the biholomorphic map provided by the framed limiting configuration. We then define $\rho_k : M \to \R$ via $\sigma(|\Real z|)$ on $V_k$, and $\rho_k|_{M \setminus V_k} \equiv 1$. Let us also define $\rho = \rho_1 \cdot \ldots \cdot \rho_{N_L}$.

Next, we define weight functions that vanish on the zeros of the quadratic differential. Given $1 \leq l \leq N_Z$, let $z : W_l \to D$ be the coordinates provided by the framed limiting configuration. Then let $r_l : M \to \R$ be defined via $\sigma(|z|)$ on $W_l$, and $r_l|_{M \setminus W_l} \equiv 1$. Set $r := r_1 \cdot \ldots \cdot r_{N_Z}$.

Given $t > 0$, define
\[
r_t = \left(t^{-4/3} + r^2\right)^{1/2}.
\]

Fix a conformally compact metric $g$ on $M^\vee$. For concreteness, we may assume $g$ to be given by 
\[
\frac{dx^2 + dy^2}{x^2}
\]
on the Strebel cylinders $V_j$ in the coordinates $z = x + i y$ as above.

In the following, we will define doubly-weighted H\"older spaces. 

For the $C^{k,\alpha}$ spaces, we first define ``pointwise norms.'' This will be done with respect to the coordinate systems and trivializations of the framed limiting configuration. We start with the $C^k$ norm. Let $p \in U_{j_0}$ for some $1 \leq j_0 \leq N_I$, and assume also that $r(p) = 1$ and $\rho(p) = 1$. Then we define
\[
\|\gamma\|_{C^k_t}(p) = \sum_{j=0}^k |\nabla^j \gamma(p)|.
\]
On the right-hand side, $\gamma$ is interpreted (via coordinates and trivialization) as a map from the disk to $\C^r$, and $\nabla^j$ denotes all the partial derivatives of order $j$.

Next, assume $p \in V_{k_0}$ for some $1 \leq k_0 \leq N_L$. Then we define
\[
\|\gamma\|_{C^k_t}(p) = \sum_{j=0}^k |\rho^j \nabla^j \gamma(p)|.
\]
Finally, assume $p \in W_{l_0}$ for some $1 \leq l_0 \leq N_Z$. Then we define
\[
\|\gamma\|_{C^k_t}(p) = \sum_{j=0}^k |r_t^j \nabla^j \gamma(p)|.
\]

The H\"older coefficients are defined similarly. Again, first assume $p \in U_{j_0}$ for some $1 \leq j_0 \leq N_I$, and assume also that $r(p) = 1$ and $\rho(p) = 1$. Then we define
\[
[\gamma]_{\alpha;t}(p) = \sup_{0 < |p-q| < \epsilon} \frac{|\gamma(p)-\gamma(q)|}{|p-q|^\alpha}.
\]
Next, assume $p \in V_{k_0}$ for some $1 \leq k_0 \leq N_L$. Then we define
\[
[\gamma]_{\alpha;t}(p) = \rho(p)^\alpha \sup_{0 < |p-q| < \epsilon} \frac{|\gamma(p)-\gamma(q)|}{|p-q|^\alpha}.
\]
Finally, assume $p \in W_{l_0}$ for some $1 \leq l_0 \leq N_Z$. Then we define
\[
[\gamma]_{\alpha;t}(p) = r_t(p)^\alpha \sup_{0 < |p-q| < \epsilon} \frac{|\gamma(p)-\gamma(q)|}{|p-q|^\alpha}.
\]

We also define
\[
\|\gamma\|_{C^{k,\alpha}_t}(p) = \|\gamma\|_{C^k_t}(p) + [\nabla^k \gamma]_{\alpha;t}(p).
\]
Then, given $\delta \in \R$ and $\nu \in \R$, we define the spaces \[\rho^\delta r_t^\nu C^{k,\alpha}_t := \left\{\rho^\mu r_t^\nu \gamma : \gamma \in C^{k,\alpha}_t \right\}\] with the norm
\[
\|\gamma\|_{\rho^\mu r_t^\nu C^{k,\alpha}_t} = \sup_{p \in M^\vee} \rho^{-\mu} r_t^{-\nu} \|\gamma\|_{C^{k,\alpha}_t}(p). 
\]
For $t = \infty$, we restrict the domain to $M^\times$.

\section{Construction of new solutions}

In this section we analyze the deformation theory of the approximate solutions constructed in
the previous section and establish the existence of genuine solutions to the self-duality equations
for large parameter $t$. Our strategy follows the general gluing scheme developed in~\cite{MSWW}, but
incorporates the new model solutions along the core loops of the Strebel cylinders. The analysis
proceeds in several steps: first, we study the local structure of the linearized operator $L_t$ and
its Fredholm properties on appropriate weighted function spaces; second, we obtain uniform
Schauder estimates and invertibility results; finally, these estimates are used to solve the nonlinear
problem by a contraction mapping argument, yielding exponentially close exact solutions. 
We then analyze the decay behavior of these solutions near the core loops and investigate the
regularity of the associated harmonic and transgressive maps. In particular, we show that under a
natural symmetry assumption the corresponding transgressive maps extend smoothly across the
equatorial $2$-sphere.

\subsection{Analysis of the linearized equation}

\subsubsection{Local forms and basic analytic properties of the operator $L_t$}

\begin{lemma}[Local form in the interior]
\label{lemma:local_form_interior}
Let $A_\infty$ be the trivial connection acting on the trivial rank $2$ bundle over $\C$ equipped with the Euclidean metric $dx^2{+}dy^2$ , and let $\Phi_\infty = \begin{pmatrix} 0 & 1 \\ 1 & 0 \end{pmatrix}$. The operators $\Delta_{A_\infty}$ and $-i \ast M_{\Phi_\infty}$ induced on $i\mf{su}(2)$ are
\[
\Delta_{A_\infty} \begin{pmatrix} u & v \\ \ol{v} & - u \end{pmatrix}
=
\begin{pmatrix}
\Delta u & \Delta v \\ \Delta \ol{v} & -\Delta u \end{pmatrix}
\]
and
\[
- i \ast M_{\Phi_\infty} \begin{pmatrix} u & v \\ \ol{v} & -u \end{pmatrix}
=
\begin{pmatrix}
16 u & 8 (v - \ol{v}) \\
8(\ol{v} - v) & -16 u
\end{pmatrix}.
\]
\end{lemma}
The expression for $\Delta_{A_\infty}$ is standard, and the expression for $-i \ast M_{\Phi_\infty}$ follows from a simple calculation.

The following two results are taken from \cite{FMSW}, Section 4.1.
\begin{lemma}[Local form near the zeros of $q$, $t=\infty$]
\label{lemma:local_form_near_zero_infinity}
Let $(A_\infty^{\fid}, \Phi_\infty^{\fid})$ be the limiting fiducial solution given in  \eqref{eqn:limiting_fid} acting on the trivial rank $2$ bundle over $\C$ equipped with the Euclidean metric $dx^2{+}dy^2$. The operator $\Delta_{A_\infty}$ induced on $i\mf{su}(2)$ is given by
\[
\Delta_{A_\infty} \begin{pmatrix} u & v \\ \ol{v} & - u \end{pmatrix}
=
\begin{pmatrix}
\Delta u & \Delta_{1/2} v \\ \Delta_{1/2} \ol{v} & -\Delta u \end{pmatrix},
\]
where $\Delta_{1/2}$ is given in polar coordinates as $-\partial_r^2 - r^{-1} \partial_r - r^{-2} \left(\partial_\theta + \frac i 2 \right)^2$. The operator $-i \ast M_{\Phi_\infty}$ is given in polar coordinates by
\[
- i \ast M_{\Phi_\infty} \begin{pmatrix} u & v \\ \ol{v} & -u \end{pmatrix}
=
r
\begin{pmatrix}
16 u & 8 (v - e^{-i\theta} \ol{v}) \\
8(\ol{v} - e^{i\theta} v) & -16 u
\end{pmatrix}.
\]
\end{lemma}

\begin{lemma}[Local form near the zeros of $q$, $t$ finite]
\label{lemma:local_form_near_zero_finite_t}
Let $(A_t^{\fid}, \Phi_t^{\fid})$ be the fiducial solution given in \eqref{eqn:fid} acting on the trivial rank $2$ bundle over $\C$ equipped with the Euclidean metric $dx^2{+}dy^2$. The operator $\Delta_{A_t^{\fid}}$ induced on $i\mf{su}(2)$ is given by
\[
\Delta_{A_t^{\fid}} \begin{pmatrix} u & v \\ \ol{v} & - u \end{pmatrix}
=
\begin{pmatrix}
\Delta u & \Delta v - 4 F_t i r^{-2} \partial_\theta v + 4 F_t^2 r^{-2} v \\ \Delta \ol{v} + 4 F_t i r^{-2} \partial_\theta \ol{v} + 4 F_t^2 r^{-2} \ol{v} & -\Delta u \end{pmatrix}.
\]
The operator $-i \ast M_{\Phi_t^{\fid}}$ is given in polar coordinates by
\[
- i \ast M_{\Phi_t^{\fid}} \begin{pmatrix} u & v \\ \ol{v} & -u \end{pmatrix}
=
r
\begin{pmatrix}
16 \cosh(2 \ell_t(r))  u & 8 (\cosh(2 \ell_t(r)) v - e^{-i\theta} \ol{v}) \\
8(\cosh(2 \ell_t(r)) \ol{v} - e^{i\theta} v) & -16 \cosh(2 \ell_t(r)) u
\end{pmatrix},
\]
where $\ell_t$ solves
\[
\frac{d^2}{dr^2} \ell_t  + \frac 1 r \frac d {dr} \ell_t = 8 t^2 r \sinh(2 \ell_t)
\]
with $\ell_t(r) \sim -\frac 1 2 \log r$ as $r \to 0$ and $\ell_t(r) \sim \frac 1 \pi K_0 \left( \frac 8 3 t r^{3/2} \right)$ as $r \to \infty$, where $K_0$ is a Bessel function. The function $F_t$ is defined as $F_t = \frac 1 2 \left( \frac 1 2 + r \partial_r \ell_t \right)$.
\end{lemma}

\begin{lemma}[Local form near the core loops]
\label{lemma:local_form_near_core_loop_finite_t}
Let $(A_t^{\model}, \Phi_t^{\model})$ be the model solution given in equation \eqref{eqn:model} acting on the trivial rank $2$ bundle over $\{z = x + i y \in \C : x \neq 0\} /2\pi i \Z$  equipped with the hyperbolic metric $g = x^{-2} (dx^2 + dy^2)$. The operator $\Delta_{A_t}$ induced on $i\mf{su}(2)$ is given by
\[
\Delta_{A_t^{\model}} \begin{pmatrix} u & v \\ \ol{v} & - u \end{pmatrix}
=
\begin{pmatrix}
\Delta_g u & (\Delta_g + 2 i p(tx) x \partial_y + p(tx)^2) v \\
(\Delta_g - 2 i p(tx) x \partial_y + p(tx)^2) \ol{v} & - \Delta_g u
\end{pmatrix}
\]
and
\[
-i \ast M_{\Phi_t^{\model}} \begin{pmatrix} u & v \\ \ol{v} & - u \end{pmatrix}
=
\begin{pmatrix}
2 x^2 E(tx) u & - 2 x^2 \ol{v} + x^2 E(tx) v \\
- 2 x^2 v +  x^2 E(tx) \ol{v} & - 2 x^2 E(tx) u
\end{pmatrix},
\]
where $E(s) = \tanh(s)^2 + \tanh(s)^{-2}$  and $p(s) = \frac{2 s}{\sinh(2s)}$.
\end{lemma}
\begin{proof}
For the purpose of this proof, set $\nabla := d_{A_t^{\model}}$. Then
\[
\nabla \gamma = d \gamma +  \left[ \frac 1 {2 x} p(tx) \begin{pmatrix} -i & 0 \\ 0 & i \end{pmatrix} dy, \gamma \right].
\]
In particular,
\[
\nabla_{\partial_x} \begin{pmatrix} u & 0 \\ 0 & -u \end{pmatrix} = \begin{pmatrix} \partial_x u & 0 \\ 0 & - \partial_x u \end{pmatrix}, \quad \nabla_{\partial_y} \begin{pmatrix} u & 0 \\ 0 & -u \end{pmatrix} = \begin{pmatrix} \partial_y u & 0 \\ 0 & - \partial_y u \end{pmatrix},
\]
\[
\nabla_{\partial_x} \begin{pmatrix} 0 & v \\ \ol{v} & 0 \end{pmatrix} = \begin{pmatrix} 0 & \partial_x v \\ \partial_x \ol{v} & 0 \end{pmatrix}, \quad \nabla_{\partial_y} \begin{pmatrix} 0 & v \\ \ol{v} & 0 \end{pmatrix} = \begin{pmatrix} 0 & \partial_y v \\ \partial_y \ol{v} & 0 \end{pmatrix} + \frac 1 x p(tx) \begin{pmatrix} 0 & -i v \\ i \ol{v} & 0 \end{pmatrix}.
\]
Using the expression
\[
\Delta_{A_t^{\model}} \gamma = -x^2\left(\nabla_{\partial_x}^2 \gamma + \nabla_{\partial_y}^2 \gamma \right),
\]
the expression for $\Delta_{A_t^{\model}}$ follows.

To improve the legibility, let us set $\Phi = \Phi_t^{\model}$ and recall that $\Phi = \Phi_t^{\model} = \frac 1 2 \begin{psmallmatrix} 0 & \tanh(tx)^{-1} \\ \tanh(tx) & 0 \end{psmallmatrix} dz$. Then $\Phi^* = \frac 1 2 \begin{psmallmatrix} 0 & \tanh(tx) \\ \tanh(tx)^{-1} & 0 \end{psmallmatrix} d\ol{z}$.

Then an explicit computation shows that for $\gamma = \begin{psmallmatrix} u & v \\ \ol{v} & -u \end{psmallmatrix}$ the identities
\[
[\Phi^* \wedge [\Phi, \gamma]] = \frac 1 2 \begin{pmatrix}  \left( \tanh(tx)^2 + \tanh(tx)^{-2} \right) u &  -\ol{v} + \tanh (tx)^2 v \\
-v + \tanh(tx)^{-2} \ol{v} & - \left( \tanh(tx)^2 + \tanh (tx)^{-2} \right) u \end{pmatrix} d\ol{z} \wedge dz,
\]
\[
[\Phi \wedge [\Phi^*, \gamma]] = \frac 1 2 \begin{pmatrix} \left( \tanh(tx)^2 + \tanh(tx)^{-2} \right) u &  -\ol{v} +  \tanh(tx)^{-2} v \\
-v + \tanh(tx)^{2} \ol{v} & - \left( \tanh(tx)^2 + \tanh(tx)^{-2} \right) u \end{pmatrix} dz \wedge d\ol{z}
\]
hold. Using that $\ast (d\ol{z} \wedge dz) = 2 i x^2$, this yields the expression for $-i \ast M_{\Phi_t^{\model}}$.
\end{proof}

\subsubsection{Fredholm property}
In this section we consider $L_t$ as a $0$-operator and show that it is a Fredholm operator on weighted spaces due to the theory of $0$-elliptic operators in Section \ref{sec:0-calculus}.

Let us first observe that the connection Laplacian $\Delta_{A_t}$ is $0$-elliptic, because its symbol is essentially that of the scalar Laplacian. Therefore, to see that $L_t$ is indeed Fredholm, we need to analyze the normal operator and the indicial roots of $\Delta_{A_t}$.

According to Lemma \ref{lemma:local_form_near_core_loop_finite_t} the operator $L_t = \Delta_{A_t^{\model}} - t^2 i \ast M_{\Phi_t^{\model}}$ can be written as
\[
L_t \begin{pmatrix} u & v \\ \ol{v} & -u \end{pmatrix}
=
\begin{pmatrix} L_t^1 u & L_t^2 v \\
\ol{L_t^2 v} & - L_t^1 u
\end{pmatrix}
\]
where
\[
L_t^1 u = -x^2(\partial_x^2 + \partial_y^2) u + 2 t^2 x^2 E(tx) u
\]
and
\[
L_t^2 v = -x^2(\partial_x^2 + \partial_y^2) v + 2 i p(tx) x \partial_y v + p(tx)^2 v - 2 t^2 x^2 \ol{v} + t^2x^2 E(tx) v.
\]
The normal operators are obtained by setting $x = 0$ in the coefficients of the operators $(x\partial_x)^j$ and $(x\partial_y)^j$. Observing that $(t^2 x^2 E(tx))|_{x=0} = 1$ and $p(tx)|_{x=0} = 1$, we obtain
\[
N(L_t^1)u = -x^2 \partial_x^2u - x^2 \partial_y^2u + 2u = -(x\partial_x)^2u + x\partial_xu - (x\partial_y)^2u + 2 u,
\]
\[
N(L_t^2)v = -x^2 \partial_x^2v - x^2 \partial_y^2v + 2 i x \partial_y v + 2v.
\]
The indicial families are therefore
\[
I_s(L_t^1) = I_s(L_t^2) = -s^2 + s + 2.
\]
Thus $-1$ and $2$ are the indicial roots of $L_t^1$ and $L_t^2$.

If $N(L_t) : L^{2,2}_0 \to L^2_0$ is an isomorphism, then according to Theorem \ref{thm:fredholm-0-calculus} the operator $L_t$ is a Fredholm operator in the weight range determined by the indicial roots. Propositions \ref{prop:no_harmonic_on_halfspace} and \ref{prop:no_twisted_harmonic_on_halfspace} show that $N(L_t)$ has no kernel on $L^2_0$. Since these operators are self-adjoint, this already implies that their cokernel also vanishes. This suffices to conclude the Fredholm property.

In fact, these operators are Banach space isomorphisms, as the next theorem shows.
\begin{thm}
\label{thm:banach_space_iso}
For finite $t$, $\delta \in \left( -1, 2\right)$ and any $\nu \in \R$ the operators
\[
L_t : \rho^{\delta-1/2} L_0^{2,2} \to \rho^{\delta - 1/2} L^2_0
\]
and
\[
L_t : \rho^{\delta} r_t^\nu C^{k+2,\alpha}_t \to \rho^\delta r_t^\nu C^{k,\alpha}_t
\]
are Banach space isomorphisms.
\end{thm}
\begin{proof}
The previous discussion together with Theorem \ref{thm:fredholm-0-calculus} shows that these operators are Fredholm of index zero and that all their kernels are identical. In particular, it suffices to show that the kernel of $L_t : L_0^{2,2} \to L^2_0$ is trivial. If $\gamma \in L_0^{2,2}$ satisfies $L_t \gamma = 0$, then we have
\[
\langle L_t \gamma, \gamma \rangle_{L^2_0} = \|d_{A_t} \gamma\|_{L^2_0}^2 + t^2 \|[\Phi_t, \gamma]\|_{L^2_0}^2 = 0.
\]
In particular, $d_{A_t} \gamma \equiv 0$. This implies $|\gamma|$ is constant. Since $\gamma \in L^{2,2}_0$ and since $M^\vee$ has infinite volume, this implies that $\gamma \equiv 0$.
\end{proof}

\subsubsection{Schauder estimate}
Theorem \ref{thm:fredholm-0-calculus} implies that for every $t$ we have an estimate \[\|\gamma\|_{\rho^{\delta} r_t^\nu C^{k+2,\alpha}_t} \leq C \|L_t \gamma\|_{\rho^\delta r_t^\nu C^{k,\alpha}_t}\] with a constant that depends on $t$. For our application it is crucial that we understand how the constant depends on $t$.
To this end, we first prove the following weaker Schauder estimate.
\begin{thm}
\label{thm:schauder_estimate}
Given $\delta, \nu \in \R$ there exist constants $C, T, \kappa > 0$, such that for every $t > T$ and every $\gamma \in \rho^\delta r_t^\nu C^{k+2,\alpha}_t$ the inequality
\[
\|\gamma\|_{\rho^\delta r_t^\nu C_t^{k+2, \alpha}} \leq C \|L_t \gamma\|_{\rho^\delta r_t^{\nu-2} C_t^{k,\alpha}} + C t^\kappa \|\gamma\|_{\rho^\delta r_t^\nu C^0}.
\]
\end{thm}
\begin{proof}
The local dependence on coefficients in the Schauder estimates can be stated as follows. If $a_{ij}, b_i, c : B_1(0) \subset \R^n \to \R$ are $\alpha$-H\"older continuous and there are constants $0 < \lambda \leq \Lambda < \infty$, such that $\lambda |\xi|^2 \leq \sum_{ij} a_{ij} \xi_i \xi_j \leq \Lambda |\xi|^2$, then there exist constants $C_1, C_2, r_0, \nu > 0$, such that for every $u \in C^2(B_{2r_0})$ the following inequality holds:
\[
\|u\|_{C^{2,\alpha}(B_{r_0})} \leq C_1 [f]_{\alpha, B_{2 r_0}} + C_2 \left( 1 + \sum_i \|b_i\|_{C^{0,\alpha}} + \|c\|_{C^{0,\alpha}}\right)^2 \|u\|_{C^0},
\]
where $C_1$ and $C_2$ depend only on $n, \alpha, \lambda, \Lambda$, and $[a_{ij}]_{\alpha}(0)$.
This follows from the Schauder estimate for constant coefficient operators via freezing the coefficients and interpolation inequalities, which are responsible for the exponent in the dependence on $b_i$ and $c$.

The estimate in the theorem will be reduced to showing that there exist  constants $C, r > 0$ such that for any $p \in M^\vee$ and any $\gamma \in \rho^\delta C^{2,\alpha}_0$ we have
\begin{equation}
\tag{$\ast$}
\|\gamma\|_{C^{2,\alpha}_0(B_r(p))} \leq C \|L_t \gamma\|_{C^{0,\alpha}_0(B_{2r}(p))} + C t^4  \|\gamma\|_{C^0_0(B_{2r}(p))}.
\end{equation}
Here the balls are understood with respect to the complete, conformally compact metric $g_0$ on $M^\vee$ introduced above. The global, weighted version then follows from the following way to compute the weighted norm on $M^\vee$: the norm $\|\gamma\|_{\rho^\delta C^{k,\alpha}_0}$ is equivalent to the norm
\[
\sup_{p \in M} \rho(p)^{-\delta} \|u\|_{C^{2,\alpha}_0(B_r(p))}.
\]
First, observe that away from the zeros the connections $A_t^{\appr}$ uniformly converge on $M^\vee$ to a limiting connection $A_\infty$, i.e.\@ for every $r_0 > 0$ we have that
\[
\|A_t^{\appr} - A_\infty\|_{C^k_0 (M^\vee \bs B_{r_0}(Z))} \to 0
\]
as $t \to \infty$.

For sufficiently small, but positive, $r < \inj(M^\vee, g_0)$ we may choose a unitary trivialization of the bundle $E$ over every $B_r(p)$, such that in normal coordinates and with respect to the trivialization $\Delta_{A_\infty}$ is as close to the Euclidean Laplace operator as we wish. Since $A_t \to A_\infty$, the same may be assumed for sufficiently large $t$.

Therefore, when applying the local Schauder estimate the constants $\lambda$, $\Lambda$ and $[a_{ij}]_\alpha(0)$ vary in a bounded manner over $M^\vee \bs B_{r_0}(Z)$. On the other hand, it is obvious that if we write the zeroth order operator $t^2 M_{\Phi_t^{\appr}}$ locally, its $C^{0,\alpha}$ bound is bounded by some $C t^2$ for some $C > 0$ over all of $M^\vee$. Therefore, the local estimate yields inequality $(\ast)$ for $p \in M^\vee \bs B_{r_0}(Z)$.

Near a zero we have the explicit local form
\begin{equation*}
\begin{split}
&L_t \begin{pmatrix} u & v \\ \ol{v} & -u \end{pmatrix}
=\\
&\begin{pmatrix}
\Delta u + 16 t^2 r \cosh(2 \ell_t(r)) u & \Delta v - 4 F_t i r^{-2} \partial_\theta v + 4 F_t^2 r^{-2} v + 8 t^2 r \left( \cosh(2 \ell_t(r)) v - e^{-i \theta} \ol{v} \right) \\
* & *
\end{pmatrix}.
\end{split}
\end{equation*}
(The lower row is determined by the fact that it is trace-free and Hermitian.) It is obvious that the highest order term satisfies the required condition. For the lower order terms this boils down to an explicit calculation using the expansions of $\ell_t(r)$ as $r \to 0$ and $r \to \infty$.
\end{proof}

\subsubsection{$C^0$-estimate of the linearization}
The goal of the next sections is to prove the following uniform $C^0$-estimate for the operators $L_t$.

\begin{thm}[Theorem \ref{thm:C0_estimate} below]
Suppose $\delta \in (1/2,1)$ and $\nu = 0$.
  There exist constants $C, T > 0$, such that for any $t > T$ and any $\gamma \in \rho^\delta r_t^\nu C^0 \cap C^2_{\mathrm{loc}}$ the following inequality holds
  \[
  \|\gamma\|_{\rho^\delta r_t^\nu C^0} \leq C \|L_t \gamma \|_{\rho^\delta r_t^{\nu-2} C^0}.
  \]
\end{thm}

Together with the Schauder estimate this estimate implies the following $t$-dependent bound on the inverse of the operators $L_t$.
\begin{thm}
  \label{thm:inverse_bounds_hoelder}
Given $k \geq 0$, $\delta \in \left( 1/2, 1\right)$ and $\nu = 0$ there are constants $C, T, \kappa > 0$, such that
\[
\|\gamma\|_{\rho^\delta r_t^\nu C^{k+2,\alpha}_t} \leq C t^\kappa \|L_t \gamma\|_{\rho^\delta r_t^{\nu-2} C^{k,\alpha}_t}
\]
for every $t \geq T$ and $\gamma \in \rho^\delta r_t^\nu C^{k+2,\alpha}_t$.
\end{thm}
\begin{proof}
Let $\gamma \in \rho^\delta r_t^\nu C^{k+2,\alpha}_t$. Then by Theorem \ref{thm:schauder_estimate}
\[
\|\gamma\|_{\rho^\delta r_t^\nu C^{k+2,\alpha}_t} \leq C \|L_t \gamma\|_{\rho^\delta r_t^{\nu-2} C^{k,\alpha}_t} + C t^\kappa \|\gamma\|_{\rho^\delta r_t^\nu C^0}.
\]
By Theorem \ref{thm:C0_estimate}
\[
\|\gamma\|_{\rho^\delta r_t^\nu C^0} \leq C \|L_t \gamma\|_{\rho^\delta r_t^{\nu-2} C^0}
\]
and this implies the theorem.
\end{proof}
The proof of the $C^0$ estimate is a proof by contradiction. We outline the basic idea here. If the estimate is false, then there exists a sequence $\gamma_n $ and $t_n$, such that
\[
\|\gamma_n\|_{\rho^\delta r_{t_n}^\nu C^0} = 1, \qquad \|L_{t_n} \gamma_n \|_{\rho^\delta r_{t_n}^{\nu-2} C^0} \to 0.
\]
This suggests that, if we could pass to a limit $\gamma_n \to \gamma_\infty$ and $t_n \to t_\infty$, then in the limit we would obtain a solution $\gamma_\infty$ satisfying $L_{t_\infty} \gamma_\infty = 0$. If we then can show that $\gamma_\infty$ is non-trivial and $L_{t_\infty}$ has trivial kernel, this would lead to a contradiction and confirm the estimate. There are two main difficulties with this approach. The first is that the underlying manifold is non-compact. This means that even if $\gamma_n$ converges to some $\gamma_\infty$, the limit can be trivial, because the support of $\gamma_n$ may wander off to infinity. This will be addressed by tracking the support of $\gamma_n$. If it does wander off to infinity, we rescale the $\gamma_n$ around its support to obtain a non-trivial solution on the model space of the infinity, which in our case is just $\H^2$. This is a blow up procedure and a standard strategy. 
The other issue is that the operators $L_t = \Delta_{A_t} - t^2 i \ast M_{\Phi_t}$ diverge as $t \to \infty$. Dealing with this issue is more delicate and requires us to also understand the zeroth order operator $M_{\Phi_t}$ in more detail. The bundle endomorphism $M_{\Phi_\infty}$ induces a 1-dimensional kernel bundle $K$. It turns out that $L_t$ does converge to a limiting operator if we restrict the operator to the kernel bundle $K$. This means that, on this subbundle, we can again extract a non-trivial sublimit, which lies in the kernel of the limiting operator. On the other hand, consider $t^{-2} L_t = - i \ast M_{\Phi_t} + t^{-2} \Delta_{A_t}$ restricted to the bundle $K^\perp$. On $K^\perp$ the endomorphism $-i \ast M_{\Phi_\infty}$ acts as $16 \id_{K^\perp}$. Therefore $t^{-2} L_t$ can be considered as a ``small'' perturbation of the identity. Note however, that the perturbation of the identity is by a second order differential operator. Such problems are studied in the field of semiclassical analysis. We follow a particular approach used in \cite{BGIM} to solve a similar problem.

There are then two main aspects in the proof: justifying and explaining the limit procedures and proving vanishing theorems for the limit operators. This will occupy the next few sections.

\subsubsection{Splitting}
Let $(A_\infty, \Phi_\infty)$ be the chosen limiting configuration and $h_0$ our background metric. Let us define the real line subbundle $K \subset i \mf{su}(E)$ fiberwise via
\[
K_p = \{ L \in i \mf{su}(E) : [L, \Phi_\infty(p)] = 0 \} = \ker M_{\Phi_\infty}(p).
\]
Let $\iota_K : K \to i\mf{su}(E)$ be the inclusion and $\pi_K : i \mf{su}(E) \to K$ be the orthogonal projection. Note that $-i \ast M_{\Phi_\infty}$ is self-adjoint and therefore the subbundles $K$ and $K^\perp$ are invariant subbundles.

We now describe the local forms of the kernel bundle, of the orthogonal projection and of the operator $-i \ast M_{\Phi_\infty}$ restricted to $K^\perp$. There are two cases to consider. Away from the zeros of the quadratic differential $\Phi_\infty$ has the local form $\begin{pmatrix} 0 & 1 \\ 1 & 0 \end{pmatrix} dz$ in the chosen coordinates and unitary frames.

\begin{lemma}
\label{lemma:K_local_form_interior}
Consider the trivial rank $2$ bundle over $\C$ equipped with the Euclidean metric and let $\Phi_\infty = \begin{pmatrix} 0 & 1 \\ 1 & 0 \end{pmatrix} dz$. Then the kernel of $-i\ast M_{\Phi_\infty}$ acting on $i \mf{su}(2)$ is given by
\[
\left\{ \begin{pmatrix} 0 & t \\ t & 0 \end{pmatrix} : t \in \R \right\}.
\]
The projections are given by
\[
\pi_K \begin{pmatrix} u & v \\ \ol{v} & -u \end{pmatrix} = \begin{pmatrix} 0 & \frac 1 2 (v + \ol{v}) \\ \frac 1 2 (v + \ol{v}) & 0 \end{pmatrix}
\]

\[
\pi_{K^\perp} \begin{pmatrix} u & v \\ \ol{v} & -u \end{pmatrix} = \begin{pmatrix} u & \frac 1 2 (v - \ol{v}) \\ \frac 1 2 (-v + \ol{v}) & - u \end{pmatrix}.
\]
The restriction of $-i \ast M_{\Phi_\infty}$ to $K^\perp$ is given by
\[
(-i \ast M_{\Phi_\infty})|_{K^\perp} = 16 \id_{K^\perp}.
\]
\end{lemma}

\begin{lemma}
Let $(A_\infty^{\fid}, \Phi_\infty^{\fid})$ be the limiting fiducial solution given in \eqref{eqn:limiting_fid} acting on the trivial rank $2$ bundle over $\C$ equipped with the Euclidean metric. Then the kernel of $-i\ast M_{\Phi_\infty^{\fid}}$ acting on $i \mf{su}(2)$ at the point with polar coordinates $(r, \theta)$ is given by 
\[
\left\{ \begin{pmatrix} 0 & z \\ \ol{z} & 0 \end{pmatrix} : z \in \C \text{ with } z = e^{-i\theta} \ol{z} \right\}.
\]
The projections are given by
\[
\pi_K \begin{pmatrix} u & v \\ \ol{v} & -u \end{pmatrix} = \begin{pmatrix} 0 & \frac 1 2 (v + e^{-i\theta} \ol{v}) \\ \frac 1 2 (\ol{v} + e^{i\theta} v) & 0 \end{pmatrix},
\]
\[
\pi_{K^\perp} \begin{pmatrix} u & v \\ \ol{v} & -u \end{pmatrix} = \begin{pmatrix} u & \frac 1 2 (v - e^{-i\theta} \ol{v}) \\ \frac 1 2 (\ol{v} - e^{i\theta} v) & - u \end{pmatrix}.
\]
The restriction of $-i \ast M_{\Phi_\infty^{\fid}}$ to $K^\perp$ is given by
\[
(-i\ast M_{\Phi^{\fid}_\infty})|_{K^\perp} = 16 \id_{K^\perp}.
\]
\end{lemma}
The previous two lemmas can be proved by direct calculation.

\subsubsection{Behavior of the operators $L_t$ and $\Delta_{A_t}$ as $t \to \infty$}
In this section we analyze the behavior of $L_t$ as $t\to \infty$. We will write $(A_t,\Phi_t)= (A_t^{\appr},\Phi_t^{\appr})$ for short. As mentioned earlier it is instrumental to consider the behavior on the kernel bundle $K$ and its orthogonal complement $K^\perp$ separately. This is because of the behavior of the term $-t^2 i \ast M_{\Phi_t}$. Let us first consider the connection Laplacian $\Delta_{A_t}$. As usual we examine the behavior in the model regions separately. Let us first observe that in the interior $A_t = A_\infty$ and therefore $\Delta_{A_t} = \Delta_{A_\infty}$.

\begin{lemma}
\label{lemma:convergence_Delta_A_t_near_zeros}
Consider the fiducial solution $(A_t^{\fid}, \Phi_t^{\fid})$ on $\C$. For $\epsilon > 0$ the coefficients of $\Delta_{A_t^{\fid}}$ converge uniformly to those of $\Delta_{A_\infty^{\fid}}$ as $t \to \infty$ on $\{|z| > \epsilon\}$. Likewise, all the derivatives of these coefficients converge uniformly on $\{|z| > \epsilon\}$. 
\end{lemma}
\begin{proof}
The function $\ell_t(r)$ has the scaling behavior $\ell_t(r) = \ell_1\left(t^{2/3} r\right)$. On the other hand, $\ell_1(r) \to 0$ as $r \to \infty$ and likewise for all derivatives. Therefore, for any $\epsilon > 0$ the function $\ell_t(r)$ converges uniformly to $0$ on $\{r > \epsilon\}$. The same is true for all derivatives.

This implies that $\cosh(2 \ell_t(r))$ uniformly converges to $1$ on $\{r > \epsilon\}$. Similarly, $F_t = \frac 1 2 \left( \frac 1 2 + r\partial_r \ell_t\right)$ converges uniformly to $\frac 1 4$ on $\{r > \epsilon\}$.

Replacing the appearances of $F_t$ by $\frac 1 4$ in the expressions for $\Delta_{A_t}^{\fid}$, we obtain the expression for $\Delta_{A_\infty}$. This shows that indeed the coefficients of $\Delta_{A_t^{\fid}}$ converge uniformly to those of $\Delta_{A_\infty^{\fid}}$ and likewise for the derivatives.
\end{proof}

\begin{lemma}
\label{lemma:convergence_Delta_A_t_near_core_loop}
Consider the model solution $(A_t^{\model}, \Phi_t^{\model})$ on $\{z \in \C : \Real z \neq 0\}$. For $\epsilon > 0$ the coefficients of $\Delta_{A_t^{\model}}$ converge uniformly to those of $\Delta_{A_\infty}$ as $t \to \infty$ on $\{|\Real z| > \epsilon\}$.  Likewise, all the derivatives of these coefficients converge uniformly on $\{|\Real z| > \epsilon\}$.
\end{lemma}
\begin{proof}
The operators $\Delta_{A_t^{\model}}$ and $\Delta_{A_\infty^{\model}}$ coincide on the diagonal. On the off-diagonal term, the operator $\Delta_{A_t^{\model}}$ is $\Delta_g + 2 i p(tx) x \partial_y + p(tx)^2$. Now $p(s) = \frac{2s}{\sinh(2s)}$ and all its derivatives converge to $0$ as $s \to \infty$. Therefore on $\{x > \epsilon\}$ the function $p(tx)$ and its derivatives converge uniformly to $0$ as $t\to \infty$.
\end{proof}

\begin{lemma}
\label{lemma:convergence_M_Phi_t_K_near_zeros}
Consider the fiducial solution $(A_t^{\fid}, \Phi_t^{\fid})$ on $\C$. For $\epsilon > 0$ the operator $(- t^2 i \ast M_{\Phi_t^{\fid}})|_K$ converges uniformly to $0$ on $\{|z| > \epsilon\}$. Likewise, all the derivatives of  $(- t^2 i \ast M_{\Phi_t^{\fid}})|_K$ converge uniformly on $\{|z| > \epsilon\}$.
\end{lemma}
\begin{proof}
A section $\gamma$ in the kernel can be written as $\begin{pmatrix} 0 & v \\ \ol{v} & 0 \end{pmatrix}$ where $v = e^{-i \theta} \ol{v}$. Then
\[
-i \ast M_{\Phi_t^{\fid}} \gamma
=
\begin{pmatrix}
0 & 8 r \left(\cosh(2 \ell_t(r)) v - e^{-i\theta} \ol{v} \right) \\
8 r \left(\cosh(2 \ell_t(r)) \ol{v} - e^{i\theta} v \right) & 0
\end{pmatrix}
.
\]
Writing $\cosh(2 \ell_t(r)) = (\cosh(2 \ell_t(r)) - 1) + 1$ and using the condition $v = e^{-i \theta} \ol{v}$ this can be rewritten as
\[
-i \ast M_{\Phi_t^{\fid}} \gamma
=
\begin{pmatrix}
0 & 8 r \left(\cosh(2 \ell_t(r)) - 1 \right) v \\
8 r \left(\cosh(2 \ell_t(r)) - 1 \right) \ol{v} & 0
\end{pmatrix}
.
\]
In the proof of Lemma \ref{lemma:convergence_Delta_A_t_near_zeros} we saw that $\ell_t$ converges uniformly to $0$ as $t \to \infty$ on $\{r > \epsilon\}$ for any $\epsilon > 0$. In fact, using the asymptotics of $\ell_t \sim \frac 1 \pi K_0\left( \frac 8 3 t r^{3/2} \right)$, we can say that this converges faster than any power of $t$. In particular it follows that $-i t^2 \ast M_{\Phi_t^{\fid}}|_K$ converges uniformly to $0$ as $t \to \infty$, including all derivatives.
\end{proof}

\begin{lemma}
\label{lemma:convergence_M_Phi_t_K_near_core_loop}
Consider the model solution $(A_t^{\model}, \Phi_t^{\model})$ on $\{z \in \C : \Real z \neq 0\}$. For $\epsilon > 0$ the operator $(- t^2 i \ast M_{\Phi_t^{\model}})|_K$ converges uniformly to $0$ on $\{|\Real z| > \epsilon\}$ as $t\to \infty$.  Likewise, all the derivatives of  $(- t^2 i \ast M_{\Phi_t^{\model}})|_K$ converge uniformly on $\{|\Real z| > \epsilon\}$.
\end{lemma}
\begin{proof}
Sections of $K$ in the unitary frame of $\{z \in \C : \Real z > 0\}/2\pi i \Z$ are given as $\begin{pmatrix} 0 & v \\ v & 0 \end{pmatrix}$ with $v$ real, i.e.\@ $v = \ol{v}$. The action of $-t^2 i \ast M_{\Phi_t^{\model}}$ on such a section is therefore
\[
- t^2 i \ast M_{\Phi_t^{\model}} \begin{pmatrix} 0 & v \\ v & 0 \end{pmatrix} = \begin{pmatrix} 0 & t^2x^2(E(tx) - 2) v \\ t^2x^2(E(tx) - 2) v & 0 \end{pmatrix}.
\]
The function $s^2 (E(s)-2)$ can also be written as
\[
16 s^2 \frac 1{(1-e^{-4s})^2} e^{-4s}.
\]
Clearly, this function and all its derivatives converge to $0$ as $s\to \infty$. Therefore for $x > \epsilon$, the function $t^2 x^2 (E(tx) - 2)$ and its derivatives converge uniformly to $0$ as $t\to \infty$.
\end{proof}

It will also be necessary to consider the commutators $[\pi_K, L_t]$. Observe that in the interior region $[\pi_K, L_t] = 0$, since there $\Delta_{A_t} = \Delta_{A_\infty}$ and $-i \ast M_{\Phi_t} = - i \ast M_{\Phi_\infty}$ and these operators clearly commute with the projection $\pi_K$. The following two lemmas compute the commutator for the fiducial and the model solutions.

\begin{lemma}
\label{lemma:convergence_commutator_zeros}
Consider the fiducial solution $(A_t^{\fid}, \Phi_t^{\fid})$ on $\C$. The commutator $[\pi_K, -i t^2 \ast M_{\Phi_t^{\fid}}]$ vanishes. The coefficients (and all their derivatives) of the operator $[\pi_K, \Delta_{A_t^{\fid}}]$ converge uniformly to $0$ on $\{r > \epsilon\}$ for every $\epsilon > 0$. Therefore the same holds also for $[\pi_K, L_t]$.
\end{lemma}
\begin{proof}
An explicit computation yields
\[
- \pi_K it^2 \ast M_{\Phi_t^{\fid}} \begin{pmatrix} u & v \\ \ol{v} & -u \end{pmatrix}
=
8 t^2 r
\begin{pmatrix}
0 & \left(\cosh(2 \ell_t(r))- 1\right) (v + e^{-i\theta} \ol{v}) \\
\left(\cosh(2 \ell_t(r)) - 1\right) (v + e^{-i\theta} \ol{v}) & 0
\end{pmatrix}.
\]
We had already seen in Lemma \ref{lemma:convergence_M_Phi_t_K_near_zeros} that
\[
-i t^2 \ast M_{\Phi_t^{\fid}}|_K \begin{pmatrix} 0 & v \\ \ol{v} & 0 \end{pmatrix}
=
8 t^2 r
\begin{pmatrix}
0 & \left(\cosh(2 \ell_t(r))- 1\right) v \\
\left(\cosh(2 \ell_t(r)) - 1\right) \ol{v} & 0
\end{pmatrix}.
\]
This shows that $-\pi_K it^2 \ast M_{\Phi_t^{\fid}} = -it^2 \ast M_{\Phi_t^{\fid}} \pi_K$ as claimed.

When computing $[\pi_K, \Delta_{A_t^{\fid}}] \begin{pmatrix} u & v \\ \ol{v} & - u\end{pmatrix}$ it is easy to see that the diagonal terms vanish. The off diagonal term is
\[
r^{-2} \left( 1 - 4 F_t \right) e^{-i \theta} \ol{v} + r^{-2} \left(1 - 4 F_t \right) 2 i e^{-i\theta} \ol{v}.
\]
From the asymptotics of $\ell_t$ it is easy to see that $F_t = \frac 1 2 \left( \frac 1 2 + r \partial_r \ell_t \right)$ converges uniformly to $\frac 1 4$ on $\{r > \epsilon\}$ for any $\epsilon > 0$, and this proves the claim.
\end{proof}

\begin{lemma}
\label{lemma:convergence_commutator_near_core_loop}
Consider the model solution $(A_t^{\model}, \Phi_t^{\model})$ on $\{z \in \C : \Real z \neq 0\}$. Then $[\pi_K, -i\ast M_{\Phi_t^{\model}}] = 0$ and
\[
[\pi_K, L_t] = [\pi_K, \Delta_{A_t^{\model}}] \begin{pmatrix} u & v \\ \ol{v} & -u \end{pmatrix} =
\begin{pmatrix} 0 & - 2 i p(tx) x\partial_y \ol{v} \\
2 i p(tx) x\partial_y v & 0 
\end{pmatrix}
.
\]
In particular, the coefficients of $[\pi_K, L_t]$ and their derivatives converge uniformly to $0$ on $\{|\Re z| > \epsilon \}$ for every $\epsilon > 0$.
\end{lemma}
\begin{proof}
The explicit expression follows from an easy computation. We have already seen that the coefficient $p(tx)$ converges uniformly, including all derivatives, on $\{x > \epsilon\}$ as $t\to \infty$.
\end{proof}

\subsubsection{Vanishing theorems}
In this section, we establish the vanishing theorems, which form the basis of our blow-up proof.

\begin{prop}
\label{prop:no_harmonic_on_halfspace}
Let $\delta \in \left(-\frac 1 2 \sqrt{1 + 4 \lambda}, \frac 1 2 \sqrt{1 + 4 \lambda}\right)$. If $u \in \rho^\delta L^2(\H^2)$ is a weak solution of $\Delta u + \lambda u = 0$ for some $\lambda \geq 0$, then $u = 0$.
\end{prop}
\begin{proof}
First, assume $u \in L^2_0(\H^2)$. By elliptic regularity (Theorem \ref{thm:0-elliptic-regularity}), the equation $\Delta u + \lambda u = 0$ implies that $u \in L^{2,2}_0$. A standard argument shows that we may apply integration by parts: choose a sequence of smooth, compactly supported functions $f_n : \H^2 \to [0,1]$ such that the sets $\{ f_n = 1\}$ form an exhaustion of $\H^2$ and such that $\|d f_n\|_{C^0} \leq \frac 1 n$. Then for any $u \in L^{2,2}_0$ we find that
\begin{align*}
\int_{\H^2} \left(\Delta u + \lambda u \right) u \vol_{\H^2}  & = \lambda \|u\|_{L^2_0}^2 + \lim_{n \to \infty} \int_{\H^2} (\Delta u) f_n u \vol_{\H^2} \\
& = \lambda \|u\|_{L^2_0}^2 + \lim_{n\to \infty} \int_{\H^2} \langle d u, d(f_nu) \rangle \vol_{\H^2} \\
& = \lambda \|u\|_{L^2_0}^2 + \int_{\H^2} \langle du, du \rangle \vol_{\H^2}
\end{align*}
where we computed the last limit using that
\[
\left|\langle du, d(f_n u) \rangle - f_n \langle du, du \rangle\right| = |\langle du, u df_n\rangle| \leq \frac 1 n |du| |u|
\]
and the fact that $|du| |u| \in L^1_0$.

Thus $\|du\|_{L^2_0}^2 + \lambda \|u\|_{L^2_0}^2 = 0$. If $\lambda > 0$ this immediately implies $u = 0$. If $\lambda = 0$ this implies that $u$ is constant. The only constant function in $L^2_0$ is the $0$ function.

In half-space coordinates $\Delta = -x^2 \partial_x^2 - x^2 \partial_y^2$. This implies that $\Delta$ is equal to its normal operator. Since $x^2 \partial_x^2 = (x\partial_x)^2 - x\partial_x$, we obtain $I_s(\Delta + \lambda) =  -s^2 + s + \lambda$. Therefore the indicial roots of this operator are $\frac 1 2 \pm \frac 1 2 \sqrt{1 + 4 \lambda}$. By Theorem  \ref{thm:fredholm-0-calculus} this implies that for $\delta \in \left(\frac 1 2 - \frac 1 2 \sqrt{1 + 4 \lambda}, \frac 1 2 + \frac 1 2 \sqrt{1 + 4 \lambda} \right) $ the operators $\Delta : \rho^{\delta - 1/2} L^{2,2}_0 \to \rho^{\delta - 1/2} L^2_0$ all have the same kernel, which implies the statement.
\end{proof}

\begin{prop}
\label{prop:stronger_no_harmonic_on_halfspace}
Suppose $u \in C^{1,\alpha}_{\lcl} (\H^2)$ is a weak solution of $\Delta u = 0$ and suppose that $|u(x,y)| \leq x^\delta$ for some $0 < \delta < 1$. Then $u \equiv 0$.
\end{prop}
\begin{proof}
By the Schwarz reflection principle, the function
\[
\hat{u}(x,y) =
\begin{cases}
u(x,y), & x > 0 \\
0, & x = 0 \\
-u(-x,y), & x < 0
\end{cases}
\]
is a harmonic function on $\R^2$. Since $|x| \leq r := \sqrt{x^2 + y^2}$ we have a harmonic function $\hat{u}$ satisfying $|\hat{u}| \leq r^\delta$ with $0 < \delta < 1$. It is well known that this implies $\hat{u} = 0$ and therefore also $u = 0$.
\end{proof}

\begin{prop}
\label{prop:no_twisted_harmonic_on_halfspace}
Suppose $c \geq 0$. Let $\delta \in (-\sqrt{5+4c}/2, \sqrt{5+4c}/2)$. If $u \in \rho^\delta L^2(\H^2)$ is a weak solution of $\Delta u + 2 i x \partial_y u + u + c u = 0$, then $u = 0$.
\end{prop}
\begin{proof}
First, assume $u \in L^2_0(\H^2)$. By elliptic regularity (Theorem \ref{thm:0-elliptic-regularity}), the equation $\Delta u + 2 i x \partial_y u + u + c u = 0$ implies that $u \in L^{2,2}_0$ and therefore we may use integration by parts as in Proposition \ref{prop:no_harmonic_on_halfspace}. Observe that
\[
\Delta u + 2i x \partial_y u + u = -x^2 \partial_x^2 u - (x\partial_y - i)^2 u.
\]
Since $x\partial_y - i$ is skew-adjoint on $L^2_0$, this shows that
\[
\langle \Delta u + 2i x \partial_y u + u + c u, u \rangle_{L^2_0} = \|x \partial_x u\|^2_{L^2_0} + \|x \partial_y u - i u\|_{L^2_0}^2 + c \|u\|_{L^2_0}^2.
\]
If $c > 0$ we directly conclude from vanishing of this expression that $u \equiv 0$. If $c = 0$ and this expression vanishes, then $\partial_x u = 0$ and $x \partial_y u = i u$. The first equation implies that $u$ is constant on $\{y = c\}$ for every fixed $c \in \R$. Now the second equation implies that $u(x,y) = A e^{i x^{-1} y}$. But then $|u(x,y)| = |A|$. Since $u \in L^2_0$, this implies $u \equiv 0$.

The operator above can also be written as $-(x\partial_x)^2 + x\partial_x - (x\partial_y)^2 + 2 i x \partial_y + 1+c$. Its indicial family is then given by $-s^2 + s + 1 + c$ and therefore its indicial roots are $\frac 1 2 \pm \frac{\sqrt{5+4c}} 2$.

By Theorem \ref{thm:fredholm-0-calculus} this implies that for $\delta \in \left(\frac 1 2 - \frac{\sqrt{5+4c}} 2, \frac 1 2 + \frac{\sqrt{5+4c}} 2 \right)$ the operators $\Delta + 2 i x \partial_y + 1 + c : \rho^{\delta - 1/2} L^{2,2}_0 \to \rho^{\delta - 1/2} L^2_0$ all have the same kernel, which implies the statement.
\end{proof}

\begin{prop}
  \label{prop:trivial_kernel_model_space_diag}
  Suppose $u : \H^2 \to \R$ is a weak solution of
  \[
  \Delta u + 2 \lambda^2 x^2 E(\lambda x) u = 0
  \]
  and suppose that $|u(x,y)| \leq x^\delta$ for some $\delta \in  \left( -1, 2 \right)$.

  Then $u \equiv 0$.
\end{prop}
\begin{proof}
  By elliptic regularity $u$ is smooth. Define $\wt{u} = x^{-\delta} u$ and let $h(x,y) = 2 \lambda^2 x^2 E(\lambda x)$. An easy computation using the facts that $\Delta (uv) = u \Delta v + v \Delta u - 2 \langle d u, d v\rangle_{\H^2}$ and $\Delta x^{-\delta} = -\delta(\delta+1) x^{-\delta}$ shows that
  \[
  \Delta \wt{u} - 2 \delta x \partial_x \wt{u} = [ \delta(\delta-1) - h ] \wt{u}.
  \]
  Note that $h \geq 2$. This implies that for $\delta \in \left( -1, 2 \right)$ the coefficient $\delta(\delta-1) - h$ is negative and uniformly bounded away from zero.

  Since $\wt{u}$ is bounded and smooth, by the Omori--Yau maximum principle there exists a sequence for points $p_k \in \H^2$ with $\wt{u}(p_k) \to \sup \wt{u}$, $|\nabla \wt{u}|(p_k) \to 0$ and $\Delta \wt{u}(p_k) \geq -\frac 1 k$. Applying this to our identity for $\wt{u}$ and passing to the limit, this implies that $\sup \wt{u}$ is non-positive.

  Applying the same reasoning to $-\wt{u}$ also shows that $\inf \wt{u}$ is non-negative. Therefore, $\wt{u}$ (and likewise $u$) vanishes.
\end{proof}

\begin{prop}
  \label{prop:trivial_kernel_model_space_off_diag}
  Suppose $v : \H^2 \to \C$ is a weak solution of
  \[
  \Delta v + 2 i p(\lambda x) x \partial_y v + p(\lambda x)^2 v - 2 \lambda^2 x^2 \ol{v} + \lambda^2 x^2 E(\lambda x) v = 0,
  \]
  and suppose that $|v(x,y)| \leq x^\delta$ for some $\delta \in  (0, 1)$.

  Then $v \equiv 0$.
\end{prop}
\begin{proof}
  The real parts and imaginary parts of $v$ satisfy the following coupled system of equations:
  \begin{align*}
    \Delta \Real v - 2 p(\lambda x) x \partial_y \Imag v + p(\lambda x)^2 \Real v + \lambda^2 x^2 (E(\lambda x) - 2) \Real v &  = 0, \\
    \Delta \Imag v + 2 p(\lambda x) x \partial_y \Real v + p(\lambda x)^2 \Imag v + \lambda^2 x^2 (E(\lambda x) + 2) \Imag v &  = 0.
  \end{align*}
  Define $f = x^{-\delta} \Real v$ and $g = x^{-\delta} \Imag v$. Then an easy but tedious computation yields
  \begin{align*}
    \Delta f - 2 p(\lambda x) x \partial_y g - 2 \delta x \partial_x f + p(\lambda x)^2 f + U f + \delta(1 - \delta) f & = 0, \\
    \Delta g + 2 p(\lambda x) x \partial_y f - 2 \delta x \partial_x g + p(\lambda x)^2 g + U g + \delta(1 - \delta) g & + 4 \lambda^2 x^2 g = 0,
  \end{align*}
  where $U(x) = \lambda^2 x^2 (E(\lambda x) - 2)$. Let $\sigma = f^2 + g^2$. Then
  \begin{align*}
    \Delta \sigma - 2 \delta x \partial_x \sigma & = 4 p(\lambda x) [f x \partial_y g - g x \partial_y f] \\
    & - 2 \left[ p(\lambda x)^2 +  U + \delta(1 - \delta) \right] \sigma - 8 \lambda^2 x^2 g^2 \\
    & - 2 |\nabla f|^2_{\H^2} - 2 |\nabla g|^2_{\H^2}.
  \end{align*}
  Using
  \[
  \left|2 p(\lambda x) [f x \partial_y g - g x \partial_y f]\right| \leq p(\lambda x)^2 \sigma + |x \partial_y f|^2 + |x\partial_y g|^2 \leq p(\lambda x)^2 \sigma + |\nabla f|^2_{\H^2} + |\nabla g|^2_{\H^2},
  \]
  we see that the identity for $\sigma$ implies
  \[
  \Delta \sigma - 2\delta x \partial_x \sigma \leq - 2 [U + \delta(1 - \delta)] \sigma - 8 \lambda^2 x^2 g^2.
  \]
  By assumption $x^{-\delta} v$ is bounded and therefore so are $f$, $g$ and $\sigma$. The Omori--Yau maximum principle supplies us with a sequence $p_k \in \H^2$, such that $\sigma(p_k) \to \sup \sigma$, $|\nabla \sigma (p_k)|_{\H^2} \to 0$, $\Delta \sigma(p_k) \geq - \frac 1 k$.

  Applying this to our inequality, we obtain
  \[
  -\frac 1 k - 2 \delta x \partial_x \sigma(p_k) \leq - 2[ U(p_k) + \delta(1 - \delta)] \sigma(p_k) - 8 \lambda^2 x(p_k)^2 g(p_k)^2.
  \]
  Now $|x \partial_x \sigma(p_k)| \leq |\nabla \sigma(p_k)|_{\H^2}$ converges to $0$. On the other hand since $\delta(1 - \delta) > 0$ for $\delta \in (0,1)$, if $\sup \sigma = \lim_{k\to\infty} \sigma(p_k)$ were positive, the right hand side would eventually be strictly smaller than $-\frac 1 k$. Therefore, $\sup \sigma = 0$. Since $\sigma$ is manifestly non-negative, $\sigma$ has to vanish. But this implies that $v$ has to vanish. This finishes the proof.
\end{proof}

\begin{prop}
\label{prop:no_A_infty_harmonic}
Let $\delta > \frac 1 2$ and let $\nu \geq 0$ be arbitrary. If $\gamma \in \rho^\delta r^\nu C^0_\infty(M^\times)$ satisfies $\Delta_{A_\infty} \gamma = 0$ in the sense of distributions, then $\gamma = 0$.
\end{prop}
\begin{proof}
Under the given weight conditions $\rho^\delta r^\nu C^0_\infty(M^\times)$ is included in $L^2_0(M^\times)$. Fix some $\epsilon > 0$. Multiplying $\gamma$ by a smooth function $f$, which is $1$ on $\{r > \epsilon\}$ and vanishes on $\{r < \epsilon/2\}$, we may apply elliptic regularity (Theorem \ref{thm:0-elliptic-regularity}) to obtain that $f\gamma \in L^{2,2}_0(M^\times)$. In particular, $\gamma \in L^{2,2}_0(\{r > \epsilon\})$.

On the other hand, regularity theory for conic operators says that a solution $\gamma$ of $\Delta_{A_\infty} \gamma = 0$ has a polyhomogeneous expansion near $\{r_\infty = 0\}$. For an account of this theory, see for example \cite[Section~4]{MW}. In the neighborhood of a zero of the quadratic differential, the operator $\Delta_{A_\infty}$ has the form
\[
\Delta_{A_\infty} \begin{pmatrix} u & v \\ \ol{v} & - u \end{pmatrix}
=
\begin{pmatrix}
\Delta u & \Delta_{1/2} v \\ \Delta_{1/2} \ol{v} & -\Delta u
\end{pmatrix}.
\]
The operator on the diagonal is the ordinary Laplacian, and it has indicial roots $k \in \Z$. Therefore $u$ admits an expansion
\[
u \sim \sum_k u_k r^{k} + u_{\log} \log(r),
\]
where the $\log$ term accounts for the fact that the indicial root at $k=0$ is a double root. Since $\gamma \in \rho^\delta r^\nu C^0_\infty(M^\times)$ with $\nu \geq 0$, $\gamma$ is bounded, and therefore no $r^k$ with $k < 0$ or $\log$ term can appear, and we conclude that $u \sim \sum_{k \geq 0} u_k r^k$. On the other hand, the indicial roots of the twisted Laplacian $\Delta_{1/2}$ are $k+\frac 1 2$, $k \in \Z$. Therefore
\[
v \sim \sum_{k \in \Z} v_k r^{k + \frac 1 2}.
\]
Using the fact that $\gamma \in \rho^\delta r^\nu C^0_\infty(M^\times)$ and is bounded, we see that the terms $v_k r^{k+ \frac 1 2}$ with $k < 0$ cannot appear, and therefore $v \sim \sum_{k \geq 0} v_k r^{k+ \frac 1 2}$.

Using this we see that we may integrate by parts in the equation $(\Delta_{A_\infty} \gamma, \gamma)_{L^2_0(M^\times)} = 0$ to obtain $\|d_{A_\infty} \gamma\|_{L^2_0}^2 = 0$. On $\{r > \epsilon\}$ we may argue as in Proposition \ref{prop:no_harmonic_on_halfspace}. For $\{r < \epsilon\}$ we can now show that the boundary term vanishes. More precisely, we want to see that in the equation
\[
\int_{M^\times \bs \{r < \epsilon\}} \langle \Delta_{A_\infty} \gamma, \gamma \rangle \vol = \int_{M^\times \bs \{r < \epsilon\}} \langle d_{A_\infty} \gamma, d_{A_\infty} \gamma \rangle \vol + \int_{\{r = \epsilon\}} \langle d_{A_\infty} \gamma( \partial_r), \gamma \rangle ds
\]
the boundary term converges to $0$. Differentiating the polyhomogeneous expansion, we see that $d_{A_\infty} \gamma$ has an expansion $\sum_{k \geq 0} f_k r^k$ on the diagonal and on the off diagonal $\sum_{k \geq 0} g_k r^{k-\frac 1 2}$. Since the $r^{-1/2}$ term of $d_{A_\infty} \gamma$  is multiplied with the $r^{1/2}$ term of $\gamma$, it follows that $\langle d_{A_\infty} \gamma( \vec{n}), \gamma \rangle$ has an expansion starting at $r^0$. On the other hand, the circle $\{r = \epsilon\}$ has length $2\pi \epsilon$, and therefore the boundary term indeed vanishes as $\epsilon \to 0$. This concludes the justification of the integration by parts.

Now, $d_{A_\infty} \gamma = 0$ implies that $|\gamma|$ is constant. Since $M^\times$ has infinite volume and $\gamma$ is in $L^2_0$, this constant must vanish. Therefore $\gamma = 0$ as claimed.
\end{proof}

\begin{prop}
\label{prop:no_bounded_eigenfunctions}
Suppose $u \in C^0(\R^2)$ is bounded and satisfies $\Delta u + \lambda u = 0$ for some $\lambda > 0$. Then $u = 0$.
\end{prop}
\begin{proof}
Regard $u$ as a tempered distribution. After Fourier transformation, this equation then becomes $(|\xi|^2 + \lambda) \hat{u}(\xi) = 0$. Since $\lambda + |\xi|^2$ is nonzero for every $\xi \in \R^2$, this implies $\hat{u} \equiv 0$, and therefore $u \equiv 0$.
\end{proof}

The solutions of the equation
\[
-u'' - \frac 1 r u' + \frac 1 {r^2} \lambda^2 u = 0
\]
are given by
\[
u = c_1 r^\lambda + c_2 r^{-\lambda}.
\]
The solutions of the equation
\[
-u'' - \frac 1 r u' + \frac 1 {r^2} \lambda^2 u + 16 r u = 0
\]
are given by $u = c_1 h_1 + c_2 h_2$, where
\[
h_1(r) = I_{\frac 2 3 \lambda} \left( \frac 8 3 r^{3/2} \right), \qquad h_2(r) = K_{\frac 2 3 \lambda} \left( \frac 8 3 r^{3/2} \right),
\]
where $I_\nu$ denotes the modified Bessel function solving
\[
x^2 \frac{d^2}{dx^2} I_\nu + x \frac {d}{dx} I_\nu - (x^2 + \nu^2) I_\nu = 0
\]
with asymptotics $I_\nu(x) \sim \frac 1 {\Gamma(\nu+1)} \left(\frac x 2 \right)^\nu$ at $x=0$ and $I_\nu \sim \frac 1 {\sqrt{2 \pi x}} e^x$ at $x=\infty$. Similarly, $K_\nu$ denotes the modified Bessel function satisfying the same equation with asymptotics $K_\nu(x) \sim \frac 1 {\sqrt{2 \pi x}} e^{-x}$ at $x=\infty$.

\begin{lemma}
  \label{lemma:laplacian_bound}
  Let $\nu, \lambda \in \R$ such that $1 + \nu \pm \lambda \neq 0$. Suppose $u,f : \R_+ \to \C$ solve
  \[
  -u'' - \frac 1 r u' + \frac 1 {r^2} \lambda^2 u = f
  \]
  and suppose that $|f(r)| \leq r^\nu$. Then there exist constants $c_1, c_2 \in \C$ such that
  \[
  |u(r) - c_1 r^\lambda - c_2 r^{-\lambda}| \leq C r^{\nu+2},
  \]
  where $C = \max \left\{ \frac 1 {2|\lambda(1+\nu+\lambda)|}, \frac 1 {2 |\lambda(1+\nu-\lambda)|} \right\}$.
\end{lemma}
\begin{proof}
The solution $u$ can be written as the sum of a particular solution and a linear combination of the two solutions $r^\lambda$ and $r^{-\lambda}$ of the homogeneous equation. 

A particular solution $u_p$ can be written as
\[
u_p(r) = A(r) r^\lambda + B(r) r^{-\lambda},
\]
where
\[
A'(r) = \frac 1 {W(r)} r^{-\lambda} f(r), \qquad B'(r) = -\frac 1 {W(r)} r^\lambda f(r).
\]
Here $W(r)$ is the Wronskian of the two solutions of the homogeneous problem, i.e.\@
\[
W(r) = (r^\lambda) (r^{-\lambda})' - (r^\lambda)' (r^{-\lambda}) = - 2 \lambda r^{-1}.
\]
Note that these equations do not determine $u_p$ uniquely, since we may modify $A$ and $B$ by constants.

Now observe that
\[
|A'(r)| = \frac 1 {2 |\lambda|} r r^{-\lambda} |f(r)| \leq \frac 1 {2|\lambda|} r^{1-\lambda + \nu}.
\]
In particular, integrating this estimate shows that $A$ is bounded by $\frac 1 {2|\lambda(1+\nu-\lambda)|} r^{\nu + 2 - \lambda}$. Therefore the term $A(r) r^{\lambda}$ is bounded by $\frac 1 {2|\lambda(1+\nu-\lambda)|} r^{\nu+2}$. A similar calculation shows that $B(r)$ has an antiderivative bounded by $\frac 1 {2|\lambda(1+\nu+\lambda)|} r^{\nu+2+\lambda}$ and therefore $B(r) r^{-\lambda}$ is bounded by $\frac 1 {2|\lambda(1+\nu+\lambda)|} r^{\nu+2}$. This shows that $|u_p(r)| \leq 2 C r^{\nu+2}$, where $C$ is as in the statement of the lemma. Since there exist constants $c_1, c_2 \in \C$ such that $u = u_p - c_1 r^\lambda - c_2 r^{-\lambda}$, this finishes the proof.
\end{proof}

\begin{prop}
\label{prop:no_limiting_fiducial_kernel}
  Suppose $\gamma = \begin{pmatrix} u & v \\ \ol{v} & -u \end{pmatrix} \in C^{1,\alpha}_{\operatorname{loc}}(\R^2 \bs \{0\}, i \mf{su}(2))$ is a weak solution of
  \[
  \Delta_{A_\infty^{\fid}} \gamma - i \ast M_{\Phi_\infty^{\fid}} \gamma = 0
  \]
  and that $\gamma$ is bounded. Then $\gamma \equiv 0$.
\end{prop}
\begin{proof}
Let us observe that since $[\Delta_{A_\infty^{\fid}}, \pi_K] \gamma = 0$ and $[-i\ast M_{\Phi_\infty^{\fid}}, \pi_K] \gamma = 0$, it follows that $\gamma_1 = \pi_K \gamma$ and $\gamma_2 = \pi_{K^\perp} \gamma$ satisfy
  \[
  \Delta_{A_\infty^{\fid}} \gamma_1 - i \ast M_{\Phi_\infty^{\fid}} \gamma_1 = \Delta_{A_\infty^{\fid}} \gamma_1= 0
  \]
  and
  \[
  \Delta_{A_\infty^{\fid}} \gamma_2 - i \ast M_{\Phi_\infty^{\fid}} \gamma_2 = \Delta_{A_\infty^{\fid}} \gamma_2 + 16 r \gamma_2 = 0.
  \]
  Writing $\gamma_1 = \begin{pmatrix} 0 & v \\ \ol{v} & 0 \end{pmatrix}$, the first equation is equivalent to $\Delta_{1/2} v = 0$. Decompose $v = \sum_k v_k(r) e^{i k \theta}$. Then
  \[
  \Delta_{1/2} v = \sum_k \left[ -v_k''(r) - r^{-1} v_k'(r) + \frac{1}{r^2}\left(k+\frac 1 2\right)^2 v_k \right] e^{i k \theta}.
  \]
  Then $\Delta_{1/2} v = 0$ implies
  \[
  v_k = \alpha_k r^{k+\frac 1 2} + \beta_k r^{-k - \frac 1 2}.
  \]
  If any of these coefficients $\alpha_k$ or $\beta_k$ were non-zero, this would violate the assumption that $v$ is bounded at one of the ends. Therefore all coefficients must vanish, and this implies $v = 0$. Therefore $\gamma_1 = 0$.

  On the other hand, writing $\gamma_2 = \begin{pmatrix} u & v \\ \ol{v} & -u \end{pmatrix}$, the second equation is equivalent to
  \[
  \Delta u + 16 r u = 0, \qquad \Delta_{1/2} v + 16 r v = 0.
  \]
  We make the same Ansatz for $v$ as before and we observe that
  \[
  \Delta_{1/2} v + 16 r v = \sum_k \left[ -v_k''(r) - r^{-1} v_k'(r) + \frac{1}{r^2}\left(k+\frac 1 2\right)^2 v_k + 16 r v_k \right] e^{i k \theta}.
  \]
  Vanishing of this implies that
  \[
  v_k = \alpha_k I_{\frac 2 3 \left(k + \frac 1 2\right)} \left( \frac 8 3 r^{3/2} \right) + \beta_k K_{\frac 2 3 \left(k + \frac 1 2\right)} \left( \frac 8 3 r^{3/2} \right).
  \]
  If $\alpha_k \neq 0$, then $\lim_{r\to \infty} v_k(r) = \pm\infty$, and therefore $v$ cannot be bounded.

  On the other hand, if $\beta_k \neq 0$, then $v_k (r)$ has leading term $r^{-\left|k + \frac 1 2 \right|}$ at $0$. Again, $v$ cannot be bounded.

  This means that all the coefficients $\alpha_k$ and $\beta_k$ must vanish. Hence $v = 0$.

  With the Ansatz $u = \sum_k u_k(r) e^{i k \theta}$, we observe that $\Delta u + 16 r u = 0$ implies
  \[
  -u_k''(r) - r^{-1} u_k'(r) + r^{-2}k^2 u_k + 16 r u_k = 0.
  \]
  By precisely the same argument as before, all $u_k$ must vanish.
\end{proof}

\begin{prop}
\label{prop:no_fiducial_kernel}
  Suppose $\gamma = \begin{pmatrix} u & v \\ \ol{v} & -u \end{pmatrix} \in C^{1,\alpha}_{\operatorname{loc}}(\R^2, i \mf{su}(2))$ is a weak solution of
  \[
  \Delta_{A_1^{\fid}} \gamma - i \ast M_{\Phi_1^{\fid}} \gamma = 0
  \]
  and that $\gamma$ is bounded. Then $\gamma \equiv 0$.
\end{prop}
\begin{proof}
  The function $u$ satisfies
  \[
  \Delta u + 16 r\cosh(2 \ell_t(r)) u = 0.
  \]
  Since $16 r \cosh(2\ell_t(r))$ is uniformly bounded below by a positive constant, the (Omori--Yau) maximum principle shows that $u$ must vanish.

It remains to be shown that $v$ vanishes.

  The function $v$ satisfies
  \[
  \Delta v - 4 F_t i \partial_\theta v + 4 F_t^2 v + 8 r (\cosh(2 \ell_t(r)) v - e^{-i\theta} \ol{v}) = 0.
  \]
  
  Writing $v = \sum_k v_k(r) e^{i k\theta}$, we obtain
  \[
  \sum_k \left( \left[ -v_k'' - r^{-1} v_k' + k^2 r^{-2} v_k + 4 r^{-2} F_t k v_k + 4 r^{-2} F_t^2 v_k + 8 r \cosh(2 \ell_t(r)) v_k \right] e^{i k\theta} - 8 r v_k e^{-i(k+1) \theta} \right) = 0.
  \]
  Comparing coefficients, we see that
  \[
  -v_k'' - r^{-1} v_k' + k^2 r^{-2} v_k + 4 r^{-2} F_t k v_k + 4 r^{-2} F_t^2 v_k + 8 r \cosh(2 \ell_t(r)) v_k - 8r v_{-(k+1)} = 0.
  \]
  Let us define $\hat{v}_k = v_k + v_{-(k+1)}$ and $\check{v}_k = v_k - v_{-(k+1)}$. This corresponds to the decomposition of the off-diagonal part of $\gamma$ into $\pi_K \gamma$ and $\pi_{K^\perp} \gamma$ in terms of Fourier modes.

  We then observe that
  \[
  -\hat{v}_k'' - r^{-1} \hat{v}_k' + k^2 r^{-2} \hat{v}_k + 4 r^{-2} F_t k \hat{v}_k + 4 r^{-2} F_t^2 \hat{v}_k + 8 r (\cosh(2 \ell_t(r)) - 1) \hat{v}_k = 0
  \]
  and
  \[
  -\check{v}_k'' - r^{-1} \check{v}_k' + k^2 r^{-2} \check{v}_k + 4 r^{-2} F_t k \check{v}_k + 4 r^{-2} F_t^2 \check{v}_k + 8 r (\cosh(2 \ell_t(r)) + 1) \check{v}_k = 0.
  \]
  Let us rewrite the first equation as
  \[
  -\hat{v}_k'' - r^{-1} \hat{v}_k' + k^2 r^{-2} \hat{v}_k + k r^{-2} \hat{v}_k + \frac 1 4 r^{-2} \hat{v}_k = r^{-2} \hat{f} \,\hat{v}_k
  \]
  where
  \[
  r^{-2} \hat{f} :=   r^{-2} k (1 - 4 F_t) + r^{-2} \left( \frac 1 4 - 4 F_t^2 \right) + 8 r (1 - \cosh(2\ell_t(r))) .
  \]
  The function $\hat{f}$ decays super-exponentially as $r \to \infty$. In particular, since $\hat{v}_k$ is bounded, this implies that the right-hand side decays faster than any $\frac 1 {r^p}$ for any $p > 0$.

  The left hand side can be rewritten as $-\hat{v}_k'' - r^{-1} \hat{v}_k' + \left(k + \frac 1 2 \right)^2 r^{-2} \hat{v}_k$. Therefore, we may apply Lemma \ref{lemma:laplacian_bound} to see that there exist constants $c_1, c_2 \in \C$, such that $\left| \hat{v}_k(r) - c_1 r^{k+ \frac 1 2} - c_2 r^{-\left(k + \frac 1 2\right)} \right|$ decays faster than any $\frac 1 {r^p}$. But if $c_1$ or $c_2$ were non-zero, the function $\hat{v}_k$ and therefore $v$ would be unbounded, which cannot happen. Hence $c_1 = c_2 = 0$.

For the $\check{v}_k$, we rewrite the equation for $\check{v}_k$ as
\[
  -\check{v}_k'' - r^{-1} \check{v}_k' + \left( k + \frac 1 2 \right)^2 r^{-2} \check{v}_k + 16 r \check{v}_k = r^{-2} \check{f} \,\check{v}_k 
  \]
where
\[
r^{-2} \check{f} :=  r^{-2} k (1 - 4 F_t) + r^{-2} \left( \frac 1 4 - 4 F_t^2 \right) + r (16 - (1 + \cosh(2\ell_t(r)))) .
\]  
  
The function $\check{f}$ decays super-exponentially as $r \to \infty$.

Let $\check{w}_k(s) = \check{v}_k\left( \left(\frac 3 8 s \right)^{2/3} \right)$. Then $\check{w}_k$ satisfies the equation
\[
-(s \partial_s)^2 \check{w}_k + \left( \left( \frac 2 3 \left(k + \frac 1 2 \right)\right)^2 + s^2 \right) \check{w}_k = \frac 4 9 \check{g} \check{w}_k,
\]
where $\check{g}(s) = \check{f}\left( \left(\frac 3 8 s \right)^{2/3} \right)$. The function $\check{g}$ decays exponentially in $s$.

Proposition \ref{prop:perturbed_bessel} now yields that
\[
|\check{w}_k(s)| \leq C (1+ \nu^2) |\check{w}_k(s_0)| e^{-\kappa s} \quad \text{and} \quad |\check{w}_k'(s)| \leq \wt{C} (1+ \nu^2) |\check{w}_k(s_0)| e^{-\kappa s}
\]
where $C, \wt{C}$ and $\kappa$ do not depend on $\nu$. This estimate turns into the estimates
\[
|\check{v}_k(t)| \leq C (1+ \nu^2) |\check{v}_k(t_0)| e^{-\kappa \frac 8 3 r^{3/2}}
\]
and analogously for $\check{v}_k'$. 

Combining these estimates yields analogous estimates for $v_k(r)$ and $v_k'(r)$. Observe that since $v$ is smooth on any concentric circle, the Fourier coefficients $v_k(t_0)$ decay faster than polynomially in $k$.

Therefore, the estimates we have shown imply that $\sum_k v_k(r) e^{ik\theta}$ converges absolutely and decays faster than any $\frac 1 {r^p}$. The same holds for the radial derivative $\partial_r v = \sum_k v_k'(r) e^{i k \theta}$. This implies that we may integrate the equation by parts to obtain
\[
\|d_{A_1^{\fid}} \gamma\|^2 + \|[\Phi_1^{\fid}, \gamma]\|^2 = 0.
\]
Now $[\Phi_1^{\fid}, \gamma] = 0$ already implies $\gamma = 0$, as claimed.
\end{proof}

\subsubsection{Proof of Theorem \ref{thm:C0_estimate}}

\begin{thm}
\label{thm:C0_estimate}
Suppose $\delta \in (1/2,1)$ and $\nu = 0$. There exist $C, T > 0$ such that for any $t > T$ and any $\gamma \in \rho^\delta r_t^\nu C^0 \cap C^2_{\lcl}$ the following inequality holds:
  \[
  \|\gamma\|_{\rho^\delta r_t^\nu C^0} \leq C \|L_t \gamma \|_{\rho^\delta r_t^{\nu-2} C^0} .
  \]
\end{thm}
\begin{proof}
Suppose the statement is false. Then there exist sequences $t_n \in \R_+$, $\gamma_n \in \rho^\delta r_{t_n}^\nu C^0 \cap C^2_{\operatorname{loc}}$, such that
\[
\|\gamma_n\|_{\rho^\delta r_{t_n}^\nu C^0} = 1, \qquad \|L_{t_n} \gamma_n\|_{\rho^\delta r_{t_n}^{\nu-2} C^0} \to 0.
\]
We can pick points $p_n$ such that
\[
\|\gamma_n\|_{\rho^\delta r_{t_n}^\nu C^0}(p_n) \geq \frac 1 2.
\]
After passing to a subsequence, we may assume that $t_n$ diverges to $\infty$ and $p_n$ converges to $p_\infty$ in $M$. Because $\gamma_n = \pi_K \gamma_n + \pi_{K^\perp} \gamma_n$, it follows that after passing to a subsequence either $\|\pi_K \gamma_n\|_{\rho^\delta r_{t_n}^\nu C^0}(p_n) \geq \frac 1 4$ or $\|\pi_{K^\perp} \gamma_n\|_{\rho^\delta r_{t_n}^\nu C^0}(p_n) \geq \frac 1 4$.

In the following, we will do a case-by-case analysis depending on $p_\infty$ and on whether the projection to the kernel or to its orthogonal complement dominates. Each case will lead to a contradiction by constructing a non-trivial sublimit of the $\gamma_n$, which solves a partial differential equation. The equation that is solved depends on the limiting point, how this limit is attained, and on whether the projection to the kernel or its orthogonal complement dominates. In every case, one of our vanishing theorems shows that the limiting solution must be zero, which is a contradiction. For the sake of readability, we break up the construction of these limits and that they lead to contradictions into a series of lemmata:
  \begin{enumerate}
  \item $p_\infty$ is neither a zero of $q$ nor does it lie in a core loop, i.e.\@ $p_\infty \in M^\times$.
    \begin{enumerate}
    \item $\|\pi_K \gamma_n\|_{\rho^\delta r_{t_n}^\nu C^0} \geq \frac 1 4$: Lemma \ref{lemma:case_2_a_i}.
    \item $\|\pi_{K^\perp} \gamma_n\|_{\rho^\delta r_{t_n}^\nu C^0} \geq \frac 1 4$: Lemma \ref{lemma:case_2_a_ii_c_ii}.
    \end{enumerate}
  \item $p_\infty$ is a zero of $q$, i.e.\@ $p_\infty \in Z$.
    \begin{enumerate}
    \item $t_n^{2/3} r(p_n)$ is bounded: Lemma \ref{lemma:case_2_b_i}.
    \item $t_n^{2/3} r(p_n)$ is unbounded: Lemma \ref{lemma:case_2_b_ii}.
    \end{enumerate}
  \item $p_\infty$ lies in a core loop, i.e.\@ $p_\infty \in M \bs M^\vee$, and $\|\gamma_n\|_{\rho^\delta r_{t_n}^\nu C^0} \geq \frac 1 2$: Lemma \ref{lemma:case_2_c_i}.
  \end{enumerate}

Since these cases cover every possibility, this shows that the sequence $\gamma_n$ as above cannot exist, and therefore the estimate must hold as claimed.
\end{proof}

\begin{lemma}
\label{lemma:case_2_a_i}
Let $\delta > 1/2$ and $\nu \geq 0$.

There are no sequences $t_n \in \R_+$, $p_n \in M$, $\gamma_n \in \rho^\delta r_{t_n}^\nu C^0 \cap C^2_{\operatorname{loc}}$, such that
\[
\|\gamma_n\|_{\rho^\delta r_{t_n}^\nu C^0} = 1, \qquad \|L_{t_n} \gamma_n\|_{\rho^\delta r_{t_n}^{\nu-2} C^0} \to 0,
\]
\[
\|\pi_K \gamma_n\|_{\rho^\delta r_{t_n}^\nu C^0}(p_n) \geq \frac 1 2
\]
and $t_n$ diverges to $\infty$, and $p_n \to p_\infty \in M^\times$.
\end{lemma}

\begin{proof}
Consider the sequence $\gamma_n^K = \pi_K \gamma_n$. We claim that $L_{t_n} \gamma_n^K$ converges uniformly to $0$ on any compact subset of $M^\times$. To see this, first observe that $L_{t_n} \pi_K \gamma_n = \pi_K L_{t_n} \gamma_n$ away from the core loops and from the zeros. On the other hand, according to Lemmas \ref{lemma:convergence_commutator_zeros} and \ref{lemma:convergence_commutator_near_core_loop}, for the fiducial solution and the model solution $[\pi_K, L_t]$ converges uniformly to $0$ away from the zeros and the core loop, respectively. This implies that the same is true for the approximate solution, because near the zeros and the core loops the approximate solution coincides with the fiducial and model solutions.

Moreover, observe that, since $\gamma_n^K$ is a section of $K$, 
\[
L_{t_n} \gamma_n^K \to \Delta_{A_\infty} \gamma_\infty^K,
\]
on any compact subset of $M^\times$, because of Lemmas \ref{lemma:convergence_Delta_A_t_near_zeros}, \ref{lemma:convergence_Delta_A_t_near_core_loop}, \ref{lemma:convergence_M_Phi_t_K_near_zeros}, \ref{lemma:convergence_M_Phi_t_K_near_core_loop}.

Moreover, $\|\gamma_n^K\|_{\rho^\delta r_{t_n}^\nu C^0} \leq C$ for some $C > 0$, and therefore we may apply elliptic regularity (for the operator $\Delta_{A_\infty}$) over compact domains in $M^\times$ to get uniform $C^{1,\alpha}$ bounds. Applying Arzel\`a--Ascoli, we obtain a limit on each such domain and by applying diagonalization we obtain a $\gamma_\infty^K$ defined on $M^\times$, which satisfies $\Delta_{A_\infty} \gamma_\infty^K = 0$, $\gamma_\infty^K(p_\infty) \neq 0$, and $\gamma_\infty^K \in \rho^\delta r_\infty^\nu C^0$. By Proposition \ref{prop:no_A_infty_harmonic} no such solution exists, leading to the claimed contradiction.
\end{proof}

\begin{lemma}
\label{lemma:case_2_b_i}
Let $\delta \in \R$ and $\nu \leq 0$.

There are no sequences $t_n \in \R_+$, $p_n \in M$, $\gamma_n \in \rho^\delta r_{t_n}^\nu C^0 \cap C^2_{\operatorname{loc}}$, such that
\[
\|\gamma_n\|_{\rho^\delta r_{t_n}^\nu C^0} = 1, \qquad \|L_{t_n} \gamma_n\|_{\rho^\delta r_{t_n}^{\nu-2} C^0} \to 0,
\]
\[
\|\gamma_n\|_{\rho^\delta r_{t_n}^\nu C^0}(p_n) \geq \frac 1 4
\]
and $t_n$ diverges to $\infty$, $p_n \to p_\infty \in Z$, and $t_n^{2/3} r(p_n)$ is bounded.
\end{lemma}

\begin{proof}
In one of the coordinate charts $U_j$ the point $p_\infty \in Z$ corresponds to $0$. Using these coordinates and the unitary frame of $E$, consider the $\gamma_n$ to be maps from $\mathbb{D} \to i \mf{su}(2)$. Likewise, the points $p_n$ are considered to be points in $\mathbb{D}$ converging to $0$. Now consider $\wt{\gamma}_n : B_{t_n^{2/3}} \to i \mf{su}(2)$, $\wt{\gamma}_n(x) = \gamma_n(t_n^{-2/3} x)$. If we define $\eta_n = L_{t_n} \gamma_n$ and define $\wt{\eta}_n(x) = \eta_n(t_n^{-2/3} x)$, then
\[
t_n^{4/3} \left(\Delta_{A_1^{\fid}} - i \ast M_{\Phi_1^{\fid}} \right) \wt{\gamma}_n = \wt{L_{t_n} \gamma_n} = \wt{\eta}_n.
\]
(See the proof of Proposition 5.9 in \cite{FMSW}.) Observe that
\[
|\wt{\gamma}_n(x)| \leq \rho^\delta(z^{-1}(t_n^{-2/3} x)) r_{t_n}^\nu(z^{-1}(t_n^{-2/3} x)) = \left(t_n^{-4/3} + t_n^{-4/3} |x|^2\right)^{\nu/2} = t_n^{-2\nu/3 } (1 + |x|^2)^{\nu/2}
\]
because $\rho \equiv 1$ and $r_{t_n}(z(x)) = (t_n^{-4/3} + |x|^2)^{1/2}$ on $U_j$. Similarly,
\[
|\wt{\eta}_n(x)| \leq \epsilon_n t_n^{-2 (\nu - 2)/3} (1 + |x|^2)^{(\nu-2)/2}
\]
where $\epsilon_n$ is a null sequence.

Therefore if we define $\hat{\gamma}_n = t_n^{2 \nu /3 } \,\wt{\gamma}_n$ we have
\[
|\hat{\gamma}_n| \leq (1 + |x|^2)^{\nu/2}, \qquad |\hat{\gamma}_n(\wt{p}_n)| \geq \frac 1 2 (1 + |\wt{p}_n|^2)^{\nu/2},
\]
where $\wt{p}_n = t_n^{2/3} p_n$. After passing to a subsequence, we may assume $\wt{p}_n$ converges to some $\wt{p}_\infty \in \R^2$.  The section $\hat{\gamma}_n$ satisfies
\[
\left(\Delta_{A_1^{\fid}} - i \ast M_{\Phi_1^{\fid}} \right) \hat{\gamma}_n = t_n^{-4/3 + 2 \nu/3 } \wt{\eta}_n,
\]
where the right-hand side satisfies
\[
|t_n^{-4/3 + 2\nu/3 } \wt{\eta}_n| \leq \epsilon_n (1 + |x|^2)^{(\nu-2)/2}.
\]
Applying elliptic regularity, Arzel\`a--Ascoli, and a diagonalization argument, we can construct a limit of $\hat{\gamma}_n$ converging locally in $C^{1,\alpha}$ to a weak solution $\hat{\gamma}_\infty$ on $\R^2$ of
\[
\left(\Delta_{A_1^{\fid}} - i \ast M_{\Phi_1^{\fid}} \right) \hat{\gamma}_\infty = 0
\]
satisfying $|\hat{\gamma}_\infty| \leq (1 + |x|^2)^{\nu/2}$. Furthermore, $|\hat{\gamma}_\infty(\wt{p}_\infty)| \geq \frac 1 2 (1 + |\wt{p}_\infty|^2)^{1/2}$. According to Proposition \ref{prop:no_fiducial_kernel}, no such solution can exist. 
\end{proof}

\begin{lemma}
\label{lemma:case_2_b_ii}
Let $\delta \in \R$ and $\nu = 0$. 

There are no sequences $t_n \in \R_+$, $p_n \in M$, $\gamma_n \in \rho^\delta r_{t_n}^\nu C^0 \cap C^2_{\operatorname{loc}}$, such that
\[
\|\gamma_n\|_{\rho^\delta r_{t_n}^\nu C^0} = 1, \qquad \|L_{t_n} \gamma_n\|_{\rho^\delta r_{t_n}^{\nu-2} C^0} \to 0,
\]
\[
\|\pi_K \gamma_n\|_{\rho^\delta r_{t_n}^\nu C^0}(p_n) \geq \frac 1 4
\]
and $t_n$ diverges to $\infty$, $p_n \to p_\infty \in Z$, and $t_n^{2/3} r(p_n)$ is unbounded.
\end{lemma}
\begin{proof}
As in the proof of Lemma \ref{lemma:case_2_b_i}, by rescaling we obtain a sequence $\hat{\gamma}_n : B_{t_n^{2/3}} \to i \mf{su}(2)$ satisfying
\[
|\hat{\gamma}_n| \leq 1, \qquad |\pi_K \hat{\gamma}_n(\wt{p}_n)| \geq \frac 1 4
\]
with $\wt{p}_n = t_n^{2/3} p_n$ (where $p_n$ now means the coordinate representation). Moreover, $\hat{\eta}_n = (\Delta_{A_1^{\fid}} - i \ast M_{\Phi_1^{\fid}}) \hat{\gamma}_n$ satisfies
\[
|\hat{\eta}_n| \leq \epsilon_n (1 + |x|^2)^{-1},
\]
where $\epsilon_n$ is a null sequence.

After passing to a subsequence we may assume that $|t_n^{2/3} p_n| \to \infty$. We perform a secondary rescaling. Define $\check{\gamma}_n(x) = \hat{\gamma}_n(|\wt{p}_n| x)$. This then satisfies
\[
|\check{\gamma}_n(x)| \leq 1.
\]
Moreover, define $\check{p}_n = p_n/|p_n|$. After passing to a subsequence $\check{p}_n$ converges to some $\check{p}_\infty$. Applying the familiar arguments, we may also assume that $\check{\gamma}_n$ locally converges in $C^{1,\alpha}$ to a weak solution $\check{\gamma}_\infty$ of
\[
\Delta_{A_\infty^{\fid}} \check{\gamma}_\infty - i \ast M_{\Phi_\infty^{\fid}} \check{\gamma}_\infty = 0,
\]
defined on $\R^2 \bs \{0\}$. By Proposition \ref{prop:no_limiting_fiducial_kernel} the only bounded solution of this equation is $0$.
\end{proof}

\begin{lemma}
\label{lemma:halfspace_weight}
On the half space $\H = \{ (x,y) \in \R^2 : x > 0 \}$ the function $x^\beta$ is contained in $\rho^\delta L^2_0$ if $\beta > 1/4$ and $\beta + \delta < 0$.
\end{lemma}
\begin{proof}
Recall that in the Poincar\'e disk model of the hyperbolic plane $\{ z = u + i v \in \C : |z|^2 < 1\}$ the weight function $\rho$ is given by $1 - u^2 - v^2$. With respect to the isometry $(u,v) \mapsto \left(\frac{1 - u^2 - v^2}{u^2 + (1-v)^2}, \frac{2 u}{u^2+(1-v)^2} \right)$ between the disk model and the half plane model, the function $x$ becomes $\frac{1 - u^2 - v^2}{u^2 + (1-v)^2}$. Therefore,
\[
\|u\|_{\rho^{\delta} L^2_0}^2 = \int_{\mathbb{D}} \left( \rho^{-\delta} \left(\frac{1 - u^2 - v^2}{u^2 + (1-v)^2}\right)^\beta \right)^2 \rho^{-2} du dv.
\]
This may be computed in polar coordinates to be
\[
I(\nu) =  \int_0^{2\pi} \int_0^1 \left( (1-r^2)^{-\delta} \left(\frac{1-r^2}{r^2 + 1 - 2r \sin(\theta)}\right)^\beta\right)^2 (1-r^2)^{-2} r dr d\theta.
\]
First, let us observe that the dominant contribution to the integral $\int_0^{2\pi} \frac 1 {(r^2 + 1 - 2r \sin(\theta))^{2\beta}} d\theta$ occurs around $\frac \pi 2$. Taylor expanding $\sin$ around $\frac \pi 2$, we see that the integrand around $\frac \pi 2$ can be approximated by $\frac 1 {\left((1-r)^2 + (t - \pi /2)\right)^{2\beta}}$. Therefore, with $\epsilon = 1 - r$ we see that this integral is dominated (up to a constant) by
\[
\int_{-1}^1 \frac 1 { (\epsilon^2 + t^2)^{2\beta}} dt = \int_{-1/\epsilon}^{1/\epsilon} \frac 1 {( 1 + s^2)^{2\nu}} \epsilon^{-4\beta + 1} ds .
\]
Now if $\beta > \frac 1 4$, the integral $\int_{-\infty}^\infty \frac 1 {(1+s^2)^{2 \beta}} ds$ is finite. We conclude that for $\beta > \frac 1 4$ and $r \in (0,1)$
\[
\int_0^{2\pi} \frac 1 { (1 + r^2 - 2r \sin(\theta))^{2\beta}} dr \leq c_\nu (1-r)^{-4\beta + 1}
\]
for some constant $c_\beta$ depending on $\beta$.

Therefore the integral $I(\beta)$ is finite if the integral $\int_0^1 (1-r^2)^{-2\delta+2\beta-2} (1-r)^{-4\beta+1} dr$ is finite. Since $1-r^2 = (1+r)(1-r)$, this is finite if $\int_0^1 (1-r)^{-2\delta+2\beta- 4\beta +1} dr$ is finite. This happens precisely when $2(-\beta-\delta) > 0$, which is equivalent to the assumption $\beta+\delta < 0$.
\end{proof}

\begin{lemma}
\label{lemma:case_2_c_i}
Let $\delta \in \left(1/4, 1\right)$ and $\nu \in \R$.

There are no sequences $t_n \in \R_+$, $\gamma_n \in \rho^\delta r_{t_n}^\nu C^0 \cap C^2_{\operatorname{loc}}$, such that
\[
\|\gamma_n\|_{\rho^\delta r_{t_n}^\nu C^0} = 1, \qquad \|L_{t_n} \gamma_n\|_{\rho^\delta r_{t_n}^{\nu-2} C^0} \to 0,
\]
\[
\|\gamma_n\|_{\rho^\delta r_{t_n}^\nu C^0}(p_n) \geq \frac 1 2
\]
and $t_n$ diverges to $\infty$ and $p_n \to p_\infty \in M \bs M^\vee$.
\end{lemma}

\begin{proof}
Denote $x_n = x(p_n)$. For large enough $n$ all points $p_n$ will lie in the coordinate system $(V_j, z_j)$. We will omit the $j$, i.e.\@ our coordinates are $z = x+iy$. From now on $\gamma_n$ will denote the coordinate representation of $\gamma_n$ with respect to these coordinates and the unitary frame specified in the framed limiting configuration.

Write $\gamma_n = \begin{pmatrix} u_n & v_n \\ \ol{v_n} & -u_n \end{pmatrix}$ and $L_{t_n} \gamma_n = \eta_n = \begin{pmatrix} f_n & g_n \\ \ol{g_n} & -f_n \end{pmatrix}$. Then according to Lemma \ref{lemma:local_form_near_core_loop_finite_t}
\[
\Delta_{\H^2} u_n + 2 t_n^2 x^2 E(t_n x) u_n = f_n,
\]
\[
\Delta_{\H^2} v_n + 2ip(t_n x) x \partial_y v_n + p(t_n x)^2 v_n - 2 t_n x^2 \ol{v_n} + t_n^2 x^2 E(t_n x) v_n = g_n.
\]
We define
\[
\hat{\gamma}_n : \left\{ z = x+iy \in \C : 0 < x < x_n^{-1} \right\}/ ( x_n^{-1} i \sigma_j \Z) \to i \mf{su}(2),
\]
\[
\hat{\gamma}_n(x,y) = x_n^{-\delta} \gamma_n(x_n x, x_n y + y_n),
\]
and analogously $\hat{\eta}_n$. The operators $\Delta_{\H^2}$ and $x \partial_y$ are dilation invariant. Therefore we have the following equations for the components of $\hat{\gamma}_n$:
\[
\Delta \hat{u}_n + 2 \lambda_n^2 x^2  E(\lambda_n x) \hat{u}_n = \hat{f}_n,
\]
\[
\Delta \hat{v}_n + 2 i p(\lambda_n x) x \partial_y \hat{v}_n + p(\lambda_n x)^2 \hat{v}_n - 2 \lambda^2 x^2 \ol{\hat{v}_n} +  \lambda_n^2 x^2 E(\lambda_n x) \hat{v}_n = \hat{g}_n.
\]
Moreover, we have the inequalities
\[
|\hat{\gamma}_n(1,0)| \geq \frac 1 2,
\]
\[
|\hat{\gamma}_n(p)| \leq x(p)^\delta,
\]
\[
|\hat{f}_n|, |\hat{g}_n|  \leq \epsilon_n x(p)^\delta,
\]
where $\epsilon_n$ is some null sequence.

After passing to a subsequence we may assume that $\lambda_n$ converges to some $\lambda \in [0,\infty]$. By the familiar process we may extract sublimits $\hat{u}_\infty$ and $\hat{v}_\infty$ of the $\hat{u}_n$ and $\hat{v}_n$, which then satisfy $|\hat{u}_\infty| \leq x^\delta$ and $|\hat{v}_\infty| \leq x^\delta$. Moreover, at least one of $\hat{u}_\infty$ and $\hat{v}_\infty$ must be non-vanishing.

Since $\delta > \frac 1 4$, lemma \ref{lemma:halfspace_weight} says that the functions $\hat{u}_\infty$ and $\hat{v}_\infty$ lie in $\rho^\mu L^2$ for every $\mu < -\delta$. Since $\delta < 1$, we therefore get that they lie in a $\rho^\mu L^2$ with $\mu > -1 > - \frac 3 2$.

If $\lambda \in (0,\infty)$ they satisfy the equations
\[
\Delta_{\H^2} \hat{u}_\infty + 2 \lambda^2 x^2 E(\lambda x) \hat{u}_\infty = 0,
\]
\[
\Delta_{\H^2} \hat{v}_\infty + 2 i p(\lambda x) x \partial_y \hat{v}_\infty + p(\lambda x)^2 \hat{v}_\infty - 2 \lambda^2 x^2 \hat{v}_\infty + \lambda^2 x^2 E(\lambda x) \hat{v}_\infty = 0.
\]
By propositions \ref{prop:trivial_kernel_model_space_diag} and \ref{prop:trivial_kernel_model_space_off_diag} and the fact that $\hat{u}_\infty$ and $\hat{v}_\infty$ satisfy $|\hat{u}_\infty|, |\hat{v}_\infty| \leq x^\delta$, this implies that $\hat{u}_\infty$ and $\hat{v}_\infty$ must vanish. If $\lambda = 0$, the limit instead satisfies
\[
\Delta \hat{u}_\infty + 2 \hat{u}_\infty = 0, \quad \Delta \hat{v}_\infty + 2 i x \partial_y \hat{v}_\infty + \hat{v}_\infty + \hat{v}_\infty = 0
\]
and in this case propositions \ref{prop:no_harmonic_on_halfspace} and \ref{prop:no_twisted_harmonic_on_halfspace} show that $\hat{u}_\infty$ and $\hat{v}_\infty$ vanish.

Therefore, all these cases are ruled out and we conclude that $\lambda_n \to \infty$. In this case, we differentiate the cases $\hat{u}_\infty \neq 0$, $\Imag \hat{v}_\infty \neq 0$ and $\Real \hat{v}_\infty \neq 0$. Note that the first two cases correspond to $\pi_{K^\perp} \hat{\gamma}_\infty \neq 0$ and the last case corresponds to $\pi_K \hat{\gamma}_\infty \neq 0$.

In the first two cases, a semiclassical argument as in Lemma \ref{lemma:case_2_a_ii_c_ii} can be performed. Assume first $\hat{u}_\infty \neq 0$ and define
\[
\check{u}_n(X,Y) = \hat{u}_n\left( (1,0) + \lambda_n^{-1} (X,Y) \right),
\]
\[
\check{f}_n(X,Y) = \hat{f}_n\left( (1,0) + \lambda_n^{-1} (X,Y) \right).
\]
Then
\[
\left( \hat{x}_n + \lambda_n^{-1} X \right)^2 \lambda_n^2 (\Delta_{\R^2} \check{u}_n)(X,Y) + 2 \lambda_n^2 \left( \hat{x}_n + \lambda_n^{-1} X \right)^2 E(\lambda_n(\hat{x}_n + \lambda_n^{-1} X)) \check{u}_n = \check{f_n}(X,Y).
\]
Recalling that $\hat{x}_n = 1$, we observe that for any finite $X$ the expression $\hat{x}_n  + \lambda_n^{-1}X$ converges to $1$, whereas the expression $E(\lambda_n(\hat{x}_n + \lambda_n^{-1} X))$ converges to $2$. We may then multiply both sides by $\lambda_n^{-2}$ and extract a limit $\hat{u}_\infty : \R^2 \to \R$, which satisfies
\[
\Delta_{\R^2} \check{u}_\infty + 4 \check{u}_\infty = 0.
\]
Note that $\check{u}_\infty$ is bounded, since the values of $\hat{u}_\infty$ depend only on the values of $\hat{u}_n$ in a ball of fixed, small radius. Such a function must vanish by Proposition \ref{prop:no_bounded_eigenfunctions}.

Next, assume that $\Imag \hat{v}_\infty \neq 0$ and define
\[
\check{v}_n(X,Y) = \hat{v}_n\left( (1,0) + \lambda_n^{-1} (X,Y) \right),
\]
\[
\check{g}_n(X,Y) = \hat{g}_n\left( (1,0) + \lambda_n^{-1} (X,Y) \right).
\]
Observing that $p(s) \to 0$ as $s \to \infty$ and arguing as in the previous case, we can extract a bounded solution $\check{v}_\infty$ of the equation
\[
\Delta_{\R^2} \check{v}_\infty - 2 \overline{\check{v}_\infty} + 2 \overline{\check{v}_\infty} = 0.
\]
The imaginary part of $\check{v}_\infty$ then satisfies the equation
\[
\Delta_{\R^2} \Imag \check{v}_\infty + 4 \Imag \check{v}_\infty = 0
\]
and therefore vanishes as the in the previous case. Note that the real part is harmonic and that this does not suffice to show that $\Real \check{v}_\infty$ vanishes. Instead, it turns out that the equation on $\H^2$ has a limit for the real part. This is expected, as the real part of $\gamma_n$ corresponds precisely the projection onto the kernel bundle $K$.

The equation satisfied by $w_n = \Real\hat{v}_n$ is
\[
\Delta_{\H^2} w_n - 2  p(t_n x_n x) x\partial_y \Imag \hat{v}_n + p(t_n x_n x)^2 w_n + t_n^2 x_n^2 x^2 (E(t_n x_n x) - 2) w_n = \Real \hat{g}_n.
\]
Now if $t_n x_n \to \infty$, then for every $\epsilon > 0$  we can extract a subsequence of the $w_n$, which converges to a $w_\infty$ satisfying $\Delta_{\H^2} w_\infty = 0$ on the set $x > \epsilon$. This is because $p(t_n x_n x) \to 0$ and $t_n^2 x_n^2 (E(t_n x_n x) - 2) \to 0$ uniformly on this set. Applying a standard diagonalization argument we obtain a non-trivial harmonic $w_\infty$ defined on $x > 0$. Since moreover $|w_\infty| \leq x^\delta$, this is a contradiction to Proposition \ref{prop:stronger_no_harmonic_on_halfspace}.
\end{proof}

\begin{lemma}
  \label{lemma:case_2_a_ii_c_ii}

  Let $\nu = 0$ and $\delta \in \R$.
  
There are no sequences $t_n \in \R_+$, $\gamma_n \in \rho^\delta r_{t_n}^\nu C^0 \cap C^2_{\operatorname{loc}}$, such that
\[
\|\gamma_n\|_{\rho^\delta r_{t_n}^\nu C^0} = 1, \qquad \|L_{t_n} \gamma_n\|_{\rho^\delta r_{t_n}^{\nu-2} C^0} \to 0,
\]
\[
\|\pi_{K^\perp} \gamma_n\|_{\rho^\delta r_{t_n}^\nu C^0}(p_n) \geq \frac 1 4
\]
and $t_n$ diverges to $\infty$, $p_n \to p_\infty \in M^\vee \bs Z$.
\end{lemma}
\begin{proof}
First let us note that if $\|L_{t_n} \gamma_n\|_{\rho^\delta r_{t_n}^\nu C^0} \to 0$, then certainly also $t_n^{-2} \|L_{t_n} \gamma_n\|_{\rho^\delta r_{t_n}^\nu C^0} \to 0$. 

Note that $t^{-2} L_t = t^{-2} \Delta_{A_t} - i \ast M_{\Phi_t}$, so we are studying a semiclassical problem here.

For each $p_n$ we may choose normal coordinates of a definite size $r_0>0$ with respect to the conformally compact metric. By abuse of notation we now consider the $\gamma_n$ to be defined on $B_{r_0}$ with respect to these coordinates. In a next step, we define $\wt{\gamma}_n : B_{t_n r_0} \to i \mf{su}(2)$, $\wt{\gamma}_n(x) = \rho^{-\delta}(p_n) r_{t_n}^{-\nu}(p_n) \gamma_n(t_n^{-1} x)$.

If in these coordinates
\[
t_n^{-2} \left( a_{n,ij}(x) \partial_{ij} \gamma_n + b_{n,i}(x) \partial_i \gamma_n + c_n(x) \gamma_n \right) + d_n(x) \gamma_n(x) = \eta_n,
\]
then
\[
\left( a_{n,ij}(t_n^{-1} x) \partial_{ij} \wt{\gamma}_n + t_n^{-1} b_{n,i}(t_n^{-1} x) \partial_i \wt{\gamma}_n + t_n^{-2} c_n(t_n^{-1} x) \wt{\gamma}_n \right) + d_n(t_n^{-1} x) \gamma_n(x) = \wt{\eta}_n.
\]
Provided we can pass to the limit, this means that $\wt{\gamma}_\infty : \R^2 \to i \mf{su}(2)$ satisfies
\[
\Delta_0 \wt{\gamma}_\infty - i \ast M_{\Phi_\infty}(p_\infty) \wt{\gamma}_\infty = 0.
\]
where $\Delta_0$ is the standard Laplacian on $\R^2$. Moreover, we know that $[\Phi_\infty(p_\infty), \wt{\gamma}_\infty(0)] \neq 0$.

Note that $p_\infty \notin Z$ by assumption so that the operator $- i \ast M_{\Phi_\infty}(p_\infty)$ has a positive eigenvalue $\lambda$ with multiplicity two and a one dimensional kernel. Recall that $\pi_{K^\perp}$ denotes the projection of $i \mf{su}(2)$ onto the $\lambda$ eigenspace. Then $\pi_{K^\perp} \wt{\gamma}_\infty$ satisfies
\[
\Delta_0 \pi_{K^\perp} \wt{\gamma}_\infty + \lambda \pi_{K^\perp} \wt{\gamma}_\infty = 0.
\]
To see that such a limit exists and that $\wt{\gamma}_\infty$ is non-trivial observe that
\[
|\wt{\gamma}_n(x)| \leq \frac{\rho^\delta (\exp_{p_n}(t_n^{-1} x)) r_{t_n}^\nu(\exp_{p_n}(t_n^{-1} x))}{\rho^\delta(p_n)r_{t_n}^\nu(p_n)} \leq C
\]
and $|\wt{\gamma}_n(0)| \geq \frac 1 4$. The $\wt{\eta}_n$ as defined above similarly converge to $0$ in $C^0$. Applying elliptic regularity as before therefore we obtain a limiting $\wt{\gamma}_\infty$ satisfying the above equation. Because $[\Phi_\infty(p_n), \gamma_n(p_n)] \to [\Phi_\infty(p_\infty), \wt{\gamma}_\infty(0)] \neq 0$ we also have $P_\lambda \wt{\gamma}_\infty(0) \neq 0$. Therefore $u = P_\lambda \wt{\gamma}_\infty$ is a non-trivial bounded solution of $\Delta_0 u + \lambda u = 0$ and such a solution does not exists according to Proposition \ref{prop:no_bounded_eigenfunctions}.
\end{proof}

\subsection{Perturbation to an exact solution}

\begin{thm}
\label{thm:existence_solution}
Let $\delta \in ( 1/2, 1)$, $\nu = 0$ and $\alpha \in (0,1)$.

For all sufficiently large $t$ there exists a $\gamma \in \rho^\delta r_t^\nu C^{2,\alpha}_t = \rho^\delta C^{2,\alpha}_t$ satisfying
\[
\mc{F}_t(\gamma) = 0.
\]
\end{thm}
There are several lemmas to prove before we give the proof of this theorem. First, we decompose the non-linear operator $\mc{F}_t$ into a zero-order term, the linearization, and a remainder:
\[
\mc{F}_t(\gamma) = \mc{F}_t(0) + L_t \gamma + Q_t(\gamma).
\]
Here $Q_t(\gamma)$ is defined by this equation. Furthermore, we define $\err_t = \mc{F}_t(0)$. Observe that the equation $\mc{F}_t(\gamma) = 0$ can then be rewritten as the fixed point problem
\[
\gamma = -L_t^{-1} \left( \err_t + Q_t(\gamma) \right).
\]
The strategy is to show that the right-hand side is a contraction and therefore that the existence of a solution will be a consequence of the Banach fixed point theorem.

According to Theorems \ref{thm:banach_space_iso} and \ref{thm:inverse_bounds_hoelder},
\[
L_t^{-1} : \rho^\delta r_t^{\nu-2} C^{0,\alpha}_t \to \rho^\delta r_t^\nu C^{2,\alpha}_t
\]
is a Banach space isomorphism, and its operator norm is bounded by $C t^\kappa$, where $C, \kappa > 0$ are some constants.

The remaining task is to show that $\err_t$ is small in $\rho^\delta r_t^{\nu-2} C^{0,\alpha}_t$ and that $Q_t : \rho^\delta r_t^\nu C^{2,\alpha}_t \to \rho^\delta r_t^{\nu-2} C^{0,\alpha}_t$ behaves like a quadratic term.

\begin{lemma}
There exist $C > 0$ and $\alpha > 0$ such that
\[
\|\err_t\|_{\rho^\delta r_t^{\nu-2} C^{0,\alpha}_t} \leq C e^{-\alpha t}.
\]
\end{lemma}
\begin{proof}
The term $\err_t$ is non-zero only in the regions where the exact solutions are interpolated. Near the core loops, this corresponds to the region $1/4 <  |\Re z| < 1/2$ in the coordinates $(V_k, z)$, and near the zeros of the quadratic differential, this is the set $1/4 < |z| < 1/2$ in the coordinates $(U_j, z)$.

Near the core loops in the coordinates $(V_k, z)$, the model solution converges to the limiting configuration at an exponential rate in every $C^k$ norm in the region $\epsilon < |\Re z| < 1$ for every fixed $\epsilon > 0$. Likewise, near the zeros of the quadratic differential in the coordinates $(U_j, z)$, the fiducial solution converges to the limiting configuration at an exponential rate in every $C^k$-norm in the region $\epsilon < |z| < 1$ for every fixed $\epsilon > 0$.
\end{proof}

\begin{lemma}
\label{lemma:Q_t_decay}
If $\delta > 0$ and $\gamma \in \rho^\delta r_t^\nu C^{k,\alpha}_t$, then $Q_t(\gamma) \in \rho^{2\delta} r_t^{2\nu -2} C^{k-2,\alpha}_t$.
\end{lemma}
\begin{proof}
In \cite{MSWW} it is shown that
\begin{align*}
Q_t(\gamma) = & d_{A_t^{\appr}} (R_{A_t^{\appr}}(\gamma)) + t^2 \left[ R_{\Phi_t^{\appr}}(\gamma) \wedge \left(\Phi_t^{\appr}\right)^*\right] + t^2 \left[ \Phi_t^{\appr} \wedge R_{\Phi_t^{\appr}}(\gamma)^* \right] \\
& + \frac 1 2 \left[ \left( (\ol{\partial}_{A_t}^{\appr} - \partial_{A_t}^{\appr}) \gamma + R_{A_t}^{\appr}(\gamma) \right) \wedge \left( (\ol{\partial}_{A_t}^{\appr} - \partial_{A_t}^{\appr}) \gamma + R_{A_t}^{\appr}(\gamma) \right) \right] \\
& + t^2 \left[ \left( [\Phi_t^{\appr}, \gamma] + R_{\Phi_t^{\appr}}(\gamma) \right) \wedge  \left( [\Phi_t^{\appr}, \gamma] + R_{\Phi_t^{\appr}}(\gamma) \right)^* \right],
\end{align*}
where $R_A(\gamma)$ and $R_\Phi(\gamma)$ are defined by the equations
\begin{align*}
A \ast e^\gamma & = A + (\ol{\partial}_A - \partial_A) \gamma + R_A(\gamma),\\
\Phi \ast e^\gamma & = \Phi + [\Phi, \gamma] + R_\Phi(\gamma).
\end{align*}
(Note that on the left-hand side we omit the Hodge star, which ensures that $Q_t(\gamma)$ is again a section of the endomorphism bundle rather than an endomorphism-valued 2-form.) The terms $R_A(\gamma)$ and $R_\Phi(\gamma)$ can also be written as
\begin{align*}
R_A(\gamma) & = \exp(-\gamma)\ol{\partial}_A(\exp(\gamma)) - \partial_A(\exp(\gamma)) \exp(-\gamma) - (\ol{\partial}_A - \partial_A) \gamma, \\
R_\Phi(\gamma) & = \exp(-\gamma) \Phi \exp(\gamma) - [\Phi, \gamma] - \Phi.
\end{align*}
Expanding the first few terms of the exponentials of $\gamma$ and $-\gamma$, one sees that
\begin{align*}
R_A(\gamma) & =  \frac 1 2 \left(\ol{\partial}_A \gamma - \partial_A \gamma \right) \gamma - \frac 1 2 \gamma \left( \ol{\partial}_A \gamma - \partial_A \gamma \right),   \\
R_\Phi(\gamma) & = \frac 1 2 \gamma^2 \Phi + \frac 1 2 \Phi \gamma^2 - \gamma \Phi \gamma.
\end{align*}
Let us observe that for $\eta_i \in \rho^{\delta_i} r_t^{\nu_i} C^{k,\alpha}_t$, $i=1,2$, the product $\eta_1 \eta_2$ lies in $\rho^{\delta_1+\delta_2} r_t^{\nu_1 + \nu_2} C^{k,\alpha}_t$. Moreover, $\Phi_t^{\appr}$ lies in $\rho^0 r_t^0 C_t^{k,\alpha}$. To see this, observe that the leading term in $\Phi_t^{\appr}$ at a core loop is $\begin{pmatrix} 0 & \frac 1 {2tx} \\ 2tx & 0 \end{pmatrix} dz$. Since $|dz| = \frac 1 {\sqrt{2}} |x|$, it follows that $|\Phi_t^{\appr}| = O(1)$ near the core loop, and for this reason $\Phi_t^{\appr}$ lies in $\rho^0 r_t^0 C^{k,\alpha}_t$ as claimed.

On the other hand, $d_{A_t^{\appr}}$ maps $\rho^\delta r_t^{\nu} C^{k,\alpha}_t$ to $\rho^\delta r_t^{\nu-1} C^{k-1,\alpha}_t$, and likewise for $\partial_{A_t^{\appr}}$ and $\ol{\partial}_{A_t^{\appr}}$. The argument for this is similar to that for $\Phi$.

Therefore, $R_{A_t^{\appr}}(\gamma) \in \rho^{2\delta} r_t^{2\nu - 1} C^{k-1, \alpha}_t$ and $R_{\Phi_t^{\appr}}(\gamma) \in \rho^{2\delta} r_t^{2\nu} C^{k,\alpha}_t$.

The claim now follows from the expression for $Q_t(\gamma)$ by further applying these properties.
\end{proof}

\begin{lemma}
\label{lemma:quadratic_term}
There exist a constant $C > 0$ and an $r > 0$ such that for $\gamma_1, \gamma_2 \in B_r(0) \subset \rho^\delta r_t^\nu C^{2,\alpha}_t$, the following inequality holds:
\[
\|Q_t(\gamma_1) - Q_t(\gamma_2)\|_{\rho^\delta r_t^{\nu-2} C^{0,\alpha}_t} \leq C \left( \|\gamma_1\|_{\rho^\delta r_t^\nu C^{2,\alpha}_t} + \|\gamma_2\|_{\rho^\delta r_t^\nu C^{2,\alpha}_t} \right) \|\gamma_1 - \gamma_2\|_{\rho^\delta r_t^\nu C^{2,\alpha}_t}.
\]
\end{lemma}
\begin{proof}
This follows from the fact that $Q_t : \rho^\delta r_t^\nu C^{2,\alpha}_t \to \rho^\delta r_t^{\nu-2} C^{0,\alpha}_t$ is a smooth map and that $dQ_t(0) = 0$.

Indeed, in that case we may write
\[
Q_t(\gamma_1) - Q_t(\gamma_2) = \int_0^1 dQ_t(\gamma_s) (\gamma_2 - \gamma_1) \, ds,
\]
where $\gamma_s = (1-s) \gamma_1 + s \gamma_2$. We may estimate
\[
\|d Q_t(\gamma_s)\| = \|dQ_t(\gamma_s) - d Q_t(0)\| \leq C \|\gamma_s - 0\| \leq C ((1-s) \|\gamma_1\| + s \|\gamma_2\|),
\]
where we use the local Lipschitz continuity of $dQ_t$, which follows from the smoothness of $Q_t$. Using the integral formula above then shows
\[
\|Q_t(\gamma_1) - Q_t(\gamma_2)\| \leq \frac 1 2 C \left( \|\gamma_1\| + \|\gamma_2\| \right) \|\gamma_2 - \gamma_1\|,
\]
as claimed.

To show the smoothness of $Q_t : \rho^\delta r_t^\nu C^{2,\alpha}_t \to \rho^\delta r_t^{\nu-2} C^{0,\alpha}_t$, let us first observe that $Q_t$ is a non-linear differential operator with smooth coefficients, and therefore smoothness of the operator between these Banach spaces is equivalent to well-definedness of the operator, i.e.\ we have to show that for $\gamma \in \rho^\delta r_t^\nu C^{2,\alpha}_t$, the image $Q_t(\gamma)$ is indeed a section of $\rho^\delta r_t^{\nu-2} C^{0,\alpha}_t$. This follows immediately from the previous lemma, since $\rho^{2\delta} r_t^{2\nu-2} C^{0,\alpha}_t \subset \rho^{\delta} r_t^{\nu-2} C^{0,\alpha}_t$.
\end{proof}

\begin{proof}[Proof of Theorem \ref{thm:existence_solution}]
For sufficiently large $t$ and appropriate $r > 0$, the map
\[
G_t : B_r(0) \subset \rho^\delta r_t^\nu C^{2,\alpha}_t \to \rho^\delta r_t^\nu C^{2,\alpha}_t, \quad G_t(\gamma) =  -L_t^{-1} \left( \err_t + Q_t(\gamma)\right)
\]
is a contraction mapping, and the Banach fixed point theorem then provides a solution. To see this, observe that according to Lemma \ref{lemma:quadratic_term} we have
\[
\|G_t(\gamma_1) - G_t(\gamma_2)\|_{\rho^\delta r_t^\nu C^{2,\alpha}_t} \leq C t^\kappa \left( \|\gamma_1\|_{\rho^\delta r_t^\nu C^{2,\alpha}_t} +  \|\gamma_2\|_{\rho^\delta r_t^\nu C^{2,\alpha}_t}\right)  \|\gamma_1 - \gamma_2\|_{\rho^\delta r_t^\nu C^{2,\alpha}_t}.
\]
Therefore, if we choose $r < \frac 1 {4 Ct^\kappa}$, we obtain
\[
\|G_t(\gamma_1) - G_t(\gamma_2)\|_{\rho^\delta r_t^\nu C^{2,\alpha}_t} \leq \frac 1 2  \|\gamma_1 - \gamma_2\|_{\rho^\delta r_t^\nu C^{2,\alpha}_t}.
\]
In particular, $\|G_t(\gamma) - G_t(0)\|_{\rho^\delta r_t^\nu C^{2,\alpha}_t} \leq \frac 1 2 \|\gamma\|_{\rho^\delta r_t^\nu C^{2,\alpha}_t}$ and therefore
\begin{align*}
\|G_t(\gamma)\|_{\rho^\delta r_t^\nu C^{2,\alpha}_t} & \leq \|G_t(\gamma) - G_t(0)\|_{\rho^\delta r_t^\nu C^{2,\alpha}_t} + \|G_t(0)\|_{\rho^\delta r_t^\nu C^{2,\alpha}_t} \\
& \leq \frac 1 2 r + \wt{C} t^\kappa e^{-a t}.
\end{align*}
For $t$ sufficiently large, this implies that $G_t$ maps the ball $B_r(0)$ to itself. This shows that $G_t : B_r(0) \to B_r(0)$ is indeed a contraction mapping, and the Banach fixed point theorem shows that there exists $\gamma \in B_r(0)$ solving $G_t(\gamma) = 0$.
\end{proof}

\subsection{Decay behavior at the core loops and construction of maps through $\rho=0$}

\begin{thm}
  \label{thm:best_gluing_regularity}
If $\gamma \in \rho^\delta C^{k,\alpha}_t$ solves $\mc{F}_t(\gamma) = 0$ for some $\delta > 0$, then $\gamma \in \rho^\mu C^{k,\alpha}_t$ for every $\mu < 2.$
\end{thm}
\begin{proof}
Rewrite $\mc{F}_t(\gamma) = 0$ as
\[
L_t \gamma = - \err_t - Q_t(\gamma).
\]
Now $\err_t$ vanishes near the boundary and is smooth. Therefore $\err_t \in \rho^\eta C^{l,\beta}_0$ for all choices of $\eta, l, \beta$.

On the other hand, assuming that $\gamma \in \rho^\delta C^{k,\alpha}_t$, Lemma \ref{lemma:Q_t_decay} gives that $Q_t(\gamma) \in \rho^{2\delta} C^{k-2,\alpha}_t$. Therefore $L_t \gamma \in \rho^{2\delta} C^{k-2,\alpha}_t$ and this implies that $\gamma \in \rho^{2\delta} C^{k,\alpha}_t$, as long as $2\delta < 2$, since in that range $L_t : \rho^\mu C^{k,\alpha}_t \to \rho^\mu C^{k-2,\alpha}_t$ is an isomorphism.

Repeating this argument $n$ times yields $\gamma \in \rho^{2^n \delta} C^{k,\alpha}_t$, as long as $2^n \delta <  2$.

Let $n_0$ be the largest natural number such that $2^{n_0} \delta <  2$. Since $\gamma \in \rho^{2^{n_0} \delta} C^{k,\alpha}_t$, it follows by the previous argument that $L_t \gamma \in \rho^{2^{n_0+1} \delta} C^{k-2,\alpha}_t$. In particular, it follows that $L_t \gamma \in \rho^\mu C^{k-2, \alpha}_t$ for every $\mu < 2$. From this we conclude that $\gamma \in \rho^\mu C^{k,\alpha}_0$, giving the desired result.
\end{proof}

With this theorem in place, we can study the regularity of the associated harmonic map and its associated map into $\mathbb{S}^3$. Before we do so, we briefly discuss properties of the harmonic map and the transgressive map associated to the model solution. To this end, recall that the flat $\SL(2,\C)$ connection associated to the model solution $(A_t^{\model}, \Phi_t^{\model})$ is given by
\begin{align}\label{formula:D_t_model}
  D_t^{\model} & = d_{A_t^{\model}} + t \Phi_t^{\model} + \left(t \Phi_t^{\model}\right)^* \\
  & = d + \;
  \begin{pmatrix}
    0 & \frac 1 2 t \tanh(t x) + \frac 1 2 t \coth(tx) \\
    \frac 1 2 t \tanh(tx) + \frac 1 2 t \coth(tx) & 0
  \end{pmatrix}
  dx \notag \\
  &
  \quad \; +
  i \begin{pmatrix}
    - t \csch(2 t x) & - \frac 1 2 t \tanh(tx) + \frac 1 2 t \coth(t x)\\
    \frac 1 2 t \tanh(tx) - \frac 1 2 t \coth(tx) & t \csch(2 t x )
  \end{pmatrix} dy \notag
\end{align}
and we can compute an explicit parallel frame for this as
\begin{equation}\label{formula:F_t_model}
  F_t^{\model}(x,y) =
  \frac 1 2
  \begin{pmatrix}
    h + \frac 1 h & -h + \frac 1 h \\ -h + \frac 1 h & h + \frac 1 h
  \end{pmatrix}
  \begin{pmatrix}
    1 + \frac i 2 y & - \frac i 2 y \\
    \frac i 2 y & 1 - \frac i 2 y
  \end{pmatrix}
\end{equation}
with $h = \sqrt{\frac{\sinh(2 t x)}{2t}}$. Note that both the frame and the connection are not defined at $x=0$. The harmonic map into the hyperbolic 3-space is then given by
\begin{equation}
  f_{t,\hyp}^{\model} =
  \begin{cases}
    - \left(F_t^{\model}\right)^* F_t^{\model}, & x < 0 \\
    \quad \left(F_t^{\model}\right)^* F_t^{\model}, & x > 0 \\    
  \end{cases}
\end{equation}
The signs in this equation are explained by the fact that the Hermitian metric on $x<0$ is negative definite and on $x>0$ is positive definite. Applying an isometry of $\H^3_\pm$ to this family yields more solutions. Such a translation of the model solution is given by $A^* f_{t,\hyp}^{\model} A$ for $A \in \SL(2,\C)$. The transgressive map associated to the model solution is given by $f_t^{\model} = \Xi^{-1}_{\SU(2)} \circ f_{t,\hyp}^{\model}$. Note that these maps are odd with respect to the variable $x$.

\begin{thm}
  \label{thm:existence_hoelder_transgressive_maps}
  For any framed limiting configuration and for sufficiently large $t > 0$ and any $\delta < 2$, there exists a $\gamma_t \in \rho^\delta C^{2,\alpha}_t$ solving $\mc{F}_t(\gamma_t) = 0$.
  
  If $\delta > 1$ and the complement of the union $\cup_k c_k \subset M$ of the core loops consists of two connected components, then associated to this solution is an equivariant map $f_t : \wt{M} \to \bb{S}^3$, such that $f_t$ is smooth away from the (preimage of) the core loops, and $\Xi \circ f_t$ is a harmonic map to $\H^3_\pm$ on this complement. Moreover, for every core loop there exists a neighborhood of the core loop and an $A \in \SL(2,\C)$, such that $f_t - \Xi_{\SU(2)}^{-1} \circ \left( A^* f_{t,\hyp}^{\model} A\right) \in \rho^{\delta - 1} C^{2,\alpha}_t$.

  In particular, the map $f_t$ is H\"older continuous.
\end{thm}
\begin{rem}
  Note that the map $f_t$ is not necessarily transgressive, because the definition requires that the map is at least $C^1$ through the equatorial $S^2$.

  In the next section, we will show that a specific class of symmetric solutions is smooth through the equator. It would be very interesting to understand necessary and sufficient conditions for smoothness through the equator.
\end{rem}
\begin{proof}
  Existence of the solution $\gamma_t$ follows from Theorem \ref{thm:existence_solution} and the fact that we can choose any $\delta < 2$ follows from Theorem \ref{thm:best_gluing_regularity}. Away from $\rho = 0$ the section $\gamma_t$ is smooth by elliptic regularity.

  The connections $D_t^{\model}$ and the associated frames $F_t^{\model}$ are not defined at $x=0$. The same holds for the solution constructed via the gluing procedure. This singularity can be gauged away by a singular gauge transformation. Indeed, consider the singular gauge transformation
  \[
  g = \frac i {\sqrt{2}} \begin{pmatrix} -\sqrt{x} & \frac 1 {\sqrt{x}} \\ \sqrt{x} & \frac 1 {\sqrt{x}} \end{pmatrix}.
  \]
  This gauge transformation turns $D_t^{\model}$ into a connection that is smooth through $x=0$:
  \begin{align*}
    \wt{D}_t^{\model} = D_t^{\model} \cdot g & = d +  \begin{pmatrix} \frac 1 {2x} - \frac {t}{2 \tanh(tx)} - \frac 1 2 t \tanh(tx) & 0 \\ 0 & \frac t {2 \tanh(tx)} - \frac 1 {2x} + \frac 1 2 t \tanh(tx) \end{pmatrix} dx \\
    & + \begin{pmatrix} 0 & 0 \\ \frac {i t x}{\cosh(tx) \sinh(tx)} & 0 \end{pmatrix} dy.
  \end{align*}
  Observe that $\wt{F}_t^{\model} = g^{-1} F_t^{\model}$ is a parallel frame for $\wt{D}_t^{\model}$ and that this frame is smooth through $x=0$.

    Denote by $D_t^{\appr}$ the connection $d + A_t^{\appr} + t \Phi_t^{\appr} + t \left(\Phi_t^{\appr}\right)^*$. Then $D_t = D_t^{\appr} \cdot g_t$ with $g_t = e^{\gamma_t}$ is a flat connection.

    By employing a cutoff function, we can extend the gauge transformation $g$ to the Riemann surface $M$. By abuse of notation we will still call this gauge transformation $g$. The connection $\wt{D}_t = D_t \cdot g$ is a connection defined on all of the Riemann surface, which is flat on $\{\rho \neq 0\}$.

    Near any core loop this connection can be written as $D_t^{\model} \cdot g_t \cdot g$ and therefore $g^{-1} g_t^{-1} F_t^{\model}$ is a parallel frame for $\wt{D}_t$ near such a core loop. More generally, for any $A \in \SL(2,\C)$ the frame $g^{-1} g_t^{-1} F_t^{\model} A$ is parallel.

    Fix one core loop and in a neighborhood of this core loop consider the frame $g^{-1}g_t^{-1} F_t^{\model}$. By analytic continuation we may extend this parallel frame to the universal cover $\wt{M}$ of $M$. We denote this frame by $\wt{F}_t$. From this we recover a parallel frame $F_t$ for $D_t$ by the formula $F_t = g \wt{F}_t$, which is defined on the complement of the core loop lifted to $\wt{M}$.

    Near any other core loop the parallel frame $\wt{F}_t$ will be of the form $g^{-1} g_t^{-1} F_t^{\model} A$ for some $A \in \SL(2,\C)$. This follows from uniqueness of parallel frames once an initial condition at any point is fixed.

  By definition of a framed limiting configuration, the complement of the core loops $M^\vee$ has exactly two components $M_+$ and $M_-$. Here, we define $M_+$ and $M_-$, such that with respect to the coordinates near the core loops $M_+$ corresponds to the $\{x > 0\}$ pieces and $M_-$ corresponds to the $\{x < 0\}$ pieces.

  The associated harmonic map corresponding to $F_t$ into $\H^3_\pm$ is then given by
  \[
  \begin{cases}
    -F_t^* F_t, & \text{on } M_- \\
    \;\;\, F_t^* F_t, & \text{on } M_+
  \end{cases}
  \]
  The sign results from the fact that the associated Hermitian metric is positive definite on $M_+$ and negative definite on $M_-$.

  Observe that $F_t^* F_t  = \wt{F}_t^* g^* g \wt{F}_t$ and that $g^*g = \begin{pmatrix} x & 0 \\ 0 & 1/x \end{pmatrix}$.

  The regularity of $\wt{F}_t$ near any core loop can be computed from the regularity of $\gamma_t$ as follows. First, observe that $g_t^{-1} = e^{-\gamma_t} = \id - \gamma_t + \frac 1 2 \gamma_t^2 + \ldots = \id + \eta_t$, where $\eta_t \in \rho^\delta C^{2,\alpha}_t$. Using this, we see that $\wt{F}_t = g^{-1} g_t^{-1} F_t^{\model} A = g^{-1} F_t^{\model} A + g^{-1} \eta_t F_t^{\model} A$, where $A \in \SL(2,\C)$. Since $g^{-1} F_t^{\model}$ is smooth and $A$ is constant, it follows that we only need to understand the regularity of $g^{-1} \eta_t F_t^{\model}$. The term of lowest order in $g^{-1}$ is proportional to $x^{-1/2}$, and likewise in the expansion of $F_t^{\model}$. Therefore $\beta_t = g^{-1} \eta_t F_t^{\model} \in \rho^{\delta - 1 } C^{2,\alpha}_t$. 

  Therefore
  \begin{align*}
    F_t^* F_t & = A^* (\wt{F}_t^{\model} + \beta_t)^* g^*g (\wt{F}_t^{\model} + \beta_t) A \\
    & = A^* \left(F_t^{\model}\right)^* F_t^{\model} A + A^* \left(\wt{F}_t^{\model}\right)^* g^*g \beta_t A +  A^* \beta_t^* g^*g \wt{F}_t^{\model} A + A^* \beta_t^* g^* g \beta_t A.
  \end{align*}
  Since $g^* g = \frac 1 x \begin{pmatrix} x^2 & 0 \\ 0 & 1 \end{pmatrix}$, we may write
  \[
  A^* \left(F_t^{\model}\right)^* F_t^{\model} A = \frac 1 x A^* \left(\wt{F}_t^{\model}\right)^* \begin{pmatrix} x^2 & 0 \\ 0 & 1 \end{pmatrix} \wt{F}_t^{\model} A = \frac 1 x U_t,
  \]
  where $U_t$ is defined by this equation. The map $U_t$ is smooth. Similarly, we may define $V_t$ via
  \begin{align*}
    F_t^* F_t - A^* \left(F_t^{\model}\right)^* F_t^{\model} A & = \frac 1 x A^* \left( \left(\wt{F}_t^{\model}\right)^* \begin{pmatrix} x^2 & 0 \\ 0 & 1 \end{pmatrix} \beta_t + \beta_t^* \begin{pmatrix} x^2 & 0 \\ 0 & 1 \end{pmatrix} \wt{F}_t^{\model} + \beta_t^* \begin{pmatrix} x^2 & 0 \\ 0 & 1 \end{pmatrix} \beta_t \right) A \\
    & = \frac 1 x V_t
  \end{align*}
  and since $\beta_t \in \rho^{\delta-1} C^{2,\alpha}_t$ we find that $V_t \in \rho^{\delta-1} C^{2,\alpha}_t$ as well. Note that $F_t^* F_t = \frac 1 x \left( U_t + V_t\right)$.

  We now can define a map into $\mathbb{S}^3$ using the embedding of the matrix model of $\H^3_\pm$ into $\mathbb{S}^3$, which is identified with $\SU(2)$. This map is given by
  \[
  \Xi_{\SU(2)}^{-1} : \left\{ A \in \SL(2,\C) : A = A^* \right\} \to \SU(2), \qquad  \Xi_{\SU(2)}^{-1}(A) = \frac 2 {\tr(A)} \left( \id + i \mathring{A} \right).
  \]
  The map into $\bb{S}^3$ associated to our solution is then $f_t = \Xi_{\SU(2)}^{-1} \circ F_t^* F_t$ and may be written as
  \begin{align*}
   \Xi_{\SU(2)}^{-1} \circ F_t^* F_t & = \frac 2 { \frac 1 x \tr\left(U_t + V_t\right)} \left( \id + i \frac 1 x \left(U_t + V_t\right)^{\circ} \right) \\
  & = \frac 2 { \frac 1 x \tr\left(U_t + V_t\right)} \left( \id + i \frac 1 x \mathring{U_t} \right)  +  \frac 2 { \tr\left(U_t + V_t\right)} \mathring{V_t}.
  \end{align*}
  Now
  \[
  \frac 2 { \frac 1 x \tr\left(U_t + V_t\right)} = \frac {2x} {\tr U_t} \frac 1 { 1 + \frac {\tr V_t}{\tr U_t}}.
  \]
  By explicit calculation one checks that $\tr U_t$ is smooth and bounded away from $0$. On the other hand,
  \[
  \frac 1 { 1 + \frac {\tr V_t}{\tr U_t}} = 1 - \frac{\tr V_t }{\tr U_t} + \left( \frac {\tr V_t}{\tr U_t} \right)^2 - \ldots.
  \]
  Therefore, $\frac 1 { 1 + \frac {\tr V_t}{\tr U_t}} - 1$ behaves like $\tr V_t$, i.e.\ it lies in $\rho^{\delta-1} C^{2,\alpha}_t$.

  Denote $f_t^{\model} = \Xi_{\SU(2)}^{-1} \circ \left(F_t^{\model}\right)^* F_t^{\model}$. Now $\frac {2 x}{ \tr U_t} \left( \id + i \frac 1 x \mathring{U}_t\right)$ is the transgressive harmonic map into $\bb{S}^3$ associated to the model solution (translated by $A$), that is, it is $\Xi_{\SU(2)}^{-1} A^* f_{t,\hyp}^{\model} A$. This map is smooth through $x = 0$.

  For the remainder
  \begin{align*}
  f_t - f_t^{\model} & =  \Xi_{\SU(2)}^{-1} \circ F_t^* F_t - \frac {2 x}{ \tr U_t} \left( \id + i \frac 1 x  \mathring{U}_t\right) \\
  & =
  \left( 1 - \frac 1 { 1 + \frac {\tr V_t}{\tr U_t}} \right) \frac {2 x}{ \tr U_t} \left( \id + i \frac 1 x \mathring{U}_t\right) + \frac {2} {\tr U_t} \frac 1 { 1 + \frac {\tr V_t}{\tr U_t}} \mathring{V}_t
  \end{align*}
  our previous analysis now yields that the first term is in $\rho^{\delta - 1} C^{2,\alpha}_t$, whereas the second term is also in $\rho^{\delta - 1} C^{2,\alpha}_t$.

  In summary, we have shown that $f_t$ is smooth away from $\rho=0$ and that in a neighborhood of a core loop $f_t - \Xi_{\SU(2)}^{-1} \circ A^*f_{t,\hyp}^{\model}A$ is in $\rho^{\delta - 1} C^{2,\alpha}_t$. This implies that $f_t$ is H\"older continuous, by Lemma 3.7, \cite{Lee}.
\end{proof}

As we noted, the maps $f_t : \wt{M} \to \bb{S}^3$ are not necessarily transgressive. In this sense, we cannot immediately apply the oblique Gau\ss\@ map construction of Theorem \ref{thm:transobl}. However, the gauge theoretical construction of this map (see formula (\ref{eq:gaugeobliqueN})) extends to this setting. In the next proposition we discuss the regularity of this associated map. Let us recall this construction in the case of the model solution. According to formula (\ref{eq:gaugeobliqueN}) the associated dual map is given by
\[
N_t^{\model} =
\begin{cases}
  -\left(F_t^{\model}\right)^* R_t^{\model} F_t^{\model}, & x < 0 \\
  \left(F_t^{\model}\right)^* R_t^{\model}  F_t^{\model}, & x > 0
\end{cases}
\]
where $R_t^{\model}$ is the orthogonal reflection on the eigenline spanned by $\begin{pmatrix} 1 \\ \tanh(tx) \end{pmatrix}$ for the eigenvalue $\frac t 2$. This reflection is given by
\[
R_t^{\model} = \frac 1 { 1 + \tanh(tx)^2} \begin{pmatrix} \tanh(tx)^2 - 1 & - 2 \tanh(tx) \\ - 2 \tanh(tx) & 1 - \tanh(tx)^2 \end{pmatrix}.
\]
Note that formula (\ref{eq:gaugeobliqueN}) is written with respect to a frame that is adapted to the eigenline and its orthogonal complement. In our case we work with the standard frame and therefore the reflection $R_t$ does not have the form $\begin{pmatrix} 1 & 0 \\ 0 & -1 \end{pmatrix}$.

\begin{cor}
  \label{cor:hoelder_dual_maps}
  With the assumptions of the second part of Theorem \ref{thm:existence_hoelder_transgressive_maps}, there exists an equivariant map $N_t : \wt{M} \to \dS_3$ with the following properties:
  \begin{enumerate}[(i)]
  \item $N_t$ is smooth and harmonic away from the preimages of the core loops,
  \item $N_t$ is the oblique hyperbolic Gau\ss\@ map associated to $\Xi_{\SU(2)} \circ f_{t,\hyp},$
  \item for every core loop there is a neighborhood of the core loop and an $A \in \SL(2,\C)$, such that $N_t - A^* N_t^{\model} A \in \rho^{\delta - 1 } C^{2,\alpha}_t$ in that neighborhood.
  \end{enumerate}
\end{cor}
\begin{proof}
  Let $F_t$ and $\wt{F}_t$ be as in the proof of Theorem \ref{thm:existence_hoelder_transgressive_maps}. The map $N_t$ is given by $N_t = F_t^* R_t F_t$, where $R_t$ is the reflection along the eigenline of $\Phi_t$. This eigenline is the kernel of $\Phi - \omega \id$.

  The first two items are immediate from Theorem \ref{thm:existence_hoelder_transgressive_maps} and formula (\ref{eq:gaugeobliqueN}).

  Let $\gamma_t \in \rho^\delta C^{2,\alpha}_t$ be the solution of $\mc{F}_t(\gamma_t) = 0$ and $g_t = e^{\gamma_t}$. From now on we work in a neighborhood of a core loop. Since $\Phi_t = g_t^{-1} \Phi_t^{\model} g_t$, the eigenline for $\Phi_t$ is spanned by $g_t^{-1} \begin{pmatrix} 1 \\ \tanh(tx) \end{pmatrix}$.

Therefore, the reflection $R_t$ is given by $R_t = g_t R_t^{\model} g_t^{-1}$.

Since $g_t = \id + \theta_t$ with $\theta_t \in \rho^\delta C^{2,\alpha}_t$, it follows that $R_t - R_t^{\model} \in \rho^\delta C^{2,\alpha}_t$. Observe that
\begin{align*}
  N_t & = F_t^* R_t F_t \\
  & = \wt{F}_t^* g^* R_t g \wt{F}_t \\
  & = \wt{F}_t^* g^* R_t^{\model} g \wt{F}_t + \wt{F}_t^* g^* (R_t - R_t^{\model}) g \wt{F}_t,
\end{align*}
where $g$ is as in the proof of the previous theorem. The identity
\[
g^* R_t^{\model} g =
\begin{pmatrix}
  2 x \tanh(tx) & \sech(tx)^2 \\
  \sech(t x)^2 & - 2 \frac{\tanh(t x)}{x}
\end{pmatrix}
\]
shows that $g^*R_t^{\model} g$ is smooth through $x=0$. From the proof of the last theorem we know that $\beta_t = \wt{F}_t - \wt{F}_t^{\model} A \in \rho^{\delta - 1} C^{2,\alpha}_t$. On the other hand, since $g \in \rho^{-1/2} C^\infty_t$ we obtain $g^* (R_t - R_t^{\model}) g \in \rho^{\delta - 1} C^{2,\alpha}_t$. Therefore the last term is also in $\rho^{\delta-1} C^{2,\alpha}_t$. Since $N_t^{\model} = \left(F_t^{\model}\right)^* R_t^{\model} F_t^{\model}$, we may write
\begin{align*}
  \wt{F}_t^* g^* R_t^{\model} g \wt{F}_t & = \left(\wt{F}_t^{\model} A + \beta_t\right)^* g^* R_t^{\model} g \left(\wt{F}_t^{\model} A + \beta_t\right) \\
  & = A^* N_t^{\model} A + \beta_t^* g^* R_t^{\model} g \wt{F}_t^{\model} A + \left(\wt{F}_t^{\model} A\right)^* g^* R_t^{\model} g \beta_t + \beta_t^* g^* R_t^{\model} g \beta_t.
\end{align*}
This shows that $\wt{F}_t^* g^* R_t^{\model} g \wt{F}_t - A^* N_t^{\model} A$ lies in $\rho^{\delta - 1} C^{2,\alpha}_t$. This concludes the proof that $N_t - A^* N_t^{\model} A$ lies in $\rho^{\delta - 1} C^{2,\alpha}_t$.
\end{proof}

\subsection{Regularity of the transgressive maps under symmetry}
As we saw in the last section, the map $f_t : \wt{M} \to \bb{S}^3$ associated to the singular solutions of the $\SU(2)$ self-duality equations is H\"older continuous through the equatorial 2-sphere, but not necessarily smooth. In this section we will show that under a certain symmetry assumption the transgressive maps are indeed smooth through the equatorial 2-sphere. Since we can find data for our gluing construction which satisfies this symmetry condition, we obtain new examples of smooth equivariant transgressive harmonic maps. The proof of this relies on elliptic regularity applied to the dual map into de Sitter 3-space. Indeed, we will show that under our symmetry assumption the dual map $N_t$ is even and that this property, together with the regularity known from the previous section, suffices to show that $N_t$ is a weak solution of the harmonic map equation.

The symmetry condition is a globalization of the reflection symmetry of the model solution. Let $r$ denote the reflection in the imaginary axis, i.e. $r : \C \to \C$, $r(x+iy) = -x+iy$, and
\[
\hat{r} : \C \times \C^2 \to \C \times \C^2, \qquad \hat{r}(z,v) = \left(r(z), \ol{v}\right).
\]
Then an explicit calculation shows that $A_t^{\model}$, $\Phi_t^{\model}$ and (consequently) $D_t^{\model}$ (see equation (\ref{formula:D_t_model})) are invariant under $r$.

This condition can be globalized to a Riemann surface with a framed limiting configuration $(A_\infty,\Phi_\infty)$ by the following data
\begin{enumerate}[(i)]
\item an antiholomorphic involution $\sigma :M \to M$,
\item an antilinear automorphism $\hat{\sigma} : E \to E$ covering $\sigma$
\end{enumerate}
satisfying the following conditions
\begin{enumerate}[(a)]
\item $(A_\infty, \Phi_\infty)$ is invariant under $\hat{\sigma}$,
\item the fixed point set of $\sigma$ is equal to the union of the core loops of the framed limiting configuration,
\item in the coordinates and frame around the core loops the map $\hat{\sigma}$ is equivalent to $\hat{r}$.
\end{enumerate}
Note that if we have an antiholomorphic involution $\sigma : M \to M$ and an antilinear automorphism $\hat{\sigma} : E \to E$ covering it, then given a $\hat{\sigma}$-invariant Higgs bundle $(E,\ol{\partial}_E, \varphi)$, we can construct such a framed limiting configuration. 

The parallel frame $F_t^{\model}$ has the symmetry
\[
r^* F_t^{\model} = -i u^{-1} F_t^{\model},
\]
where $u = \begin{pmatrix} 0 & 1 \\ 1 & 0 \end{pmatrix}$. The symmetry above implies that $f_{t,\hyp}^{\model}$ is odd in the variable $x$. Recall that to construct $N_t^{\model}$ we used the reflection $R_t^{\model}$ in the eigenline of $\Phi$. This reflection has the symmetry
\[
r^*R_t^{\model} = -u^{-1} R_t^{\model} u.
\]
This implies that $N_t^{\model} = (F_t^{\model})^* R_t^{\model} F_t^{\model}$ is even.

\begin{thm}
  \label{thm:existence_smooth_transgressive_maps}
  Suppose that $M$ is a compact Riemann surface and that we are given a framed limiting configuration with underlying Higgs bundle $(E, \ol{\partial}_E, \varphi)$ such that the complement of the union $\cup_k c_k \subset M$ of the core loops consists of two connected components.
  
  Suppose moreover that there is an involutive automorphism $\hat{\sigma} : E \to E$ covering an antiholomorphic involution $\sigma : M \to M$ satisfying the conditions (a)-(c) above. Then for sufficiently large $t > 0$, there exist solutions $\gamma_t$ of $\mc{F}_t(\gamma_t) = 0$, such that
  \begin{enumerate}
  \item the solutions $(A_t, \Phi_t) = (A_t^{\appr}, \Phi_t^{\appr}) \ast {\exp(\gamma_t)}$ are singular along the core loops given by the fixed point set,
  \item the associated equivariant harmonic maps extend to smooth equivariant harmonic transgressive maps $f_t : \wt{M} \to S^3$,
  \item the oblique Gau\ss\@ maps $N_t : \wt{M} \to \dS_3$ are smooth.
  \end{enumerate}
\end{thm}
\begin{rem}\label{rem:many-examples}
There exist many framed limiting configurations satisfying the conditions of the theorem, as we now explain. 
By the classification of real algebraic curves \cite{GrHa} of genus $g$, the fixed-point set of an antiholomorphic involution 
$\sigma$ on $M$ has $n \in \{1, \dots, g+1\}$ connected components, with $n \equiv g+1 \pmod 2$.
The complement of this fixed set must have exactly two components in our situation.

For $n=g+1$ and $n=1$ (when $g$ is even) or $n=2$ (when $g$ is odd), hyperelliptic curves provide natural examples:
those with Weierstrass points lying on the real axis (for $n=g+1$) or symmetrically reflected across the real axis but not on it (for $n=1$ or $n=2$) 
admit an antiholomorphic involution with the required number of fixed components and two complementary components. 
These hyperelliptic curves carry quadratic differentials $q$ with simple zeros away from the fixed locus of $\sigma$ such that 
$\sigma^*\bar q = q$. 

For $n=g+1$ or $n=2$ (and $g$ odd), there also exist real spin bundles $S \to M$, 
i.e.\ spin bundles admitting an antiholomorphic involutive lift of $\sigma$ (see \cite[Section~7]{Hi}). 
The corresponding point in the Hitchin section, with underlying holomorphic bundle $E = S \oplus S^*$, 
then provides a framed limiting configuration satisfying the conditions of Theorem~\ref{thm:existence_smooth_transgressive_maps}.
\end{rem}

\begin{lemma}
  \label{lemma:smoothness_N}
  Let $M$ be a Riemann surface, $\Gamma \subset M$ a smooth curve. Suppose that $N : M \to \dS_3$ is a map, such that $N|_{M \bs \Gamma}$ is smooth and harmonic,and such that there exist constants $C > 0$, $\nu > \frac 3 4$, a tubular neighborhood $U(\Gamma)$ of $\Gamma$ and coordinates
  \[
  (x,y) : U(\Gamma) \subset M \to (-1,1) \times \R/(r \Z)
  \]
  with the following properties:
  \begin{enumerate}
  \item $|N(x,y) - N(0,y)| \leq C |x|^{\nu}$,
  \item $|\partial_x N(x,y)|, |\partial_y N(x,y)| \leq C |x|^{\nu-1}$,
  \item $N(x,y) = N(-x,y)$.
  \end{enumerate}
  Then $N$ is smooth and harmonic through $x=0$.
\end{lemma}
\begin{proof}
  On $\{ \rho\neq 0\}$ the map $N$ satisfies the harmonic map equation
  \[
  \Delta N = |dN|^2 N.
  \]
  We claim that this equation holds weakly on the whole domain, i.e.\ for any smooth $\varphi$ with compact support we have
  \[
  \int \langle N, \Delta \varphi \rangle \, d\vol = \int \langle |dN|^2 N, \varphi \rangle \, d\vol.
  \]
  Working in the coordinates of the model solution, where $\rho = x$, this follows from the computation
  \begin{align*}
\int \langle N, \Delta \varphi \rangle \, d\vol & = \lim_{\epsilon \to 0} \int_{\{|x| > \epsilon\}} \langle N, \Delta \varphi \rangle \, d\vol \\
& = \lim_{\epsilon \to 0} \left(  \int_{\{|x| > \epsilon\}} \langle \Delta N,  \varphi \rangle \, d\vol +  \int_{\{|x| = \epsilon\}} \langle \partial_x N, \varphi \rangle \, dy \right) \\
& = \lim_{\epsilon \to 0} \left( \int_{\{|x| > \epsilon\}} \langle |dN|^2 N,  \varphi \rangle \, d\vol +  \int_{\{|x| = \epsilon\}} \langle \partial_x N, \varphi \rangle \, dy \right)\\
& = \int \langle |dN|^2 N, \varphi \rangle \, d\vol,
  \end{align*}
  where we used $\lim_{\epsilon \to 0} \int_{\{|x| = \epsilon\}} \langle \partial_x N, \varphi \rangle \, dy = 0$. To see that this is the case, first observe that $\partial_x N$ is odd. Since $\varphi$ is smooth, we may write $\varphi(x,y) = \varphi_0(y) + x \varphi_1(x,y)$ for smooth functions $\varphi_0$ and $\varphi_1$.

  Then we compute
  \begin{align*}
    \int_{\{|x| = \epsilon\}} \langle \partial_x N, \varphi \rangle \, dy & = \int \big\langle \partial_x N(\epsilon,y), \varphi(\epsilon,y) \big\rangle + \big\langle \partial_x N(-\epsilon, y), \varphi(-\epsilon,y) \big\rangle \, dy \\
    & = \int \left\langle \partial_x N(\epsilon,y), \varphi(\epsilon, y) - \varphi(-\epsilon, y) \right\rangle dy \\
    & = \int \left\langle \partial_x N(\epsilon,y), 2 \epsilon \varphi_1(x,y) \right\rangle dy.
  \end{align*}
  By assumption $|\partial_x N(\epsilon,y)| \leq C \epsilon^{\nu-1}$ and therefore
  \[
  \left| \int \left\langle \partial_x N(\epsilon,y), 2 \epsilon \varphi_1(x,y) \right\rangle dy \right| \leq 2 C \epsilon^\nu.
  \]
  This implies 
  \[
  \lim_{\epsilon \to 0} \int_{\{|x| = \epsilon\}} \langle \partial_x N, \varphi \rangle \, dy = 0
  \]
  and this finishes the argument that $N$ is a weak solution of the harmonic map equation.

  Since $\nu > \frac 3 4$, the condition $|dN| \leq C |x|^{\nu-1}$ implies that $dN$ is locally in $L^{2 p}$ for some $p > 2$. Since $N$ is also H\"older continuous, this implies that $|dN|^2 N$ is locally in $L^p$. Elliptic regularity then implies that $N$ is in $W^{2,p}$, which embeds in $C^{1,\alpha}$ for $\alpha = 1 - \frac 2 p$. 

  Then $|dN|^2 N$ is in $C^{0,\alpha}$ and it follows by elliptic regularity that $N \in C^{2,\alpha}$. This in turn gives that $|dN|^2 N \in C^{1,\alpha}$ and therefore $N \in C^{3,\alpha}$. Repeating this process yields that $N$ is smooth as claimed.
\end{proof}

\begin{proof}[Proof of Theorem \ref{thm:existence_smooth_transgressive_maps}]
  Note that the involution $\sigma$ divides the Riemann surface into two components and one of those components corresponds to the regions $\{x > 0\}$ at the core loops. As in the proof of Theorem \ref{thm:existence_hoelder_transgressive_maps}, we denote this component by $M_+$ and the other component, corresponding to $\{x < 0\}$,by $M_-$. Let $\gamma_t \in \rho^{\delta} C^{2,\alpha}_t$ be a solution of $\mc{F}_t(\gamma_t) = 0$ with $\delta > \frac 7 4$. According to Theorem \ref{thm:existence_hoelder_transgressive_maps} this solution induces an equivariant map $f_t : \wt{M} \to \bb{S}^3$, which is harmonic away from the preimages of the core loops, but a priori only $C^{\alpha}$ through the core loops. According to Corollary \ref{cor:hoelder_dual_maps} there exists a map $N_t : \wt{M} \to \dS_3$, which is dual to $f_t$ away from the preimages of the core loops. For a given core loop, there exists a neighborhood of the core loop and an $A \in \SL(2,\C)$, such that $N_t - A^* N_t^{\model} A \in \rho^{\delta - 1} C^{2,\alpha}_t$ in that neighborhood.

  Due to the symmetry assumption on the framed limiting configuration, we may assume that the solution $(A_t^{\appr}, \Phi_t^{\appr}) \ast e^{\gamma_t}$ satisfies the same symmetry condition. Indeed, if not,we may extend the gauge transformation from $M_+$ to $M_-$ by means of the involution $\hat{\sigma}$, enforcing the symmetry condition. In particular, in the coordinates of the core loop $N_t$ is even. Together with $N_t - A^* N_t^{\model} A \in \rho^{\delta-1} C^{2,\alpha}_t$ and since $\delta - 1 > \frac 3 4$, this implies that the conditions of Lemma \ref{lemma:smoothness_N} are met. Therefore the map $N_t$ is smooth as claimed.

  We want to apply Theorem \ref{thm:transgressive_map_from_de_sitter_map} to obtain a transgressive map from the map $N_t$. If we can apply this theorem, then this map must be $f_t$, since $f_t$ coincides with it on the dense set $\{\rho \neq 0\}$ and $f_t$ is continuous. To see that Theorem \ref{thm:transgressive_map_from_de_sitter_map} applies to $N_t$, we need to ascertain that the rank drops transversally without signature change. By explicit calculation this holds for $N_t^{\model}$. Since $N_t$ is even in $x$, we can expand it in the coordinates of the model solution in the following form
  \[
  N_t(x,y) = N_t^0(y) + x^2 N_t^2(x,y).
  \]
  Likewise,
  \[
  N_t^{\model}(x,y) = N_t^{\model,0}(y) + x^2 N_t^{\model,2}(x,y).
  \]
  Moreover, since $N_t(0,y) = N_t^{\model}(0,y)$ for every $y$, we find that
  \[
  N_t(x,y) - N_t^{\model}(x,y) = x^2 \left( N_t^2(x,y) - N_t^{\model,2}(x,y) \right).
  \]
  This means that $N_t$ and $N_t^{\model}$ agree to first order along $x=0$, and therefore the rank of $N_t$ drops transversally without signature change exactly as in $N_t^{\model}$. Therefore Theorem \ref{thm:transgressive_map_from_de_sitter_map} applies and $f_t$ is indeed a smooth transgressive map.
\end{proof}

\section{Construction of $\tau$-real sections}

In this final section, we construct $\tau$-real (i.e. invariant under the real structure) holomorphic sections of the Deligne--Hitchin moduli space with arbitrarily high energy which are not twistor lines.

By
\cite{BeHR}, the energy $\mathcal E$ of a section $s$ of the Deligne--Hitchin moduli space of $M$ can be computed as follows: take a local lift 
$\hat s(\lambda)=(\lambda,\bar\partial+\lambda\Psi_1+\dots,\Phi_{-1}+\lambda\partial+\dots)$ with stable Higgs pair $(\bar\partial,\Phi_{-1})$ at
$\lambda=0.$ Then
\[\mathcal E(s)=2i\,\int_M\tr(\Phi_{-1}\wedge \Psi_1).\] 

Since the present paper focuses on a different aspect compared to \cite{BeHR}, we adopt a slightly different normalization of the energy $\mathcal E$.
In particular, for  $s$ being a twistor line corresponding to an equivariant harmonic map $f$ 
with Dirichlet energy $E(f)$ into
$\mathbb H^3$,  it holds
\[\mathcal E(s)=E(f).\]

Assume from now on that $\det\Phi=-\omega^2$ is the square of a holomorphic $1$-form.
Let $\nabla^\lambda=\lambda^{-1}\Phi+\nabla+\lambda\Phi^*$ be the associated family corresponding to $f$, and
\[\hat\nabla^\lambda=\nabla^\lambda \cdot g(\lambda)=\lambda^{-1}\hat\Phi+\hat\nabla+\lambda\hat\Phi^\dagger\]
as in
\eqref{eq:twistingnabla} for $g(\lambda)=\begin{psmallmatrix}\lambda&0\\0&1\end{psmallmatrix}$ with respect to $L\oplus L^\perp$ for
the $\omega$-eigenline bundle $L$ of $\Phi.$ As we have seen in 
Theorem \ref{thm:SU2SU11}, $\hat\nabla^\lambda$ is the associated family of flat connections for the oblique hyperbolic Gau\ss\@ map $N$ of $f.$
Using the notations of \eqref{eq:SDSsplit} and \eqref{eq:SDSsplittwist},
we obtain from the flatness of $\nabla^\lambda$ and $\hat\nabla^\lambda$
\begin{equation}\label{eq:curvelL}
F^{\nabla^L}+\gamma^*\wedge \gamma+\alpha\wedge\alpha^*=0.
\end{equation}
Let $\deg(L)$ be the degree of $L$, i.e., $\deg(L)=\tfrac{i}{2\pi}\int_MF^{\nabla^L}.$
Thus, the Dirichlet energies of $f$ and $N$ (see \eqref{eq:defEN})
are related by
\begin{equation}\label{eq:energiesfN}
\begin{split}
E(f)&=-E(N)+2i \int_M(\alpha\wedge\alpha^*+\gamma^*\wedge\gamma)
=-E(N)-2\deg(L).
\end{split}
\end{equation}
This relation extends to the setting of (equivariant) transgressive harmonic maps. While the energy of the harmonic map restricted to the finite part (i.e.\ the preimage of the two
copies of hyperbolic space inside the conformal $3$-sphere) is clearly infinite if the singular set is non-empty, the renormalized 
energy given by $\mathcal E$ remains finite.

\begin{prop}\label{prop:energies}
Let $f$ be an equivariant transgressive harmonic map from the compact Riemann surface $M$ to hyperbolic $3$-space. Let 
$s$ be its associated section of the Deligne--Hitchin moduli space. Assume that $s$ is stable and that the determinant of the Higgs field
at $\lambda=0$ is the square of a holomorphic $1$-form on $M$. Let $L$ be the eigenline bundle of the Higgs field with respect to $\omega,$
 and 
$N$ be its oblique hyperbolic Gau\ss\@ map. Then
\[\mathcal E(s)=-E(N)-2\deg(L).\]
\end{prop}

\begin{proof}
Consider a lift $D^\lambda=\lambda^{-1}\Phi+D+\lambda\Psi+\lambda^2\Psi_2+\dots$ of $s$, where $(\bar\partial^D,\Phi)$ is stable and $\det(\Phi)=-\omega^2.$ By \cite[Proposition 2.1]{BeHR}, $\mathcal E(s)=2i\int_M\tr(\Phi\wedge\Psi)$ is independent of the choice of the lift. 

Let $U=f^{-1}(\mathbb H^3_+\cup\mathbb H^3_-)\subset M$ be the open and dense subset where $f$ does not intersect the boundary at infinity $ \mathbb S^2_{\operatorname{eq}}$. On $U$, there exists a holomorphic family of $\mathrm{SL}(2,\C)$ gauge transformations $g(\lambda)=g_0+g_1\lambda+\dots$ such that $D^\lambda \cdot g(\lambda)$ is the associated family of flat connections for the equivariant harmonic map $f_{\mid U}$ to hyperbolic $3$-space.

Consider a complementary line bundle $\tilde L$ of $L$, and define $d(\lambda):=\begin{psmallmatrix}\lambda&0\\0&1\end{psmallmatrix}$ with respect to $L\oplus \tilde L$. Using the construction in
Theorem \ref{thm:twisor-oblique}, there exists a second family of $\mathrm{SL}(2,\C)$ gauge transformations $h(\lambda)=h_0+h_1\lambda+\dots$ such that
$D^\lambda.(d(\lambda)h(\lambda))$ is the family of flat connections associated to the equivariant oblique Gau\ss\@ map $N$ of $f_{\mid U}.$ 

Since $f$ is transgressive and $\det(\Phi)=-\omega^2$ is square, there exists an (equivariant) oblique harmonic Gau\ss\@ map 
$\tilde N\colon \widetilde M\to \dS_3$ by Theorem \ref{thm:transobl}. By uniqueness  $\tilde N_{\mid U}=N.$ Thus, using Theorem \ref{thm:hards3asso}, the associated family of flat connections $D^\lambda.(d(\lambda)h(\lambda))$ extends smoothly
through $M\setminus U.$ The remainder of the proof proceeds exactly as in the derivation of \eqref{eq:energiesfN}.
\end{proof}

We now state and prove our main geometric existence theorem.
\begin{thm}\label{thm:exinenergy}
For every $g>1$, there exists a Riemann surface $M$ of genus $g$ such that its $\mathrm{SL}(2,\C)$ Deligne--Hitchin moduli space
admits  $\tau$-real negative sections $s$ of arbitrarily large energy which are not twistor lines.
\end{thm}

\begin{proof} 
By Remark \ref{rem:many-examples}, there exists for each genus $g$ a Riemann surface of the given genus together with an anti-holomorphic involution $\sigma$ fulfilling the conditions in Theorem \ref{thm:existence_smooth_transgressive_maps}.
Thus, for all sufficiently large $t$, we obtain smooth equivariant transgressive harmonic maps $f_t : \wt{M} \to S^3$  with Hopf differential $-t^2 q.$ 

We first work on the Hitchin covering $\hat M\to M$ branched over the simple zeros of $q.$ The (equivariant) transgressive harmonic map $f_t$ admits an (equivariant) oblique harmonic Gau\ss\@ map
 $N_t : \wt{\hat M} \to \dS_3$. Then $N_t$ induces
a family of flat connections $\hat\nabla^\lambda=\lambda^{-1}\hat\Phi +\hat\nabla+\lambda\hat\Psi$ on $\hat M.$ 
As a consequence of Theorem \ref{thm:SU2SU11} and Theorem \ref{thm:SU11SU2}, the eigenline bundle $\hat L$ of the Higgs field $\hat\Phi$ is null exactly where the transgressive harmonic map $f_t$ intersects the boundary at infinity. 
As before, we may choose a complementary (non-orthogonal) line bundle $\widetilde L$ of $\hat L$ and apply the construction from the proof of Theorem \ref{thm:twisor-oblique} with respect to the non-orthogonal splitting $\hat L\oplus\widetilde L\to\hat M$.
This yields a new family of flat connections $\hat D^\lambda=\lambda^{-1}\Phi+\nabla+\lambda\Psi_1+\lambda^2\dots$, which is not in self-duality form. 
From the construction, and by the proof of Theorem \ref{thm:existence_smooth_transgressive_maps}, the Higgs field $(\bar\partial^\nabla,\Phi)$ at $\lambda=0$ is gauge equivalent to the pullback of the Higgs bundle $(E,\bar\partial^E,t\Phi)$ of the limiting configuration.
In particular, it is stable, for otherwise the eigenline bundle $\hat L$ would have degree zero.

We claim that (upon choosing $\widetilde L$ appropriately)  $\lambda\mapsto \hat D^\lambda$ is the pullback of a family of flat connections $\lambda\mapsto D^\lambda$ defined on $M$ via the Hitchin covering. 
Away from small neighborhoods of the preimages of the fixed-point set of $\sigma$, this can be achieved by taking $\widetilde L$ to be the orthogonal complement of $\hat L$.
For each connected component of the fixed-point set of $\sigma$ (respectively its small neighborhood), we choose one of the two components of the preimages. We then select a smooth interpolating complementary bundle on these components and transport them via the Hitchin involution $\pm\sqrt q\mapsto\mp\sqrt{q}$ to the remaining components. In this way, $\hat D^\lambda$ becomes invariant under the Hitchin involution and hence is
the pullback of a family of flat connections $\lambda\mapsto D^\lambda$ on $M$.

Since its Higgs field is stable, $\lambda\mapsto D^\lambda$ induces a section over $\C$ of the Deligne--Hitchin moduli space of $M$. 
Moreover, it is $\tau$-real and therefore extends to a global section $ s=s_t$ of $\mathcal M_{DH}\to\C P^1.$
Because $\hat\nabla^\lambda$ is the associated family of an equivariant harmonic map to de Sitter $3$-space on the Hitchin covering, we deduce -- by the same reasoning as in Theorem \ref{thm:twisor-oblique} -- that the pullback of $s$ (and hence $s$ itself) is negative.
Since $f_t$ is singular along the (non-empty) fixed-point set of $\sigma$, $s_t$ cannot be a twistor line. 

It remains to show that the energy $\mathcal E(s_t)$ tends to infinity for $t\to\infty.$ By Proposition \ref{prop:energies}, this is equivalent to $E(N_t)\to-\infty$ for $t\to\infty$ 
on the Hitchin covering.
Note that the energy density for harmonic maps to de Sitter $3$-space is (in general) not non-positive.
We analyze the energy densities for $t\to\infty$ on the three different types of domains $U_j,V_k,W_\ell$ in Definition \ref{def:FLC}.

We first consider the model solutions as $t\to\infty.$ On the cylinders $U_j$, the energy density of the oblique Gau{\ss} map becomes
$-2i\, t^2\, dz\wedge d\bar z.$ On the finitely many closures of the sets $W_\ell$, the Dirichlet energy  of the oblique Gau{\ss} map of the fiducial solutions 
is bounded from above by \eqref{eq:energiesfN}. Finally, the Dirichlet energy of the oblique Gau{\ss} map of the transgressive model solution
on (the closures of) the cylinders $V_k$ tends to $-\infty$ for $t\to \infty$, as can be directly deduced by 
rescaling \eqref{exa:assof11} (for $t=1$).

Since the smooth equivariant harmonic maps $N_t$ become arbitrarily close in  $W^{1,2}$ to the respective  model solutions on the three different regimes, the claim $E(N_t)\to-\infty$ for $t\to\infty$ follows as stated.
\end{proof}

\begin{rem}
Previously constructed non-twistor $\tau$-real negative holomorphic sections
had energies bounded from above. In contrast, the sections obtained here exhibit unbounded  energy.
Furthermore, for a nilpotent Higgs field, the energy is essentially the negative of that of the associated equivariant 
Willmore surface in the conformal $3$-sphere. One therefore expects the existence of solutions with arbitrarily high Willmore energy. 
If this expectation is confirmed, the energy of the space of non-twistor $\tau$-real negative holomorphic sections would be unbounded both from above (by Theorem \ref{thm:exinenergy}) and from below.
\end{rem}
\appendix

\section{Perturbed Bessel-type equations}\label{app:A}
Let $I_\nu$ be the modified Bessel function solving
\[
-\left(x \partial_x\right)^2 I_\nu + \left(x^2 + \nu^2\right) I_\nu = 0
\]
with asymptotics $I_\nu(x) \sim \frac 1 {\Gamma(\nu+1)} \left( \frac x 2 \right)^\nu$ at $x=0$ and $I_\nu \sim \frac 1 {\sqrt{2\pi x}} e^x$ at $x=\infty$. Let $K_\nu$ be the modified Bessel function satisfying the same equation with asymptotics $K_\nu \sim \sqrt{\frac {\pi} {2 x}} e^{-x}$ at $x=\infty$.

\begin{lemma}
There exists a $\nu_0 > 0$ such that for all $\nu > \nu_0$
\begin{enumerate}
  \item $I_\nu$ is monotonically increasing and so is $e^{-x/2} I_\nu(x)$,
  \item $K_\nu$ is monotonically decreasing and so is $e^x K_\nu(x)$,
  \item $I_\nu(x) K_\nu(x) \leq \frac {1+\epsilon(\nu)} {2 x}$
\end{enumerate}
for $x \in (0, \infty)$, where $\epsilon(\nu) > 0$ satisfies $\lim_{\nu \to \infty} \epsilon(\nu) = 0$.
\end{lemma}
\begin{proof}
It is classical that $I_\nu$ (resp.\@ $K_\nu$) is monotonically increasing (resp.\@ decreasing) for every $\nu > 0$. For large $\nu$ and $z \in (0,\infty)$ the Debye expansions (see \cite{Olver}, for example) say
\[
I_\nu(\nu z) \sim \frac{e^{\nu \eta}}{(2\pi \nu)^{1/2} (1+z^2)^{1/4}} \sum_{k=0}^\infty \frac{U_k(p)}{\nu^k},
\]
\[
K_\nu(\nu z) \sim \left(\frac{\pi}{2\nu}\right)^{1/2} \frac{e^{-\nu \eta}}{(1+z^2)^{1/4}} \sum_{k=0}^\infty (-1)^k \frac{U_k(p)}{\nu^k}
\]
where
\[
\eta = \left(1 + z^2\right)^{1/2} + \log \left( \frac z {1 + (1+z^2)^{1/2}} \right)
\]
and $p = (1+z^2)^{-1/2}$. The $U_k(p)$ are polynomials of degree $3k$ in $p$ and $U_0(p) = 1$ and the other $U_k$ are given recursively by
\[
U_{k+1}(p) = \frac 1 2 p^2 (1 - p^2) U_k'(p) + \frac 1 8 \int_0^p (1-5t^2) U_k(t) \, dt.
\]
The error terms for these expansions are uniform in $z$ (see \cite{Olver}, Section 10.7), i.e.\@ we can write
\[
I_\nu(\nu z) = \frac{e^{\nu \eta}}{(2\pi \nu)^{1/2} (1+z^2)^{1/4}} \left( \sum_{k=0}^n \frac{U_k(p)}{\nu^k} + O\left( \frac 1 {\nu^{n+1}}\right) \right)
\]
and likewise for $K_\nu$.

Note that the first claim is equivalent to the statement that $f_\nu(z) = e^{-\nu z/2} I_\nu(\nu z)$ is monotonically increasing in $z$.

Note that $f_\nu$ is monotonically increasing if and only if
\[
\log(f_\nu(z)) = -\frac 1 2 \nu z + \log \left(  \frac{e^{\nu \eta}}{(2\pi \nu)^{1/2} (1+z^2)^{1/4}} \right) + \log \left(  \sum_{k=0}^\infty \frac{U_k(p)}{\nu^k} \right)
\]
is monotonically increasing.

Let
\[
S_\nu(z) = \sum_{k=0}^\infty \frac{U_k(p)}{\nu^k}.
\]
By explicit computation it can be shown that
\[
\frac d{dz} \log \left(  \frac{e^{\nu \eta}}{(2\pi \nu)^{1/2} (1+z^2)^{1/4}} \right) \geq \nu - \frac 1 4 \geq \frac 3 4 \nu
\]
for every $z \in (0,\infty)$ and every $\nu > 1$. Therefore it suffices to show that
\[
\frac {d}{dz} \log \left( S_\nu(z)  \right) \geq - \frac 1 4 \nu.
\]
By the properties of the expansion, for any fixed $\nu_0$ the sum $S_{\nu_0}(z)$ converges and is uniformly bounded with respect to $z$ on $(0,\infty)$.

Therefore, for any such $\nu_0$ there exists a $C > 0$, such that
\[
\left|S_\nu(z) - 1 \right| \leq \frac C \nu,
\]
for all $\nu > \nu_0$ and $z \in (0, \infty)$.

Similar reasoning shows that $S_\nu'(z) = O(1/\nu)$. This implies
\[
\frac {d}{dz} \log \left( S_\nu(z)  \right) = \frac{O(1/\nu)}{1 + O(1/\nu)} = O\left(\frac 1 \nu\right)
\]
and so for all sufficiently large $\nu$ the inequality $\frac {d}{dz} \log \left( S_\nu(z)  \right) \geq - \frac 1 4 \nu$ holds as desired.

The proof that $e^x K_\nu(x)$ is monotonically decreasing is very similar.

The inequality $I_\nu(x) K_\nu(x) \leq \frac {1+\epsilon(\nu)} {2 x}$ is equivalent to $I_\nu(\nu z) K_\nu(\nu z) \leq \frac {1+\epsilon(\nu)} { 2 \nu z}$. Writing
\[
A_\nu(z) = \sum_{k=0}^\infty (-1)^k \frac{U_k(p)}{\nu^k}
\]
we see that
\[
I_\nu(z) K_\nu(z) \sim \frac 1 {2 \nu} \frac 1 {(1+z^2)^{1/2}} S_\nu(z) A_\nu(z).
\]
The same reasoning as for $S_\nu$ shows that $A_\nu = 1 + O\left(\frac 1 \nu \right)$.

Since $(1+z^2)^{-1/2} \leq \frac 1 z$, this shows the claim.
\end{proof}

\begin{prop}
\label{prop:perturbed_bessel}
  Suppose that $h : \halfopen{0}{\infty} \to \C$ is $C^\infty$ and satisfies
  \[
  |h(x)| \leq A e^{-\alpha x}
  \]
  for $A, \alpha > 0$. Then there exist constants $C, x_0, \nu_0 > 0$, such that for any $\nu > \nu_0$ there is a unique bounded solution of the equation
  \[
  -(x\partial_x)^2 u + (x^2 + \nu^2)u = h u.
  \]
  This solution satisfies
  \[
  |u(x)| \leq C (1 + \nu^2) |u(x_0)| e^{-\beta x}
  \]
  and
  \[
  |\partial_x u(x)| \leq C (1 + \nu^2) |\partial_x u(x_0)| e^{-\beta x},
  \]
  where $\beta = \frac 1 4 \min \left\{ 1, \alpha \right\}$. The constants $C, x_0$ depend on $A$, $\alpha$ and $\nu_0$, but not on $\nu$ or $u$.
\end{prop}

\begin{proof}
  Choose $\nu_0$ as in the previous lemma, so that for any $\nu > \nu_0$ the properties stated in that lemma hold.
  
  For our purposes it is helpful to define $K_\nu^*(x) = \frac{K_\nu(x)}{e^{x_0}K_\nu(x_0)} $ for some given $x_0 > 0$. This function satisfies
  \[
  K_\nu^*(x) \leq e^{-x}
  \]
  for $x \geq x_0$, using the monotonicity of $e^x K_\nu(x)$.
  
  By a variation of parameters ansatz one can show that a solution $u$ as above satisfies the integral equation
  \[
  u(x) = \lambda_1 I_\nu(x) + \lambda_2 K_\nu^*(x) - \int_x^\infty K_\nu(s) h(s) u(s) \frac{ds} s I_\nu(x) + \int_{x_0}^x I_\nu(s) h(s) u(s) \frac{ds} s K_\nu(x)
  \]
  for some $\lambda_1, \lambda_2 \in \C$ and $x_0 > 0$. Conversely, if $u$ is $C^2$ and $u$ satisfies the integral equation, then $u$ satisfies the ODE.

  As we will establish later, if $u$ is bounded then the last two terms also stay bounded as $x \to \infty$. On the other hand the Bessel function $I_\nu$ diverges for any $\nu$ as $x \to \infty$. Therefore for any bounded solution $\lambda_1 = 0$.

  Let
  \[
  f_u(x) = - \int_x^\infty K_\nu(s) h(s) u(s) \frac{ds} s I_\nu(x) + \int_{x_0}^x I_\nu(s) h(s) u(s) \frac{ds} s K_\nu(x).
  \]
  Then a bounded $u$ satisfies $u = f_u + \lambda K_\nu^*$. Moreover, $u(x_0) = f_u(x_0) + \lambda K_\nu^*(x_0)$. Since $K_\nu^*(x_0) = 1$ by definition, this implies that $u(x_0) = u_0$ is equivalent to
  \[
  \lambda = u_0 - f_u(x_0).
  \]
  Therefore the initial value problem is equivalent to the pair of equations
  \[
  u = f_u + \lambda K_\nu^*, \qquad \lambda = - f_u(x_0) + u_0.
  \]
  Now define the operator
  \[
  A(u, \lambda) = \left( f_u + \lambda K_\nu^*, -f_u(x_0)\right).
  \]
  Then the equations above can be written as
  \[
  (u, \lambda) = A \left(u, \lambda \right) + \left(0, u_0 \right).
  \]
  Writing this in the form $\left(\id - A\right) (u,\lambda) = \left(0, u_0 \right)$ we see that we can find our solution by inverting $\id - A$.

  To this end, let us introduce the weighted function spaces $L^2_\kappa = L^2_\kappa ( \halfopen{x_0}{\infty}, \C)$ via the norm
  \[
  \|u\|_{L^2_\kappa} = \left( \int_{x_0}^\infty \left| e^{\kappa x} u(x)\right|^2 dx \right)^{1/2}
  \]
  for $\kappa >0$. We claim that for $\kappa = \frac 1 2 \min \{\alpha, 1 \} = 2 \beta$ the operator $A$ is well-defined as an operator $A : L^2_\kappa \to L^2_\kappa$. Moreover, we claim that for any $\epsilon > 0$ we can arrange that $\|A\| < \epsilon$ by choosing $x_0$ appropriately.

  Define
  \[
  F(x) = - \int_x^\infty K_\nu(s) h(s) u(s) \frac{ds} s I_\nu(x),
  \]
  \[
  G(x) = \int_{x_0}^x I_\nu(s) h(s) u(s) \frac{ds} s K_\nu(x).
  \]
  To get an estimate for the operator norm of $A$ we need to estimate $\|K_\nu^*\|_{L^2_\kappa}$, $\|F\|_{L^2_\kappa}$ and $\|G\|_{L^2_\kappa}$ in terms of $x_0$.

  We have for $\kappa<1$
  \[
  \|K_\nu^*\|_{L^2_\kappa} \leq \|e^{-x}\|_{L^2_\kappa} = \left( \int_{x_0}^\infty e^{2 (\kappa - 1) x} dx \right)^{1/2} = \frac 1 {\sqrt{2(1-\kappa)}} e^{(\kappa-1)x_0}.
  \]
  For $x_0$ sufficiently large, this becomes as small as desired. Next we estimate $\|F\|_{L^2_\kappa}$. To this end we first give an estimate of $|F(x)|$:
  \begin{align*}
    |F(x)| & \leq \int_x^\infty K_\nu(s) |h(s)| |u(s)| \frac {ds} s I_\nu(x) \\
    & = \int_x^\infty K_\nu(s) e^{-\kappa s} |h(s)| e^{\kappa s} |u(s)| \frac {ds} s I_\nu(x) \\
    & \leq K_\nu(x) e^{-\kappa x} \int_x^\infty \frac {|h(s)|} s e^{\kappa s} |u(s)| ds I_\nu(x) \\
    & \leq e^{-\kappa x} \left( \int_x^\infty \frac{|h(s)|^2}{s^2} ds \right)^{1/2} \|u\|_{L^2_\kappa} I_\nu(x) K_\nu(x) \\
    & \leq \left( \int_{x_0}^\infty \frac{|h(s)|^2}{s^2} ds \right)^{1/2} \|u\|_{L^2_\kappa} e^{-\kappa x}\frac {1+c_{\nu_0}} {2 x},
  \end{align*}
  where we used the monotonicity of $K_\nu$ and $e^{-\kappa x}$ and the inequality $I_\nu(x) K_\nu(x) \leq \frac {1 + c_{\nu_0}} {2x}$, where $c_{\nu_0} > 0$ is a constant depending on $\nu_0$. Since $h$ decays exponentially, for any $\epsilon > 0$ we may find a $x_0$ (depending only on $h$ and $\nu_0$), such that $(1+ c_{\nu_0}) \left( \int_{x_0}^\infty \frac{|h(s)|^2}{s^2} ds \right)^{1/2} \leq \epsilon$. We then obtain
  \[
  |F(x)| \leq \epsilon \|u\|_{L^2_\kappa} e^{-\kappa x} \frac 1 {2x}.
  \]
  This implies $\|F\|_{L^2_\kappa} \leq \frac \epsilon {4 x_0} \|u\|_{L^2_\kappa}$. In particular, for $x_0 \geq 1$ we get $\|F\|_{L^2_\kappa} \leq \frac \epsilon 4 \|u\|_{L^2_\kappa}$.

  To estimate $\|G\|_{L^2_\kappa}$ we again first estimate $|G(x)|$:
  \begin{align*}
    |G(x)| & \leq \int_{x_0}^x I_\nu(s) |h(s)| |u(s)| \frac{ds} s K_\nu(x) \\
    & \leq \int_{x_0}^x I_\nu(s) e^{-\kappa s} |h(s)| e^{\kappa s} |u(s)| \frac{ds} s K_\nu(x) \\
    & \leq I_\nu(x) e^{-\kappa x} \int_{x_0}^x |h(s)| e^{\kappa s} |u(s)| \frac{ds} s K_\nu(x),
  \end{align*}
  where we used that $I_\nu(x) e^{-\kappa x}$ is monotonically increasing, since $\kappa \leq \frac 1 2$. We then continue as before to conclude
  \begin{align*}
    |G(x)| & \leq \|u\|_{L^2_\kappa} \left(\int_{x_0}^\infty \frac {|h(s)|}{s^2} ds \right)^{1/2} e^{-\kappa x} \frac {1+c_{\nu_0}} {2 x} \\
    & \leq \epsilon \|u\|_{L^2_\kappa} e^{-\kappa x} \frac {1} {2x}
  \end{align*}
  and from this we again obtain $\|G\|_{L^2_\kappa} \leq \frac \epsilon 4 \|u\|_{L^2_\kappa}$.  

  We have therefore shown that for any $\epsilon > 0$ we can find $x_0$, such that
  \[
  A : L^2_\kappa(\halfopen{x_0}{\infty}, \C) \times \C \to L^2_\kappa(\halfopen{x_0}{\infty}, \C) \times \C
  \]
  satisfies $\|A u \|_{L^2_\kappa} \leq \epsilon \|u\|_{L^2_\kappa}$.

  For $\epsilon < 1$ the operator $\id - A$ is invertible by the Neumann series, i.e.\@
  \[
  (\id - A)^{-1} = \sum_{k=0}^\infty A^k.
  \]
  In particular one gets
  \[
  \left\| (\id - A)^{-1} \right\| \leq \frac {1}{1 - \|A\|}.
  \]
  Now choose $x_0$, such that $\epsilon \leq \frac 1 2$. Then $ \left\| (\id - A)^{-1} \right\| \leq 2$. In particular, we obtain for $(u, \lambda) = \left(\id - A \right)^{-1} \left(0, u_0 \right)$ the bound
  \[
  \left(\|u\|_{L^2_\kappa}^2 + |\lambda|^2\right)^{1/2} \leq \left\|\left(\id - A \right)^{-1} \left(0, u_0 \right)\right\| \leq 2 \left| u_0 \right|.
  \]
  We now improve this bound from an $L^2_\kappa$ bound to a pointwise bound. To do this, first observe that $u$ satisfies the integral equation and $h$ is $C^\infty$, which implies that $u$ is $C^\infty$ and satisfies the differential equation.

  Next, we observe that the equation can be rewritten as
  \[
  \partial_x^2 u = u + \frac{\nu^2}{x^2} u - \frac h{x^2} u - \frac 1 x \partial_x u.
  \]
  If we choose $x_0$, such that $x_0 \geq 1$ and $|h(x)| \leq 1$ for $x \geq x_0$, we may then estimate
  \[
  \|\partial_x^2 u\|_{L^2_\kappa} \leq (2 + \nu^2) \|u\|_{L^2_\kappa} + \|\partial_x u\|_{L^2_\kappa}.
  \]
  Using the interpolation inequality
  \[
  \|\partial_x u\|_{L^2_\kappa} \leq \epsilon \|\partial_x^2 u\|_{L^2_\kappa} + C_\epsilon \|u\|_{L^2_\kappa}
  \]
  then allows us to absorb the $\partial_x u$ term on the right hand side and we obtain that there is some $C > 0$, such that
  \[
  \|\partial_x^2 u\|_{L^2_\kappa} \leq (C + \nu^2 ) \|u\|_{L^2_\kappa}.
  \]
  Using the equation or the interpolation inequality we get (with a potentially different $C$)
  \[
  \|\partial_x u\|_{L^2_\kappa} \leq C (1 + \nu^2) \|u\|_{L^2_\kappa}
  \]
  Using the estimate for $\|u\|_{L^2_\kappa}$ we conclude (again with a new $C$) $\|\partial_x u\|_{L^2_\kappa} \leq C ( 1+ \nu^2) |u_0|$.

  Now observe that
  \[
  (e^{\frac \kappa 2 x} u)'(x) = \frac \kappa 2 e^{\frac \kappa 2 x} u(x) + e^{\frac \kappa 2 x} u'(x).
  \]
  We then estimate
  \begin{align*}
  \int_{x_0}^\infty |(e^{\frac \kappa 2 x} u)'(x)| dx & \leq \frac \kappa 2 \int_{x_0}^\infty \left| e^{\frac \kappa 2 x} u(x) \right| dx + \int_{x_0}^\infty \left| e^{\frac \kappa 2 x} u'(x) \right| dx \\
  & = \frac \kappa 2 \int_{x_0}^\infty e^{\kappa x} |u(x)| e^{-\frac \kappa 2 x} dx + \int_{x_0}^\infty e^{\kappa x} |u'(x)| e^{-\frac \kappa 2 x} dx \\
  & \leq \frac \kappa 2 \|u\|_{L^2_\kappa} \left( \int_{x_0}^\infty e^{-\kappa x} \right)^{1/2} + \|\partial_x u\|_{L^2_\kappa} \left( \int_{x_0}^\infty e^{-\kappa x} \right)^{1/2} \\
  & =  e^{-\kappa x_0} \left( \frac 1 2 \|u\|_{L^2_\kappa} + \frac 1 \kappa \|\partial_x u\|_{L^2_\kappa} \right).
  \end{align*}
  Using the previous estimates we obtain (with some new $C > 0$)
  \[
  \int_{x_0}^\infty |(e^{\frac \kappa 2 x} u)'(x)| dx \leq C(1+ \nu^2) |u_0|.
  \]
  Using the fundamental theorem of calculus we obtain
  \[
  \left| e^{\frac \kappa 2  x} u(x) - e^{\frac \kappa 2 x_0} u_0 \right| \leq C (1+\nu^2) |u_0|,
  \]
  which we can rearrange to
  \[
  |u(x)| \leq C ( 1+ \nu^2) |u_0| e^{-\frac \kappa 2 x},
  \]
  which is the claim.

  To estimate $|\partial_x u(x)|$ we may perform exactly the same analysis, now applied to $\|\partial_x^2 u\|_{L^2_\kappa}$. We then obtain
\[
|\partial_x u(x)| \leq C (1 + \nu^2) |\partial_x u(x_0)| e^{-\frac \kappa 2  x},
\]
with perhaps a different $C > 0$.
\end{proof}

For the equation $-(x \partial_x)^2 u + \nu^2 u = hu$ the same estimates can be proven. Indeed, the proof is significantly easier, because the solutions of the homogeneous equation $-(x\partial_x)^2 u + \nu^2 u = 0$ are just linear combinations of $x^{\nu}$ and $x^{-\nu}$. We omit the details of the proof.
\begin{prop}
\label{prop:perturbed_euler}
  Suppose that $h : \halfopen{0}{\infty} \to \C$ is $C^\infty$ and satisfies
  \[
  |h(x)| \leq A e^{-\alpha x}
  \]
  for $A, \alpha > 0$.
  
  Then there exist constants $C, x_0, \nu_0 > 0$, such that for any $\nu > \nu_0$ there is a unique bounded solution of the equation
  \[
  -(x\partial_x)^2 u + \nu^2 u = h u.
  \]
  This solution satisfies
  \[
  |u(x)| \leq C (1 + \nu^2) |u(x_0)| e^{-\kappa x}
  \]
  and
  \[
  |\partial_x u(x)| \leq C (1 + \nu^2) |\partial_x u(x_0)| e^{-\kappa x},
  \]
  where $\kappa = \frac 1 2 \min \left\{ 1, \alpha \right\}$.

  The constants $C, x_0$ depend on $A$ and $\alpha$, but not on $\nu$ or $u$.
\end{prop}


\begin{thebibliography}{ZZZZ}


\bibitem{AlMa}
S.~Alexakis, R.~Mazzeo, {\em 
Renormalized area and properly embedded minimal surfaces in hyperbolic $3$-manifolds},
Commun. Math. Phys. 297 (2010), no.~3, 621--651.

\bibitem{Anderson}
M.~Anderson, {\em Complete minimal varieties in hyperbolic space}, Invent. Math. 69 (1982), no.~3, 477--494.

\bibitem{BaBo}
M.~Babich, A.~I.~Bobenko, {\em Willmore tori with umbilic lines and minimal surfaces in hyperbolic
space}, Duke Math. J. 72 (1993), no.~1, 151--185.


\bibitem{Bar}
D.~Baraglia, {\em Cyclic Higgs bundles and the affine Toda equations}, Geom. Dedicata 174 (2015), 25--42.

\bibitem{BGIM}
E.~Bahuaud, C.~Guenther, J.~Isenberg, R.~Mazzeo, {\em Well-posedness of nonlinear flows on manifolds of bounded geometry}, Ann. Global Anal. Geom. 65 (2024), no.~4, 39~pp.

\bibitem{BeHR}
F.~Beck, S.~Heller, M.~R\"oser, {\em Energy of sections of the Deligne--Hitchin twistor space}, Math. Ann. 380 (2021), no.~3--4, 1169--1214.

\bibitem{BHR}
I.~Biswas, S.~Heller, M.~R\"oser, {\em Real sections of the Deligne--Hitchin moduli space}, Comm. Math. Phys. 366 (2019), no.~3, 1099--1133.



\bibitem{Blaschke}
W.~Blaschke, {\em Vorlesungen \"uber Differentialgeometrie und geometrische Grundlagen von Einsteins Relativit\"atstheorie. III: Differentialgeometrie der Kreise und Kugeln}, bearbeitet von G.~Thomsen, Grundlehren der mathematischen Wissenschaften in Einzeldarstellungen, 1929.

\bibitem{BHS}
A.~I.~Bobenko, S.~Heller, N.~Schmitt, {\em Minimal $n$-noids in hyperbolic and anti-de Sitter $3$-space}, Proc. R. Soc. A, 2019, DOI: 10.1098/rspa.2019.0173.

\bibitem{BPP-Schwarzian}
F.~Burstall, F.~Pedit, U.~Pinkall, {\em Schwarzian derivatives and flows of surfaces}, in: Differential geometry and integrable systems (Tokyo, 2000), Contemp. Math., vol.~308, Amer. Math. Soc., Providence, RI, 2002.



\bibitem{BGPG}
S.~B.~Bradlow, O.~Garc\'ia-Prada, P.~B.~Gothen, {\em Surface group representations and $\mathrm{U}(p,q)$-Higgs bundles}, J. Differential Geom. 64 (2003), no.~1, 111--170.


\bibitem{Donaldson}
S.~K.~Donaldson, {\em Twisted harmonic maps and the self-duality equations}, Proc. Lond. Math. Soc. (3) 55 (1987), 127--131.

\bibitem{Donald}
S.~K.~Donaldson, {\em Riemann Surfaces}, Oxford Graduate Texts in Mathematics, 22, Oxford Univ. Press, 2011.

\bibitem{DIK}
J.~F.~Dorfmeister, J.~Inoguchi, S.-P.~Kobayashi, {\em Constant mean curvature surfaces in hyperbolic $3$-space via loop group}, J. Reine Angew. Math. 686 (2014), 1--36.

\bibitem{DPW}
J.~F.~Dorfmeister, F.~Pedit, H.~Wu, {\em Weierstrass type representation of harmonic maps into symmetric spaces}, Comm. Anal. Geom. 6 (1998), no.~4, 633--668.

\bibitem{Douady-Hubbard}
A.~Douady, J.~H.~Hubbard, {\em On the density of Strebel differentials}, Invent. Math. 30 (1975), 175--179.

\bibitem{FMSW}
L.~Fredrickson, R.~Mazzeo, J.~Swoboda, H.~Weiss, {\em Asymptotic geometry of the moduli space of parabolic $\mathrm{SL}(2,\C)$-Higgs bundles}, J. Lond. Math. Soc. (2) 106 (2022), no.~2, 590--661.

\bibitem{GrHa}
B.~H.~Gross, J.~Harris, {\em Real algebraic curves}, Ann. Sci. \'Ecole Norm. Sup. (4) 14 (1981), no.~2, 157--182.

\bibitem{H}
S.~Heller, {\em A spectral curve approach to Lawson symmetric CMC surfaces of genus $2$}, Math. Ann. 360 (2014), no.~3, 607--652.

\bibitem{HH}
L.~Heller, S.~Heller, {\em Higher solutions of Hitchin's self-duality equations}, J. Integrable Syst. 5 (2020), 1--48.

\bibitem{Hgraft}
S.~Heller, {\em Real projective structures on Riemann surfaces and new hyper-K\"ahler manifolds}, Manuscripta Math. 171 (2023), no.~1--2, 241--262.


\bibitem{HiSD}
N.~J.~Hitchin, {\em The self-duality equations on a Riemann surface}, Proc. Lond. Math. Soc. (3) 55 (1987), no.~1, 59--126.

\bibitem{Hi}
N.~J.~Hitchin, {\em Harmonic maps from a $2$-torus to the $3$-sphere}, J. Differential Geom. 31 (1990), no.~3, 627--710.

\bibitem{La}
H.~B.~Lawson, {\em Complete minimal surfaces in $\mathbb{S}^3$}, Ann. Math. (2) 92 (1970), no.~3, 335--374.

\bibitem{Lee}
J.~M.~Lee, {\em Fredholm Operators and Einstein Metrics on Conformally Compact Manifolds}, Mem. Amer. Math. Soc. 183 (2006), no.~864.

\bibitem{LiTam0}
P.~Li, L.-F.~Tam, {\em The heat equation and harmonic maps of complete manifolds}, Invent. Math. 105 (1991), no.~1, 1--46.

\bibitem{LiTam}
P.~Li, L.-F.~Tam, {\em Uniqueness and regularity of proper harmonic maps}, Ann. Math. (2) 137 (1993), no.~1, 167--201.

\bibitem{MN}
F.~C.~Marques, A.~Neves, {\em Min-max theory and the Willmore conjecture}, Ann. Math. (2) 179 (2014), no.~2, 683--782.

\bibitem{Mazzeo}
R.~Mazzeo, {\em Elliptic theory of differential edge operators. I}, Comm. Partial Differential Equations 16 (1991), no.~10, 1615--1664.

\bibitem{MM}
R.~Mazzeo, R.~Melrose, {\em Meromorphic extension of the resolvent on complete spaces with asymptotically constant negative curvature}, J. Funct. Anal. 75 (1987), no.~2, 260--310.

\bibitem{MSWW}
R.~Mazzeo, J.~Swoboda, H.~Weiss, F.~Witt, {\em Ends of the moduli space of Higgs bundles}, Duke Math. J. 165 (2016), no.~12, 2227--2271.

\bibitem{MW}
R.~Mazzeo, H.~Weiss, {\em Teichm\"uller theory for conic surfaces}, in {\em Geometry, analysis and probability}, 127--164, Progr. Math., 310, Birkh\"auser/Springer, Cham, 2017.


\bibitem{Moc}
T.~Mochizuki, {\em Asymptotic behaviour of certain families of harmonic bundles on Riemann surfaces}, J. Topol. 9 (2016), no.~4, 1021--1073.

\bibitem{Olver}
F.~W.~J.~Olver, {\em Asymptotics and special functions}, A K Peters, Wellesley, MA, 1997.

\bibitem{OSWW}
A.~Ott, J.~Swoboda, R.~Wentworth, M.~Wolf, {\em Higgs bundles, harmonic maps and pleated surfaces}, Geom. Topol. 28 (2024), no.~7, 3135--3220.

\bibitem{PS}
A.~Pressley, G.~Segal, {\em Loop Groups}, Oxford Mathematical Monographs, Oxford Univ. Press, New York, 1986.


\bibitem{Sim2}
C.~T.~Simpson, {\em Harmonic bundles on noncompact curves}, J. Amer. Math. Soc. 3 (1990), no.~3, 713--770.

\bibitem{Sim}
C.~T.~Simpson, {\em Higgs bundles and local systems}, Publ. Math. Inst. Hautes \'Etudes Sci. 75 (1992), 5--95.

\bibitem{Simpson}
C.~T.~Simpson, {\em The Hodge filtration on nonabelian cohomology}, in: Algebraic geometry -- Santa Cruz 1995, Proc. Sympos. Pure Math., 62, Part 2, Amer. Math. Soc., Providence, RI, 1997, 217--281.




\bibitem{Usula}
M.~Usula, {\em Yang--Mills connections on conformally compact manifolds}, Lett. Math. Phys. 111 (2021), no.~2, Paper No.~56, 23~pp.

\end{thebibliography}
\end{document}